\documentclass[11pt]{amsart}
\usepackage{a4wide,enumerate,enumitem,amsmath,color}
\usepackage[utf8]{inputenc}
\usepackage[scale = .85]{geometry}
\usepackage{amssymb}
\usepackage{amsthm}
\usepackage{graphicx}
\usepackage{subcaption}
\usepackage[noadjust]{cite}
\usepackage{bigints}
\usepackage{breqn}
\usepackage{tikz}
\usepackage{mathtools}
\usepackage{hyperref}
\usepackage{cite}
\usepackage{dsfont}
\usepackage {float}

\newtheorem{theorem}{Theorem}
\newtheorem{corollary}{Corollary}
\newtheorem{lemma}{Lemma}
\newtheorem{proposition}{Proposition}
\newtheorem{remark}{Remark}
\newtheorem{definition}{Definition}

\def\R{\mathbb{R}}

\def\<{\langle}
\def\>{\rangle}
\def\a{\alpha}

\def\eps{\varepsilon}

\def\dis{\displaystyle}

\newcommand{\T}{\mathbb{T}^2}
\newcommand{\K}{\mathcal{K}}

\newcommand{\deltat}{\frac{\Delta t}{\Delta x}}

\newcommand{\dd}{\mathrm{d}}
\newcommand{\pa}{\partial}
\newcommand{\per}{\mathrm{per}}

\newcommand{\Z}{\mathbb{Z}}

\newcommand{\I}{\mathcal{I}_N}
\newcommand{\Imm}{\mathcal{I}_{N_m}}
\newcommand{\N}{\mathbb{N}}
\newcommand{\meannp}{\langle \rho^{\pm,n+1,\per}\rangle}
\newcommand{\meann}{\langle \rho^{\pm,n,\per}\rangle}

\begin{document}

\title[Convergence of a scheme for a nonlocal system]{Convergence of a scheme for a two dimensional nonlocal system of transport equations}

\author[D. Al Zareef]{Diana Al Zareef$^{1,2}$}
\author[A. El Hajj]{Ahmad El Hajj$^{1}$}
\author[A. Zurek]{Antoine Zurek$^{1}$}
\address{$^{1}$Universit\'e de Technologie de Compi\`egne, LMAC, 60200 Compi\`egne, France.} 
\address{$^{2}$Faculty of Science I, Mathematics Department, Lebanese University, Hadath, Lebanon.}
\email{diana.al-zareef@utc.fr, elhajjah@utc.fr, antoine.zurek@utc.fr.}

\date{\today}

\begin{abstract}
In this paper, we numerically study a two-dimensional system modeling the dynamics of dislocation densities. This system is hyperbolic, but not strictly hyperbolic, and couples two non-local transport equations. It is characterized by weak regularity in both the velocity and the initial data. We propose a semi-explicit finite difference (IMEX) numerical scheme for the discretization of this system, after regularizing the singular velocity using a Fejér kernel. We show that this scheme preserves, at the discrete level, an entropy estimate on the gradient, which then allows us to establish the convergence of the discrete solution to the continuous solution. To our knowledge, this is the first convergence result obtained for this type of system. We conclude with some numerical illustrations highlighting the performance of the proposed scheme.

\bigskip
		

\medskip
		
\end{abstract}

\maketitle


\tableofcontents


\section{Introduction}
\subsection{Physical motivation}
In this paper, we are interested in the numerical study of a phase-field model introduced by Cannone et al.~\cite{cannone2010global}, which describes the dynamics of dislocation curves in crystalline materials. Dislocations are linear defects in the crystal structure that were first introduced by Taylor, Orowan, and Polanyi in 1934. They observed that these defects play a central role in explaining plastic deformation at the microscopic scale of materials (see~\cite{19, hull2011introduction} for a physical description of dislocations).

In crystalline solids, a dislocation can be characterized by its Burgers vector, a vector that quantifies the magnitude and direction of the lattice distortion caused by the dislocation. In the case of edge dislocations, which are line defects where an extra half-plane of atoms is inserted into the crystal, the Burgers vector lies perpendicular to the dislocation line and indicates the direction in which the dislocation moves under applied stress. This motion is driven by the so-called Peach-Koehler force, which arises from elastic interactions in the material.

In this work, we focus on a two-dimensional model of dislocation dynamics, originally proposed by Groma and Balogh in~\cite{PhysRevB.56.5807, GROMA19993647} and rigorously re-derived in~\cite{cannone2010global}, in which parallel edge dislocations move within a planar cross-section of a three-dimensional crystal. The model describes the evolution of two scalar functions, $\rho^+$ and $\rho^-$, representing 
respectively the plastic distortions associated with the dislocations
moving according to opposite Burgers vectors $\vec{\mathbf{b}} = (1,0)$ and $-\vec{\mathbf{b}}$. 
These functions satisfy a system of nonlocal transport equations:
{\small \begin{equation}\label{original}
    \begin{cases}
 \partial_t\rho^+ (t, x) = - \bigg(R_1^2 R_2^2 (\rho^+(t,.) - \rho^-(t, .))(x) +a(t)\bigg)
 \vec{\mathbf{b}}\cdot \nabla \rho^+(t, x)   & \mbox{for a.e. } (t,x)\in  (0, T) \times \R^2, \\
 \partial_t \rho^-(t, x) =  \;\;\;\bigg(R_1^2 R_2^2 (\rho^+(t,.) - \rho^-(t, .))(x)+a(t)\bigg) 
 \vec{\mathbf{b}}\cdot \nabla \rho^-(t, x) 
 & \mbox{for a.e. } (t,x)\in (0, T) \times \R^2,
\end{cases}
\end{equation}}
with $x:=(x_1,x_2)$ and $T>0$.  The unknowns $\rho^+$ and $\rho^-$ are real-valued functions, denoted for simplicity by $\rho^\pm$,  whose spatial derivatives with respect to $x_1$ ($\partial_{x_1}\rho^\pm$) are positive quantities representing the dislocation densities  of Burger's vector $\pm \vec{\mathbf{b}}$.
Here the function $a(\cdot)$ represents the exterior shear strain field, and we assume:
\begin{enumerate}[label=\textbf{(H{\arabic*})}]
	\item \label{a} Shear strain field: the function $a$ belongs to $C(0,T)$.	\medskip
\end{enumerate}

In what follows, we consider only solutions  $\rho^\pm$ for which 
 $\partial_{t}\rho^\pm$, $\nabla \rho^\pm$ and $(\rho^+ - \rho^-)$ are $\Z^2$ -periodic. In this context, we define the operators $R_1$ and $R_2$ as the periodic Riesz transforms along  $x_1$ and $x_2$, respectively. 
 That is, for any function $f\in L^2(\T)$ with  $\T := (\R/\Z)^2$, the Fourier coefficients 
 of $R_if$ ($i=1,2$) are given by
 
\begin{equation}\label{coeff Fourier Ri}
\left\{\begin{array}{ll}
  \dis  c_{(0,0)}(R_i f)= 0\\\\
 \dis   c_m(R_i f) = \frac{m_i}{|m|} c_m(f) \quad \mbox{for all } m=(m_1,m_2) \in \Z^2 \setminus \{(0,0)\},
    \end{array}\right. 
\end{equation}
where $|m|=\sqrt{m_1^2+m_2^2}$ and  $$c_m(f) = \int_{\T} f(x_1,x_2) e^{-2i \pi (m_1 x_1+m_2x_2)} \, dx_1 dx_2.$$
Based on \eqref{coeff Fourier Ri}, Fourier analysis and the classical theory of singular integral operators (cf. \cite{operateursin, grafakos2014classical}) ensure that the following 
singular kernel $\K(x)=\sum_{m\in \Z^2} \frac{m_1^2m_2^2}{|m|^4} e^{2i\pi m\cdot x}$ (defined in the distributional sense), satisfies

\begin{align}\label{def_conv0}
    R^2_1R^2_2 f &= \K \ast f
     \quad \mbox{for all } f \in L^2(\T), 
\end{align}
where $\ast$ here denotes the periodic convolution defined in the distributional sense on the torus $\T$. Moreover, for all $f \in L^2(\T)$ we have $\K \ast f \in  L^2(\T)$ and 
\begin{equation}\label{coeff Fourier K}
\left\{\begin{array}{ll}
  \dis  c_{(0,0)}(\K \ast f)= 0\\\\
 \dis   c_m(\K \ast f) = c_m(\K) c_m(f)=\frac{m_1^2m_2^2}{|m|^4} c_m(f) \quad \mbox{for all } m=(m_1,m_2) \in \Z^2 \setminus \{(0,0)\}.
    \end{array}\right. 
\end{equation}
We supplement system \eqref{original} with the following initial data:
\begin{align}\label{IC.original}
\rho^{\pm}(0,x) = \rho_0^{\pm}(x) = \rho_0 ^{\pm, \per}(x) + Lx_1, 
\quad \mbox{for a.e. } x\in\R^2,  
\end{align}
where $\rho_0 ^{\pm, \per}$ is a periodic function defined on $\T$ and $L>0$ is a given constant which models the initial total dislocation densities of $\pm$ type on the periodic cell. We note that this selection of initial data allows us to study the behavior inside the material, away from the boundaries.

We now focus on solutions of the same type as the initial data, namely 
\begin{align}\label{decompo.sol.original}
\rho^{\pm}(t,x) = \rho^{\pm,\per}(t,x) + L \, x_1, \quad \mbox{for a.e. } (t,x) \in (0,T)\times\T, 
\end{align}
where $\rho^{\pm, \per}$ are $1$-periodic functions  in $x_1$ and $x_2$. This, together with \eqref{def_conv0}, ultimately leads us to the following system equivalent to~\eqref{original}, set on $(0,T)\times\mathbb{T}^2$:
\begin{equation}\label{original.per}
    \begin{cases}
 \partial_t\rho^{+,\per} (t, x) &= - \bigg( \K \ast (\rho^{+,\per}(t,.) - \rho^{-,\per}(t, .))(x) +a(t)\bigg)\left(\partial_{x_1}\rho^{+,\per}(t, x)+L\right), \\
 \partial_t \rho^{-,\per}(t, x) &=  \;\;\;\bigg(\K \ast (\rho^{+,\per}(t,.) - \rho^{-,\per}(t, .))(x)+a(t)\bigg) \left(\partial_{x_1}\rho^{-,\per}(t, x) + L \right),
\end{cases}
\end{equation}
with, obviously, the initial data
\begin{align}\label{IC.original.per}
\rho^{\pm,\per}(0, x) =  \rho_0 ^{\pm, \per}(x), \quad \quad \mbox{for a.e. } x\in\T.
\end{align}
 
 \begin{remark}\;
 \begin{itemize}
 \item It should be noted that, in~\eqref{coeff Fourier K} 
 the Fourier coefficients of the kernel $\K$ are defined in 
 the distributional sense and are all positive. 
 \item  It should also be noted that the presence of the  nonlocal term $\K \ast (\rho^+ - \rho^-)$  in the velocity arises from the inversion of the elasticity equation in the modeling (see~\cite[Section 2.]{cannone2010global}). This term physically represents the  resolved  shear stress.
 \end{itemize}
 
\end{remark}

\subsection{Brief review of the literature} 
The system~\eqref{original} was initially introduced by Groma and Balogh in~\cite{GROMA19993647} to model the evolution of dislocation densities at the mesoscopic scale. The setting corresponds to a particular geometry in which dislocations are assumed to be parallel lines in space, moving only along two opposite directions, $\pm(1,0)$. This model was later revisited in~\cite{cannone2010global}, where a more rigorous mathematical formulation was proposed.

In~\cite{Meurs2}, the authors showed that, under certain assumptions, system~\eqref{original} can be interpreted as the limit of two-species particle systems, with positive and negative particles interacting via a common logarithmic singularity potential. This analysis highlights the crucial role played by the scaling and regularization of the interaction kernel in the validity, or invalidity, of the Groma-Balogh model. Building on this multi-scale approach, the authors establish in~\cite{Meurs3} a connection between atomistic, stochastic, and continuous descriptions, showing that these nonlocal transport equations governing dislocation density dynamics can be viewed as mean-field limits of particle interaction systems. Similar results on mean-field limits for particle systems with Coulomb interactions have also been obtained in a stochastic framework in~\cite{Meurs1}.

From the perspective of mathematical analysis, system~\eqref{original} belongs to the class of transport equations with irregular velocity fields. For such equations, existence and uniqueness results generally rely on the assumption of bounded divergence of the velocity field (see~\cite{DiPernaLions1989, Ambrosio2004, AmbrosioSerfaty2008}). Although this assumption is not satisfied in the present setting, a first global-in-time existence result for \eqref{original} was established in~\cite{cannone2010global} using an entropy estimate. However, due to the lack of regularity of the velocity field, the question of global uniqueness remains open. Local-in-time existence and uniqueness results have nevertheless been obtained in~\cite{ElHajj2010, LiMiaoXue2014} for regular solutions belonging to certain Sobolev spaces. The well-posedness of this system over long times was also proved in~\cite{WanChen2016GBLongtime}  for small initial data in Sobolev norm. It is also mentioned that, in the one-dimensional case, global existence and uniqueness can be established using a comparison principle (see~\cite{MR2373180, MR2349873, ElHajjOussaily2021, AlZohbiElHajj2024}).

From the perspective of numerical analysis, most existing schemes have been developed for the discretization of hyperbolic systems in conservative form, in order to properly preserve the Rankine-Hugoniot jump conditions. We refer to \cite{LeVeque2002} for a detailed presentation of the main classes of numerical schemes. Nevertheless, convergence results remain rare for hyperbolic systems, particularly when they are not strictly hyperbolic. The Lax–Wendroff theorem~\cite{LaxWendroff1960} guarantees that if a consistent and conservative numerical scheme converges in $L^1$ with uniformly bounded total variation, its limit is a weak solution of the hyperbolic system. However, convergence of the scheme requires a stability property, usually expressed in terms of total variation control.

For the scalar Godunov scheme, convergence follows from the total variation diminishing property, which is not necessarily satisfied for a wide class of systems. Stability can nevertheless be established for certain strictly hyperbolic linear systems of size $(2\times2)$ belonging to Temple’s class~\cite{LeVequeTemple1985, Temple1983a, Temple1983b}, or for nonlinear systems with straight characteristic fields~\cite{Bressan2000}. Nonlinear stability can also be established using invariant domains and entropy inequalities~\cite{Bouchut2004}, notably for HLLC and kinetic solvers applied to the Euler equations. For conservative systems whose initial data have sufficiently small total variation, Glimm’s random choice method~\cite{Glimm1965} is known to be convergent. A deterministic version, based on equidistributed sampling, has also been shown to converge under the same assumptions~\cite{Liu1977}.

For non-conservative and non-strictly hyperbolic systems, convergence results for numerical schemes are limited to very specific one-dimensional systems, in particular those arising from dislocation dynamics. In this context, convergence of an implicit upwind-type scheme for a $(2\times2)$ system was proven in~\cite{MR4734513} using a discrete energy estimate. Convergence of an explicit scheme, based on viscosity solution theory and a comparison principle, was established in~\cite{MR2373180}. This latter result was extended in~\cite{alzohbi_elhajj_jazar_2022} to quasi-monotone $(d\times d)$ systems, an essential property for preserving the comparison principle. Further convergence results have been obtained for restricted classes of $(d\times d)$ systems satisfying discrete Lipschitz estimates~\cite{MR4221324}, as well as via discrete entropy estimates in~\cite{MoMo14}.

In this work, we establish the convergence of a semi-explicit scheme by first relying on a discrete entropy estimate, previously introduced in~\cite{cannone2010global} for the continuous setting, and then adapting techniques recently developed in~\cite{ZHIZ25} for a scalar equation. We also refer to~\cite{ghorbel2010well, AlvarezCarliniMonneauRouy2006} for convergence results concerning nonlocal scalar equations under monotonicity assumptions on the initial data or on the velocity field. To the best of our knowledge, no convergence results currently exist for numerical schemes for non-conservative, nonlocal hyperbolic systems in two dimensions with large-amplitude initial data.

\subsection{Mathematical structure of the initial system and description of our results}

In this section, our aim is to formally describe the mathematical structure of~\eqref{original} (or equivalently  of~\eqref{original.per}). For this, let us first recall the definition of the so-called Zygmund space :
\begin{equation}\label{defLlogL}
    L\, \log \,L(\T) = \left\{ f \in L^1_{{\rm loc}}(\T) \, : \, \int_{\T} |f|\ln( e + |f|) \, \dd x \dd y<  +\infty \right\},
\end{equation}
which is a Banach space with the norm (see~\cite{111})
\[
\|f\|_{  L\, \log \,L(\T)} = \inf \left\{ \mu > 0 : \int_{\mathbb{T}^2} \frac{|f|}{\mu} \ln \left(e + \frac{|f|}{\mu} \right) \dd x \dd y \leq 1 \right\}.
\]
We refer to the appendix for more properties of this space and, more generally, of Orlicz spaces.\medskip

In the sequel, we will assume that:
\begin{enumerate}[start=2,label=\textbf{(H{\arabic*})}]
\item \label{H2} Initial data:  Let $L>0$, the functions $\rho_0^{\pm}$ given by~\eqref{IC.original} belong to $L^2_{{\rm loc}}(\R^2)$ with
\begin{itemize}
\item[(i)] $\rho_0^{\pm} (x_1 + 1, x_2) = \rho_0^{\pm}(x_1, x_2) + L$, a.e. on $\mathbb{R}^2$.
\item[(ii)] $\rho_0 ^{\pm} (x_1, x_2 + 1) = \rho_0 ^{\pm} (x_1, x_2)$, a.e. on $\mathbb{R}^2$.
\end{itemize}
Moreover, $\pa_{x_1} \rho^{\pm}_0 \in L \log L(\R^2)$ with $\pa_{x_1} \rho^{\pm}_0 \geq 0$ a.e. on $\R^2$.\medskip
\end{enumerate}
    
Then, following~\cite{cannone2010global}, one can show that the solutions to~\eqref{original} (or equivalently to~\eqref{original.per}) satisfy the following gradient entropy estimate:
\begin{multline}\label{gradient entropy initial}
\sum_{\pm} \int_{\T} \pa_{x_1}\rho^{\pm} (t,x) \, \ln( \pa_{x_1} \rho^{\pm}(t,x)) \,dx + \int_0^t \int_{\T} \left(    
R_1^2 R_2^2 \pa_{x_1}(\rho^+(s,.) - \rho^-(s, .))(x)\right)^2 \, dx ds\\ \leq \sum_{\pm} \int_{\T} \pa_{x_1} \rho^{\pm}_0(t,x) \, \ln( \pa_{x_1} \rho^{\pm}_0 (t,x)) \,dx \quad \mbox{for a.e. }t\in(0,T).
\end{multline}
This inequality is a key estimate to establish the existence of solutions to~\eqref{original} in the distributional sense. Therefore, in order to design a convergent scheme for~\eqref{original}, we have to preserve at the discrete level this gradient entropy estimate. However, due to a lack of regularity in the velocity field, we cannot directly work with the initial model. Instead, we need to derive a regularized version of~\eqref{original.per} which admits a gradient entropy estimate similar to~\eqref{gradient entropy initial}. Hence, our paper has two main objectives:
\begin{itemize}
    \item Deriving a reliable regularized model from \eqref{original.per}.
    \item Designing a convergent numerical scheme for such regularized version of~\eqref{original.per}.
\end{itemize}
The first objective will be achieved by regularizing the singular kernel $\K$ in~\eqref{original.per} thanks to the F\'ejer kernel. Then, as in \cite{ZHIZ25}, for the second main objective, we will introduce in Section~\ref{section-scheme} a one step implicit/explicit (IMEX) Euler in time and upwind in space scheme.

\begin{remark}
    The fact that we need to regularized the system~\eqref{original} is reminiscent to~\cite{cannone2010global}. Indeed, in~\cite{cannone2010global}, the authors regularized the system~\eqref{original} by adding a viscosity term of the form $\eps \Delta \rho^{\pm}$. In this paper, following the strategy of~\cite{ZHIZ25}, we propose instead to consider a regularization of the (singular) kernel $\K$. As a by-product of our subsequent analysis, we obtain a new existence proof for the system~\eqref{original}.
\end{remark}

\subsection{Derivation of a regularized model and gradient entropy estimate}\label{section-periodic}

Let $M\geq 1$ be a given positive integer, we introduce the Fej\'er kernel $F_M : \R \to \mathbb{R}$ as: 
 \begin{equation*}
      F_M(z)= \frac{1}{M} \sum_{\ell = 0}^{M - 1}  D_{\ell}(z) \quad \forall z \in \R,  
  \end{equation*}
 where $D_\ell$ is the Dirichlet  kernel  defined by: 
   $$
 D_\ell(z)=  \sum_{|j|\leq \ell} e^{2 i \pi j z}
      \quad\mbox{for} \quad    \ell = 0, \ldots, M-1.
 $$
 Then, we define the two dimensional Fej\'er kernel $F_M^2 : \R^2 \to \mathbb{R}$ as
\begin{align}\label{fejer kernel_def}
    F^2_M(x) = F_M(x_1) \, F_M(x_2), \quad \forall x=(x_1,x_2) \in \R^2,
\end{align}
and we regularize the kernel $\K$ appearing in~\eqref{original} through the ``generalized'' Ces\`aro mean of order $M$ of the Fourier series of $\K$:
\begin{align}\label{def sigmaM}
\sigma_M^{\mathcal{K}}(x)  = \frac{1}{M^2} \sum_{n_1, \, n_2 = 0}^{M-1} \sum_{|m_1| \leq n_1}\sum_{|m_2|\leq n_2} c_{(m_1,m_2)}(\K)\, e^{2i \pi(m_1 x_1+m_2 x_2)}, \quad \forall x=(x_1,x_2) \in \mathbb{T}^2
\end{align}
with $ c_{(m_1,m_2)}(\K) = \frac{m_1^2m_2^2}{|m|^4}$. Classical computations show that
\begin{align*}
\sigma_M^{\mathcal{K}}(x) = F^2_M \ast \mathcal{K}(x), \quad \forall x \in \mathbb{T}^2,
\end{align*}
where also here $\ast$ denotes (as in \eqref{def_conv0}) the periodic convolution defined in the distributional sense on the torus $\T$. Hence, in this paper, we consider the following regularized system: 
\begin{equation}\label{original1}
    \begin{cases}
  \partial_t\rho^{+,\per}_M = - \left(\sigma_M^{\mathcal{K}} \ast(\rho^{+,\per}_M - \rho^{-,\per}_M)+a(t)\right)\, (\partial_{x_1}\rho^{+,\per}_M+L)  & \mbox{in }\mathbb{T}^2 \times (0, T), \\
   \partial\rho^{-,\per}_M=  \;\;\;\;\left(\sigma_M^{\mathcal{K}}\ast(\rho^{+,\per}_M- \rho^{-,\per}_M )+a(t)\right) \, (\partial_{x_1}\rho^{-,\per}_M+L)  & \mbox{in }\mathbb{T}^2 \times (0, T),
\end{cases}
\end{equation}
complemented with the following regularized initial data
\begin{align}\label{IC.reg.per}
\rho^{\pm,\per}_M(0, x) =  \rho_{M,0} ^{\pm, \per}(x) := F_M^2 \ast \rho^{\pm,\per}_0(x)
= \int_{\T} F_M^2(x-y) \rho^{\pm,\per}_0 (y)\, dy, \quad \forall x\in\T.
\end{align}
Thanks to~\textbf{(H1)}, we notice that $\rho^{\pm,\per}_{M,0} \in C^{\infty}(\T)$ and in 
particular $\rho^{\pm,\per}_{M,0} \in W^{1,\infty}(\T)$. Of course this Lipschitz regularity blows up as $M \to \infty$. Let us now derive (formally) the gradient entropy estimate satisfy by the solutions of~\eqref{original1}. Formal computations lead to
\begin{multline*}
\dfrac{d}{dt} \sum_{\pm} \int_{\T} \left(\pa_{x_1} \rho^{\pm,\per}_M(t,x) + L \right) \ln\left(\pa_{x_1} \rho^{\pm,\per}_M(t,x) + L \right) dx \\= - \int_{\T} \
\sigma_M^{\mathcal{K}}\ast \left(\pa_{x_1}\rho^{+,\per}_M(t,\cdot)-\pa_{x_1}\rho^{-,\per}_M(t,\cdot)\right)(x) \left(\pa_{x_1}\rho^{+,\per}_M(t,x)-\pa_{x_1}\rho^{-,\per}_M(t,x)\right) dx.
\end{multline*}
Then, applying Parseval's equality we get
\begin{multline}\label{regul gradient ent}
\displaystyle\dfrac{d}{dt} \sum_{\pm}  \int_{\T} \left(\pa_{x_1} \rho^{\pm,\per}_M(x,t) + L \right) \ln\left(\pa_{x_1} \rho^{\pm,\per}_M(t,x) + L \right) \,dx \\= 
\displaystyle- \sum_{m \in \Z^2} c_m\left(\sigma_M^{\mathcal{K}}\right) \, \left|c_m\left(\pa_{x_1} \left(\rho^{+,\per}_M-\rho^{-,\per}_M\right)\right)\right|^2.
\end{multline}
Moreover, for any $m = (m_1,m_2) \in \Z^2$ we have, thanks to~\eqref{coeff Fourier K},
\begin{align*}
c_m\left(\sigma_M^{\mathcal{K}}\right) = c_m(F_M^2) \,\dfrac{m_1^2 \, m_2^2}{|m|^4} = c_{m_1}(F_M) \, c_{m_2}(F_M) \, \dfrac{m_1^2 \, m_2^2}{|m|^4}, 
\end{align*}
where, for any $p \in \Z$, we have
\begin{align*} 
c_p(F_M) =
\left\{
    \begin{array}{ll}
        \left(1-\dfrac{|p|}{M}\right) & \mbox{if } |p|< M, \\
        0 & \mbox{if } |p|\geq M.
    \end{array}
\right.
\end{align*}
Therefore, we deduce that the right hand side of~\eqref{regul gradient ent} is nonpositive. Then, thanks to a discrete version of~\eqref{regul gradient ent}, we will be able to prove the convergence of the solutions to~\eqref{original1}--\eqref{IC.reg.per} toward the solutions of the periodic initial system~\eqref{original.per}--\eqref{IC.original.per} as $M \to +\infty$, see Theorem~\ref{convergence} below. However, let us already emphasize that the dissipation term appearing in~\eqref{regul gradient ent} (and similarly for~\eqref{gradient entropy initial}) only yields estimate on the $x_1$-derivative of the function $\rho^{+,\per}_M-\rho^{-,\per}_M$. Therefore, the derivation of convenient estimates on the gradient of this function requires some cares.

\subsection{Outline of the paper} The paper is structured as follows. In Section~\ref{section-scheme}, we propose a semi-explicit numerical scheme and state our main theorems. We study in Section~\ref{proof of the first theorem} its main properties on fixed grids in space and time
and we establish uniform estimates with respect to the  discretization steps in Section~\ref{Uniform-estimates}. We next prove the convergence of the proposed scheme for fixed values of $M$ in Section~\ref{convergence for fixed M}. In addition, Section~\ref{section M to infinity} is devoted to the proof of convergence of the regularized system to the original system as $M$ tends to infinity. Finally, numerical experiments are presented in Section~\ref{section-numerical}.

\section{Numerical scheme and main results}\label{section-scheme}

\subsection{Numerical scheme} Let us now introduce a numerical scheme for~\eqref{original1}--\eqref{IC.reg.per}. We start by introducing some notation. First,  for the uniform meshes in time and space, we introduce two positive integers $N \geq 1$ and $N_T$. Then, we define the space step $\Delta x= \Delta x_1 = \Delta x_2 = 1/N$, the time step $\Delta t = T/N_T$ and the sequences 
\begin{align*}
    t_n = n\Delta t, \quad \forall n \in \{0, ..., N_T\},
\end{align*}
and
\begin{align*}
    x_i = i \Delta x, \quad \forall i \in \I =\Z/N\Z.
\end{align*}
We also define the (open) squares:
\begin{align}\label{def : square}
C_{i,j} = (x_i,x_{i+1}) \times (x_j, x_{j+1}), \quad \forall (i,j) \in \I^2.
\end{align}
Eventually, for any $x \in \R$, we introduce its positive and negative parts as
\begin{align}\label{pos neg pats}
    x_+ = \frac{x+|x|}{2}, \quad x_- = \frac{|x|-x}{2}.
\end{align}
Then, for a given $M\geq 1$, we start by discretizing the initial condition as
\begin{align}\label{initial data}
      \rho^{\pm, 0, \per}_{M,i,j} = \rho_{M,0}^{\pm, \per}(x_i, x_j), \quad \forall (i,j) \in \I^2.
\end{align}
Now, let $n \in \{0,\ldots, N_T-1\}$ be fixed and let $\rho^{\pm, n,\per}_M= (\rho_{M,i,j}^{\pm,n,\per})_{(i, j) \in \I^2}$ be a given vectors. We define the vectors $\rho^{\pm,n + 1,\per}_M = (\rho_{M,i,j}^{\pm,n+1,\per})_{(i, j) \in \I^2}$ as the solution, for any $(i,j) \in \I^2$, of the following system:
\begin{equation}\label{compacted numerical scheme}
\dfrac{\rho_{M,i, j}^{\pm,n + 1, \per} - \rho_{M,i,j}^{\pm, n, \per}}{\Delta t} = \lambda^{\pm}_{i, j}[\rho^{n + 1,\per}_M]_+\bigg(\theta^{\pm,  n, \per}_{M,i +1/2, j} + L\bigg) - \lambda^{\pm}_{i, j}[\rho^{n + 1,\per}_M]_-\bigg(\theta^{\pm,  n, \per}_{M,i -1/2, j} + L\bigg),
\end{equation}
with
\begin{align}\label{diff rho plus moins}
    \rho^{n+1,\per}_{M,i,j} =  \rho^{+,n+1,\per}_{M,i,j} - \rho^{-,n+1,\per}_{M,i,j}, \quad \forall (i,j) \in \I^2,
\end{align}
and where $\theta^{ \pm, n,  \per}_{M,i +1/2, j}$ denotes the discrete partial $x_1$-derivative of $\rho^{\pm,n,\per}_{M,i,j}$ defined by 
\begin{equation}\label{theta_i}
    \theta^{\pm,  n, \per}_{M,i +1/2, j} = \frac{\rho^ {\pm,n,\per}_{M,i + 1, j} - \rho^{\pm,n, \per}_{M,i, j}}{\Delta x}, \quad \forall (i,j) \in \I^2.
\end{equation}
We discretize the nonlocal velocity field as 
\begin{align}\label{lambda}
    \lambda^{\pm}_{i,j}[\rho^{n + 1,\per}_M] = \mp
    \left(a(t_{n+1})+\sum_{(\ell,r) \in \I^2} (\Delta x)^2 \, \sigma_{M, \ell, r}^{\K} \, \rho^{n + 1, \per}_{M, i - \ell, j - r} \right), \quad \forall (i, j)     
    \in \I^2,
\end{align}
where, recalling definition~\eqref{def sigmaM} of $\sigma_M^\K$, we set
\begin{align}\label{sigma_def}
\sigma_{M, i, j}^\K = \sigma_{M}^\K(x_i, x_j) = (F_M \ast \K)(x_i, x_j), \quad \forall (i,j) \in \I^2.
\end{align}
Eventually, for any $0 \leq n \leq N_T$, we define the vectors $\rho^{\pm,n}_M=(\rho^{\pm,n}_{M,i,j})_{(i,j)\in \Z^2}$ as
\begin{align}\label{def rho pas per}
\rho^{\pm,n}_{M,i,j} = \rho^{\pm,n,\per}_{M,i,j} + L \, i \,\Delta x, \quad \forall (i,j) \in \Z^2.
\end{align}



\subsection{Main results}

In the following, we denote by $f :  \R_+ \to \R_+$ the function defined by
\begin{align}\label{function f}
    f(x) = x\ln(x+e), \qquad \forall x \in \R_+. 
\end{align}
Moreover, for any vector $v = (v_{i,j})_{(i,j) \in \I^2}$, we define its mean value in the $x_1$-direction as
\begin{align}\label{mean_expression}
\langle v \rangle_j :=  \sum_{i \in \I} \Delta x \, v_{i,j}, \quad \forall j \in \I,
\end{align}
and its $\ell^\infty$ norm as
\begin{align*}
\|v\|_{\ell^\infty(\I^2)} :=\max_{(i, j) \in \I^2} \left| v_{i, j}\right|.
\end{align*}
Our first main result deals with the well-posedness of the scheme and the qualitative properties of its solutions. 

\begin{theorem}[Well-posedness of the scheme]\label{theorem-properties}
Let assumptions~\emph{\ref{a}} and~\emph{\ref{H2}} hold, and assume that $N \geq M\geq 1$ with
\begin{align}\label{delta t delta x}
\dfrac{\Delta t}{\Delta x} < 
\min \left\{\frac{1}{4(4 M^2 L + \| a\|_{L^{\infty}(0,T)})}, 
\frac{1}{18 M^2L} \right\}
\end{align}
and
\begin{align}\label{delta t}
\Delta t < \frac{1}{2 L(4 M^2 L + \| a\|_{L^{\infty}(0,T)})}.
\end{align} 
Then, the scheme~\eqref{initial data}--\eqref{sigma_def} admits a unique solution $\rho^{\pm,n,\per}_M = (\rho^{\pm,n,\per}_{M,i,j})_{(i,j)\in \I^2}$ for any $1\le n \le N_T$, such that   
\begin{align}\label{bound of rho}
   \max_{\pm}\left( \|\rho^{\pm,n,\per}_M\|_{\ell^\infty(\I^2)}\right) &\leq 
    \max_{\pm}\left(\|\rho_{M,0}^{\pm,\per}\|_{L^{\infty}(\T)}\right) 
    + L\left(4 M^2 L + \| a\|_{L^{\infty}(0,T)}\right)T \qquad \forall 1\leq n \leq N_T,
\end{align}
and
\begin{align}\label{bound rho mass rho} 
  \max_{(i,j)\in \I^2} \left|\rho^{\pm, n, \per}_{M,i,j} - \langle \rho^{\pm, n ,\per}_M\rangle_j\right|
     &\leq 2 \, L, \qquad \forall 1 \leq n \leq N_T.   
\end{align}
Moreover, for  $0\le n \le N_T-1$, we have
\begin{multline}\label{gradient entropy estimate}
\sum_{\pm} \sum_{(i,j)\in \I^2}(\Delta x)^2 \, f\left(\theta_{M,i +1/2, j}^{\pm, n+1, \per} + L\right) + \sum_{k=0}^n \Delta t  \, \mathcal{D}\left[\rho^{k+1,\per}_M\right] \\  \leq   \sum_{\pm} \sum_{(i,j)\in \I^2}(\Delta x)^2\, f\left(\theta_{M,i +1/2,j} ^{\pm,0, \per} + L\right) + 2\left(f(e) + L \ln(2)\right), 
\end{multline}
where we recall definition~\eqref{function f} of $f$ and with $\mathcal{D}\left[\rho^{n,\per}_M\right]$ a discrete dissipation term given, for any $1\leq n \leq N_T$, by
\begin{align}\label{discrete dissipation}
\mathcal{D}\left[\rho^{n,\per}_M\right] =  \sum_{m \in \I^2} \, c^d_m\left(\bar{\sigma}_M^\K\right) \, \left|c^d_m\left(\theta^{+,n,\per}_{M,x_1} -\theta^{-,n,\per}_{M,x_1}\right)\right|^2 \geq 0.
\end{align}
In the above formula, $c_m^d(\cdot)$ denotes the Discrete Fourier Transform (DFT) coefficient (defined in \eqref{c-DFT}), and
\[
\bar{\sigma}_M^\K = \left(\sigma_{M,i,j}^\K\right)_{(i,j)\in\I^2}, \qquad \theta^{\pm,n+1,\per}_{M,x_1} = \left(\theta^{\pm,n+1,\per}_{M,i+1/2,j}\right)_{(i,j)\in\I^2}.
\]
\end{theorem}

The proof of Theorem~\ref{theorem-properties} is done in Section~\ref{proof of the first theorem} and, except for estimate~\eqref{bound rho mass rho}, its proof is quite similar to the proof of~\cite[Theorem 1]{ZHIZ25}. More precisely, we prove first the existence of an unique solution thanks to Banach's fixed point theorem. Then, we prove the main properties of this solution.

In our second main result, we establish the convergence of the numerical scheme~\eqref{initial data}--\eqref{sigma_def}. In order to state the result, we introduce some further notation. Let the assumptions of Theorem~\ref{theorem-properties} hold. Then, for any $\Delta t$ and $\Delta x$, we introduce the size of the meshes $\eps : = \max(\Delta t, \Delta x)$, and the following function:
\begin{align}\label{defining at m of rho}
    \rho^{\pm, \per, \varepsilon}_M(t_n, x_i, x_j) = \rho^{\pm, n, \per}_{M,i,j}, \quad \forall (i,j)\in\I^2, \, 0\leq  n \leq N_T.
\end{align}
Now, recalling definition~\eqref{def : square} of $C_{i,j}$, we also define, for any $(t, x_1, x_2) \in [t_n, t_{n + 1}] \times \overline{C}_{i,j}$, the so-called $Q^1$ extension, still denoted by $\rho^{\pm,\per, \eps}_M$, as
\begin{align}\label{continuous rho}
     \rho^{\pm, \per, \eps}_M(t, x_1, x_2) &= \left(\frac{t - t_n}{\Delta t}\right) \left[ \left(\frac{x_1 - x_i}{\Delta x}\right) \left(\frac{x_2 - x_j}{\Delta x}\right) \rho_{M,i+1, j+1}^{\pm, n + 1,\per} + \left(1 - \frac{x_1 - x_i}{\Delta x}\right) \left(\frac{x_2 - x_j}{\Delta x}\right) \rho_{M,i, j+1}^{\pm,n + 1,  \per} \right. \\
&\left. + \left(\frac{x_1 - x_i}{\Delta x}\right) \left(1 - \frac{x_2 - x_j}{\Delta x}\right) \rho_{M,i+1, j}^{\pm,n + 1, \per} + \left(1 - \frac{x_1 - x_i}{\Delta x}\right) \left(1 - \frac{x_2 - x_j}{\Delta x}\right) \rho_{M,i, j}^{\pm,n + 1, \per} \right] \nonumber\\
&+ \left(1 - \frac{t - t_n}{\Delta t}\right) \left[ \left(\frac{x_1 - x_i}{\Delta x}\right) \left(\frac{x_2 - x_j}{\Delta x}\right) \rho_{M,i+1, j+1}^{\pm, n, \per} + \left(1 - \frac{x_1 - x_i}{\Delta x}\right) \left(\frac{x_2 - x_j}{\Delta x}\right) \rho_{M,i, j+1}^{\pm, n, \per} \right. \nonumber\\
&\left. + \left(\frac{x_1 - x_i}{\Delta x}\right) \left(1 - \frac{x_2 - x_j}{\Delta x}\right) \rho_{M,i+1, j}^{\pm, n, \per} + \left(1 - \frac{x_1 - x_i}{\Delta x}\right) \left(1 - \frac{x_2 - x_j}{\Delta x}\right) \rho_{M,i, j}^{\pm, n, \per} \right],\nonumber
 \end{align}
so that $\rho^{\pm,\per,\eps}_M(t_n,x_i,x_j)=\rho^{\pm,n,\per}_{M,i,j}$. In the sequel, we consider a sequence of meshes such that $\eps_m = \max(\Delta t_m, \Delta x_m)$ satisfies $\eps_m \to 0$ as $m \to +\infty$.  Moreover, we will denote the sequence of associated reconstruction function $(\rho^{\pm, \per,m}_M)_{m \in \mathbb{N}}$ defined through
\[ 
\rho^{\pm, \per, m}_M(t, x_1,x_2) =  \rho^{\pm, \per, \eps_m}_M(t, x_1,x_2), \quad \forall (t,x_1,x_2) \in (0, T) \times \T.
\]
Then, our main objective is to show that the sequences  $( \rho^{\pm,\per, m}_M)_{m \in \mathbb{N}}$ converge toward a solution (in the distributional sense) of the regularized system~\eqref{original1}.

\begin{theorem}[Convergence of the scheme]\label{convergence}
Let us assume that the assumptions of Theorem~\ref{theorem-properties} hold with $M \leq N_m = 1/\Delta x_m$ for any $m$, and let $\eps_m = \max(\Delta t_m,\Delta x_m)$ be a sequence tending to zero as $m \to \infty$, with each component satisfying~\eqref{delta t delta x}--\eqref{delta t}. Let $(\rho^{\pm,\per,m}_M)_{m \in \mathbb{N}}$ be the associated family of reconstruction functions of the numerical scheme~\eqref{initial data}--\eqref{sigma_def}. Then, there exist functions 
\[
\rho^{\pm,\per}_M \in W^{1,\infty}((0,T)\times\T),
\]
such that, up to a subsequence, as $m \to \infty$, it holds
\begin{align*}
\rho^{\pm,\per,m}_M \to \rho^{\pm,\per}_M \quad \mbox{strongly in } C([0,T]\times \T).
\end{align*}
Moreover, the functions $\rho^{\pm,\per}_M$ are solutions to the regularized system~\eqref{original1}-\eqref{IC.reg.per} in the distributional sense.
\end{theorem}

The proof of Theorem~\ref{convergence} is done in Section~\ref{convergence for fixed M} and relies on a compactness approach. More precisely, we first establish estimates on the gradient of $\rho^{\pm,\per,m}_M$ uniform in $m$. For this, we notice that the discrete gradient entropy~\eqref{gradient entropy estimate} only yields uniform estimate on $\pa_{x_1} \rho^{\pm,\per,m}_M$. Therefore, thanks to our regularization approach we will show an estimate on $\pa_{x_2} \rho^{\pm,\per,m}_M$ uniform in $m$ but non uniform in $M$. Then, these estimates, and others, will provide enough compactness properties on the sequences $(\rho^{ \pm,\per,m}_M)_{m\in\N}$. Finally, we will identify these limit functions as solutions in the distributional sense of~\eqref{original1}-\eqref{IC.reg.per}.

Finally, our last main result deals with the convergence of the solutions to the regularized system~\eqref{original1}-\eqref{IC.reg.per}~toward the solutions of the initial system~\eqref{original.per}--\eqref{IC.original.per} as $M \to +\infty$.

\begin{theorem}[Limit $M$ to $\infty$]\label{theorem-final convergence}
Let the assumptions of Theorem~\ref{convergence} hold and let \( (M_{\ell})_{\ell \in \mathbb{N}} \) be a sequence of positive integers such that $M_\ell \to \infty$ as $\ell \to \infty$. Finally, let \( \left(\rho^{\pm, \mathrm{per}}_\ell\right)_{\ell \in \mathbb{N}}= 
 \left(\rho^{\pm, \mathrm{per}}_{M_\ell}\right)_{\ell \in \mathbb{N}}
 \)
 be the associated sequences of solutions to~\eqref{original1}--\eqref{IC.reg.per} constructed in Theorem~\ref{convergence}. Then, there exist functions $\rho^{\pm,\per}$ such that, up to a subsequence, as $\ell \to \infty$, it holds
\begin{align*}
\rho^{\pm, \per}_\ell &\rightharpoonup \rho^{\pm, \per} \quad \mbox{weakly in } L^2((0,T) \times \mathbb{T}^2),\\
\partial_{x_1} \rho^{\pm, \per}_\ell &\overset{*}{\rightharpoonup} \partial_{x_1} \rho^{\pm, \per} \quad \mbox{in } L^\infty((0,T); L \log L(\mathbb{T}^2)),\\
\sigma_{M_\ell}^\K \ast (\rho^{+, \per}_\ell - \rho^{-,\per}_\ell) &\rightarrow \K\ast (\rho^{+,\per} - \rho^{-,\per}) \quad \mbox{strongly in } L^1(0,T;E_\exp(\T)),
\end{align*}
where we refer to Appendix~\ref{app2} for the definition of the Orlicz space $E_\exp(\T)$. Furthermore, the functions $\rho^{\pm,\per}$ are solutions to~\eqref{original.per}--\eqref{IC.original.per} in the distributional sense.
\end{theorem}

The proof of Theorem~\ref{theorem-final convergence} is done in Section~\ref{section M to infinity}. The key argument is to show that $\sigma_M^\K\ast(\rho^{+,\per}_M-\rho^{-\per}_M)$ is uniformly bounded in $L^2(0,T;H^1(\T))$. For this purpose, we will establish a convenient estimate on the $L^2(0,T;H^1(\T))$ norm of
\begin{align*}
\sigma_M^\K \ast(\rho^{+,\per,m}_M-\rho^{-\per,m}_M) \in L^2(0,T;H^1(\T)).
\end{align*}
Then, thanks to this result and adapting the arguments of~\cite{cannone2010global} we will prove Theorem~\ref{theorem-final convergence}.

\section{Well-posedness and qualitative properties of the scheme}\label{proof of the first theorem}

In this section, we prove Theorem~\ref{theorem-properties} in several steps and in order to lighten the notation, we will neglect $M$ in the indices of the vectors $\rho^{\pm,n,\per}_M$ for any $0\leq n \leq N_T$. As already said, this proof is quite similar to the proof of~\cite[Theorem 1]{ZHIZ25}. Therefore, we will only focus on the main differences and refer the interested reader to~\cite{ZHIZ25} for more details. But first, let us show some technical results needed for the proof of Theorem~\ref{theorem-properties}.

\subsection{Technical lemmas}

\begin{lemma}\label{lemma bound sigmaM}
Let $N \geq M \geq 1$ be fixed. Then, $\sigma_M^\K$ the generalized Ces\`aro mean of order $M$ of $\K$ satisfies the following bound           
\begin{align}\label{bound of the kernel}
|\sigma_M^\K (x) | \leq M^2, \quad \forall x \in \T.
\end{align}
\end{lemma}

\begin{proof}
Let $x=(x_1, x_2) \in \T$ be fixed. We first notice, thanks to the definition~\eqref{def sigmaM} of $\sigma_M^\K$, that it holds
\begin{align*}
|\sigma_M^\K (x)| \leq \frac{1}{M^2} \sum_{n_1, \, n_2 = 0}^{M-1} \sum_{|m_1| \leq n_1}\sum_{|m_2|\leq n_2} |c_{(m_1,m_2)}(\K)|.
\end{align*}
Now, according to~\eqref{coeff Fourier K}, we deduce that $|c_{(m_1,m_2)}(\K)|\leq 1$ for any $(m_1,m_2) \in \Z^2$. Therefore, we readily conclude that $|\sigma_M^\K(x)| \leq M^2$.
\end{proof}

\begin{lemma}\label{discrete-Poincare-Wirtinge}
Let $w=(w_i)_{i \in \I}$ be an $N$-periodic sequence in $\R^\Z$ satisfying
\begin{equation}\label{cond-positive}
w_{i+1}-w_i + L\Delta x \ge 0, \qquad \forall i \in \I.
\end{equation}
Then, it holds $\|w - \langle w\rangle\|_{\ell^\infty(\I)} \leq 2L$, where $\langle w\rangle = \sum_{j \in \I} \Delta x \, w_j$.
\end{lemma}

\begin{proof}
Applying a discrete Poincaré--Wirtinger like inequality, we obtain
\begin{align*}
\|w- \langle w\rangle\|_{\ell^\infty(\I)}
\leq \sum_{i =0}^{N-1} \left|w_{i+1}-w_{i}\right| 
&\leq L + \sum_{i=0}^{N-1} \left| (w_{i+1}-w_i) + L \Delta x\right|\\
&\leq L + \sum_{i=0}^{N-1} \left( (w_{i+1}-w_i) + L \Delta x \right) \\
&= L + w_{N}-w_0 + L = 2L.
\end{align*}
Here, we have used hypothesis~\eqref{cond-positive} in the second line and the periodicity of $w$ in the last line.
\end{proof}

\begin{lemma}\label{borne-sur-lambda}
Let $v = v^+ - v^-$, with $v^\pm = (v^\pm_{i,j})_{(i,j) \in \I^2}$ two $N$-periodic sequences
in $\R^{\Z^2}$, such that for each $j \in \I$, the sequence $(v^\pm_{i,j})_{i \in \I}$ satisfies hypothesis~\eqref{cond-positive}, i.e.,
\[
v^\pm_{i+1,j} - v^\pm_{i,j} + L \Delta x \geq 0, \qquad \forall (i,j) \in \I^2.
\]
Then
\begin{align}\label{borne-vitesse}
|\lambda^\pm_{i,j}[v]| \le 4 M^2 L + \| a\|_{L^{\infty}(0,T)},
\qquad \forall (i,j) \in \I^2.
\end{align}
\end{lemma}

\begin{proof}
From definitions \eqref{lambda} an \eqref{mean_expression}, for all $(i,j) \in \I^2$, we have
\begin{align*}
&|\lambda^{\pm}_{i,j}[v]|
= \left|
a(t_{n+1})+ \sum_{(\ell,r) \in \I^2} (\Delta x)^2 \, \sigma_{M, \ell, r}^{\K} \,
v_{ i - \ell, j - r}  \right| \\
&\le \| a\|_{L^{\infty}(0,T)} 
+ \sum_{\pm}\left|
\sum_{(\ell,r) \in \I^2} (\Delta x)^2 \, \sigma_{M, \ell, r}^{\K} \,
\left(v^{\pm}_{ i - \ell, j - r} - \langle v^{\pm} \rangle_{j - r}\right) \right|
+ \left|
\sum_{(\ell,r) \in \I^2} (\Delta x)^2 \, \sigma_{M, \ell, r}^{\K} \,
\langle v \rangle_{j - r} \right|.
\end{align*}
Using Lemmas \ref{lemma bound sigmaM} and \ref{discrete-Poincare-Wirtinge}, this gives
\begin{align*}
|\lambda^{\pm}_{i,j}[v]|
\le \| a\|_{L^{\infty}(0,T)} + 4 M^2 L
+ \left|
\sum_{(\ell,r) \in \I^2} (\Delta x)^2 \, \sigma_{M, \ell, r}^{\K} \,
\langle v \rangle_{j - r} \right|.
\end{align*}
To conclude, it remains to show that the last term is zero, i.e.,
\begin{align}\label{nulle-vitesse}
\sum_{(\ell,r) \in \I^2} (\Delta x)^2 \, \sigma_{M, \ell, r}^{\K}
\, \langle v \rangle_{j - r} = 0.
\end{align}
Indeed, according to the definition of $\sigma_M^\K$ (see \eqref{def sigmaM}), we have
\begin{align*}
\sum_{(\ell,r) \in \I^2} (\Delta x)^2 \,& \sigma_{M, \ell, r}^{\K}
\, \langle v \rangle_{j - r}
= \sum_{r \in \I} (\Delta x)^2 \langle v \rangle_{j - r}
\sum_{\ell \in \I} \sigma_{M}^{\K}(x_\ell,x_r) \\
&= \sum_{r \in \I} (\Delta x)^2 \langle v \rangle_{j - r}
\sum_{\ell \in \I}
\left(\frac{1}{M^2} \sum_{n_1,n_2=0}^{M-1}
\sum_{|m_1| \le n_1}\sum_{|m_2|\le n_2}
c_{(m_1,m_2)}(\K)\,
e^{2i \pi m_1 \ell \Delta x } e^{2i \pi m_2 r \Delta x} \right) \\
&= \frac{1}{M^2}
\sum_{r \in \I} (\Delta x)^2 \langle v \rangle_{j - r}
\left( \sum_{n_1,n_2=0}^{M-1}
\sum_{|m_1| \le n_1}\sum_{|m_2|\le n_2}
c_{(m_1,m_2)}(\K)\, e^{2i \pi m_2 r \Delta x} \right)
\sum_{\ell \in \I} e^{2i \pi m_1 \ell \Delta x}.
\end{align*}
Since $c_{(m_1,m_2)}(\K)=0$ when $m_1=0$ (see \eqref{coeff Fourier K}) and
$\sum_{\ell \in \I} e^{2i \pi m_1 \ell \Delta x } = 0$ for $m_1 \neq 0$ and
$1 \le M \le N$, we finally obtain \eqref{nulle-vitesse}, and therefore \eqref{borne-vitesse}.
\end{proof}

\subsection{Well-posedness of the scheme}

We prove the well-posedness of the scheme by induction. Hence, let $0 \le n \le N_T - 1$ be fixed, and let $\rho^{ \pm, n,\per} = (\rho^{\pm, n, \per}_{i,j})_{(i,j) \in \I^2}$ be given vectors such that 
\begin{align}\label{hyp induction 1}
   \max_{\pm} \left(\| \rho^{\pm, n, \per}\|_{\ell^\infty(\I^2)}\right) &\le \beta = 
    \max_{\pm} \left(\|\rho_0^{\pm,\per}\|_{L^{\infty}(\T)}\right)+ L\left(4 M^2 L + \| a\|_{L^{\infty}(0,T)}\right) T.
\end{align}
Finally, we also assume that  $\theta^{\pm,  n, \per}_{i+1/2, j} + L \geq 0$ for any $(i,j) \in \I^2$.

\subsubsection{Construction of a contraction mapping and first properties}\label{existence-sch}
In the set $\mathcal{A}$, of the $N$ periodic sequences in $(\R^{\Z^2})^2$, we introduce the following compact subset: 
\[
\mathcal{U}_\beta = \left\{ (v^+, v^-)  \in \mathcal{A} \, : \, \max_{\pm}\left(\|v^\pm\|_{\ell^\infty(\I^2)}\right) \leq \beta+1 
\mbox{ and } \left(v^\pm_{i,j}\right)_{i \in \I} \mbox{ satisfies } \eqref{cond-positive} \mbox{ for each } j \in \I
 \right\}.
\]
Now, for any $(i, j) \in \I^2$, we define the maps $F^{\pm}_{i,j} : \mathcal{U}_\beta \to \mathbb{R}$ as follows:
\begin{equation}\begin{array}{ll}\label{definitionF_plus}
    F^\pm_{i,j}(v^\pm)
    & \displaystyle = \rho_{i,j}^{\pm,n , \per} + \Delta t \lambda^\pm_{i,j}[v]_+ \,  \theta_{i+1/2,j}^{\pm,n,\per}  - \Delta t \lambda^\pm_{i,j} [v]_- \, \theta_{i-1/2,j}^{\pm,n,\per}  + L \,\Delta t \, \lambda^\pm_{i,j}[v],\\
 & \displaystyle= \left(1- \frac{\Delta t}{\Delta x} \left|\lambda^\pm_{i,j}[v]\right| \right)\rho_{i,j}^{\pm,n , \per} + 
\frac{\Delta t}{\Delta x} \lambda^\pm_{i,j}[v]_+  \rho_{i+1,j}^{\pm,n , \per} 
 + \frac{\Delta t}{\Delta x} \lambda^\pm_{i,j} [v]_- \rho_{i-1,j}^{\pm,n , \per}
  + L \,\Delta t \, \lambda^\pm_{i,j}[v],
\end{array}
\end{equation}
where $v = v^+-v^-$. We notice that the solution of~\eqref{compacted numerical scheme} at step $n+1$, is given by 
\begin{equation}\label{solution_definitionF_plus}
\rho^{\pm,n + 1, \per}_{i, j} = F^\pm_{i, j}(\rho^{\pm, n+1, \per}), \quad \forall (i,j) \in \I^2.
\end{equation}
We now observe that, by Lemma \ref{borne-sur-lambda}, if $(v^+,v^-)\in \mathcal{U}_\beta$, then
$\lambda^\pm_{i,j}[v]$ satisfies the bound \eqref{borne-vitesse}. As a consequence, the range of
$F^\pm_{i,j}$ is contained in the interval
$\left[ -(\beta+1),(\beta+1) \right]$. Indeed, under conditions
\eqref{delta t delta x}--\eqref{delta t} and assumption \eqref{hyp induction 1}, we have
\begin{align*}
  |F_{i, j}^\pm(v^\pm)| \leq \| \rho^{\pm, n, \per}\|_{\ell^\infty(\I^2)}  + L \,\Delta t \, |\lambda^\pm_{i,j}[v]|
  \leq \beta + \Delta t L(4 M^2 L + \| a\|_{L^{\infty}(0,T)}) \le \beta +1. \quad 
  \forall  (i,j) \in \I^2. 
\end{align*}
To establish that the mappings $F_{i, j}^\pm$ are well-defined in $\mathcal{U}_\beta$, it remains to show that the sequence $(F_{i, j}^\pm(v^\pm))_{i \in \I}$ satisfies \eqref{cond-positive} for each $j \in \I$. To this end, using \eqref{definitionF_plus}, we compute
\[
DF_{i+\frac {1}{2}, j}^\pm(v^\pm)
= \frac{F_{i+1, j}^\pm (v^\pm) - F_{i, j}^\pm(v^\pm)}{\Delta x},
\]
which gives
\begin{multline*}
DF_{i+\frac {1}{2}, j}^\pm(v^\pm) = 
\bigg(1 - \frac{\Delta t}{\Delta x} \big( 
\lambda^\pm_{i,j}[v]_+ + \lambda^\pm_{i+1,j}[v]_-\big) \bigg)  \theta_{i+1/2,j}^{\pm,n,\per} 
+ \frac{\Delta t}{\Delta x} \lambda^\pm_{i+1,j}[v]_+ \,\theta_{i+3/2,j}^{\pm,n,\per} \\
+ \frac{\Delta t}{\Delta x} \lambda^\pm_{i,j}[v]_- \, \theta_{i-1/2,j}^{\pm,n,\per}
+ L \frac{\Delta t}{\Delta x}\big( \lambda^\pm_{i+1,j}[v] - \lambda^\pm_{i,j}[v] \big).
\end{multline*}
By adding $L$ to both sides, this can be rewritten equivalently as
\begin{multline}\label{theta_convexcombin0}
DF_{i+\frac {1}{2}, j}^\pm(v^\pm) + L = 
\bigg(1 - \frac{\Delta t}{\Delta x} \big( 
\lambda^\pm_{i,j}[v]_+ + \lambda^\pm_{i+1,j}[v]_-\big) \bigg)  (\theta_{i+1/2,j}^{\pm,n,\per}+L) \\
+ \frac{\Delta t}{\Delta x} \lambda^\pm_{i+1,j}[v]_+ \,(\theta_{i+3/2,j}^{\pm,n,\per}+L) 
+ \frac{\Delta t}{\Delta x} \lambda^\pm_{i,j}[v]_- \, (\theta_{i-1/2,j}^{\pm,n,\per}+L).
\end{multline}
It then follows from the CFL condition \eqref{delta t delta x} and the bound \eqref{borne-vitesse} that 
$DF_{i+\frac {1}{2}, j}^\pm(v^\pm) + L$ is a convex combination of non-negative terms,
since, by the induction hypothesis, $\theta_{i+1/2,j}^{\pm,n,\per}+L \ge 0$ for all $i \in \I$.  
Hence,
\[
DF_{i+\frac {1}{2}, j}^\pm(v^\pm) + L \ge 0,
\]
which in particular shows that the sequence $(F_{i, j}^\pm(v^\pm))_{i \in \I}$ satisfies \eqref{cond-positive}.

Now, we define the map $F \;\colon\; \mathcal{U}_\beta \to \mathcal{U}_\beta$ by
\[
F(v^+,v^-) = \left( F^+_{i,j}(v^+), F^-_{i,j}(v^-) \right)_{(i,j) \in \I^2}, 
\]
where $F^\pm_{i,j}$ are defined in \eqref{definitionF_plus}.  
Our main objective is to show that $F$ is a contraction mapping on $\mathcal{U}_\beta$.  
For this purpose, let $(v_1^+, v_1^-), (v_2^+, v_2^-) \in \mathcal{U}_\beta$.  
Then, for any $(i,j) \in \I^2$, we have
 \begin{multline*}
    F^\pm_{i, j}(v_1^\pm) - F^\pm_{i, j}(v_2^\pm) =  L \, \Delta t \left(\lambda^\pm_{i, j}[v_1] - \lambda^\pm_{i, j}[v_2] \right) \\
    +  \Delta t \left(\lambda^\pm_{i, j}[v_1]_+ - \lambda^\pm_{i, j}[v_2]_+\right) \, \theta_{i + 1/2,j}^{\pm,n,\per} -  \Delta t \left(\lambda^\pm_{i, j}[v_1]_- - \lambda^\pm_{i, j}[v_2]_-\right) \,\theta_{i-1/2,j}^{\pm,n,\per},
\end{multline*}
with $v_i = v_i^+-v_i^-$ for $i=1$, $2$.
Moreover, since \(\theta_{i+1/2,j}^{\pm,n,\per} + L \ge 0\), the sequence
\((\rho^{\pm, n, \per}_{i,j})_{i \in \I}\) satisfies condition~\eqref{cond-positive} for each \(j \in \I\). By Lemma~\ref{discrete-Poincare-Wirtinge}, it follows that
\begin{align}\label{hyp induction 2}
\max_{(i,j)\in \I^2} \left|\rho^{\pm, n, \per}_{i,j} - \langle \rho^{\pm, n ,\per}\rangle_j\right|
     &\le 2L.
\end{align}
Consequently, we have
\begin{align*}
\bigl|\theta_{i \pm 1/2,j}^{\pm,n,\per}\bigr|
\le 
\frac{
    \bigl|\rho^{\pm, n, \per}_{i\pm1,j} - \langle \rho^{\pm, n ,\per}\rangle_j\bigr|
    + \bigl|\rho^{\pm, n, \per}_{i,j} - \langle \rho^{\pm, n ,\per}\rangle_j\bigr|
}{\Delta x} 
\le \frac{4L}{\Delta x}.
\end{align*}
Together with Lemma~\ref{lemma bound sigmaM}, this yields
\begin{align*}
\max_{\pm}\, \bigl| F_{i,j}^\pm(v_1^\pm) - F_{i,j}^\pm(v_2^\pm) \bigr|
&\le 2 M^2 \Bigl( L \Delta t + 8 L  \frac{\Delta t}{\Delta x}  \Bigr)
   \max \Bigl(
        \|v_1^+ - v_2^+\|_{\ell^\infty(\I^2)}, \,
        \|v_1^- - v_2^-\|_{\ell^\infty(\I^2)}
   \Bigr) \notag\\
&= 18 M^2 L  \frac{\Delta t}{\Delta x}\,
   \max \Bigl(
        \|v_1^+ - v_2^+\|_{\ell^\infty(\I^2)}, \,
        \|v_1^- - v_2^-\|_{\ell^\infty(\I^2)}
   \Bigr).
\end{align*}
Hence, the CFL condition~\eqref{delta t delta x} ensures that \(F\) is a contraction mapping on \(\mathcal{U}_\beta\), and Banach’s fixed point theorem guarantees the existence and uniqueness of
\((\rho^{+, n+1, \per}, \rho^{-, n+1, \per}) \in \mathcal{U}_\beta\), 
which is the solution of the scheme.
\subsubsection{Nonnegativity of the discrete gradient and discrete $L^\infty$  estimate.}
Since \((\rho^{+, n+1, \per}, \rho^{-, n+1, \per}) \in \mathcal{U}_\beta\), we know in particular that 
the sequence \((\rho^{\pm, n+1, \per}_{i,j})_{i \in \I}\) satisfies condition~\eqref{cond-positive} for each \(j \in \I\). 
Hence,
\begin{align*}
\rho^{\pm,n+1}_{i+1,j}-\rho^{\pm,n+1}_{i,j} = \Delta x \left( \theta^{\pm,n+1,\per}_{i+1/2,j} + L\right) \geq 0, \quad \forall (i,j) \in \I^2.
\end{align*} 
By Lemma~\ref{discrete-Poincare-Wirtinge}, we then obtain the following estimate:
\begin{align*}
\|\rho^{\pm,n+1,\per}_{\cdot,j} - \langle \rho^{\pm,n+1,\per}\rangle_j\|_{\ell^\infty(\I)} 
\le 2L, \quad \forall j \in \I,
\end{align*}
as announced in~\eqref{bound rho mass rho}. Next, to show the discrete \(L^\infty\) bound~\eqref{bound of rho}, we rewrite the scheme~\eqref{solution_definitionF_plus} in the form
\begin{equation*}
    \rho^{\pm,n+1, \per}_{i,j}
    \displaystyle = \rho_{i,j}^{\pm,n , \per} 
    + \Delta t \, \lambda^\pm_{i,j}[\rho^{n+1, \per}]_+ \, \theta_{i+1/2,j}^{\pm,n,\per}  
    - \Delta t \, \lambda^\pm_{i,j}[\rho^{n+1, \per}]_- \, \theta_{i-1/2,j}^{\pm,n,\per}  
    + L \, \Delta t \, \lambda^\pm_{i,j}[\rho^{n+1, \per}],
\end{equation*}
or equivalently
\begin{multline*}
    \rho^{\pm,n+1, \per}_{i,j} \displaystyle = \left(1- \frac{\Delta t}{\Delta x} \bigl|\lambda^\pm_{i,j}[\rho^{n+1, \per}]\bigr| \right)\rho_{i,j}^{\pm,n , \per} 
    + \frac{\Delta t}{\Delta x} \lambda^\pm_{i,j}[\rho^{n+1, \per}]_+ \, \rho_{i+1,j}^{\pm,n , \per} 
    + \frac{\Delta t}{\Delta x} \lambda^\pm_{i,j}[\rho^{n+1, \per}]_- \, \rho_{i-1,j}^{\pm,n , \per} \\ + L \, \Delta t \, \lambda^\pm_{i,j}[\rho^{n+1, \per}],
\end{multline*}
where \(\rho^{n+1, \per} = \rho^{+, n+1, \per} - \rho^{-, n+1, \per}\). Thanks to Lemma~\ref{borne-sur-lambda}, the discrete velocity field \(\lambda^\pm_{i,j}[\rho^{n+1, \per}]\) satisfies the bound~\eqref{borne-vitesse}, so that the CFL condition~\eqref{delta t delta x} implies
\begin{align*}
|\rho^{\pm,n+1, \per}_{i,j}| 
&\le \|\rho^{\pm, n, \per}\|_{\ell^\infty(\I^2)} + L \, \Delta t \, |\lambda^\pm_{i,j}[\rho^{n+1, \per}]| \\
&\le \|\rho^{\pm, n, \per}\|_{\ell^\infty(\I^2)} + L \, \Delta t \, \bigl( 4 M^2 L + \| a\|_{L^{\infty}(0,T)} \bigr),
\end{align*}
and consequently
\begin{align*}
\max_{\pm} \|\rho^{\pm, n+1, \per}\|_{\ell^\infty(\I^2)} 
\le \max_{\pm} \|\rho^{\pm, n, \per}\|_{\ell^\infty(\I^2)} + L \, \Delta t \, \bigl( 4 M^2 L + \| a\|_{L^{\infty}(0,T)} \bigr).
\end{align*}
Iterating this estimate with respect to \(n\) yields the bound~\eqref{bound of rho}. 

In conclusion, starting from \((\rho^{+,0,\per}, \rho^{-,0,\per})\) satisfying~\ref{H2} 
and $(\Delta x, \Delta t)$ satisfying~\eqref{delta t delta x}--\eqref{delta t}, 
we conclude by induction that there exists a unique solution \((\rho^{+,n,\per}, \rho^{-,n,\per})\) to~\eqref{compacted numerical scheme} for any \(1 \le n \le N_T\). 
Moreover, these solutions satisfy the estimates~\eqref{bound of rho}--\eqref{bound rho mass rho} and
\(\theta^{\pm,n,\per}_{i,j} + L \ge 0\) for any \((i,j) \in \I^2\).
\subsection{Discrete gradient entropy estimate} \label{section gradient} 

Now, in order to complete the proof of Theorem~\ref{theorem-properties}, it remains to establish the discrete gradient entropy estimate~\eqref{gradient entropy estimate}. For this purpose, we will adapt at the discrete level the method of~\cite{cannone2010global}. For the convenience of the reader, we first recall the technical lemma~\cite[Lemma 3.1]{MoMo14}.
\begin{lemma}\label{convexity inequality of g} 
let $g : x \in \R_+ \mapsto x \ln(x)$, let $a_k$ and $\theta_k$ be two finite sequences of nonnegative real numbers such that $0< \sum_k a_k < \infty$, and let $\theta = \sum_k a_k \theta_k \geq 0$. Then, the following inequality holds:
\[
g(\theta) \leq \sum_k a_k \, g(\theta_k) + \theta \, \ln\left(\sum_k a_k\right).
\]
\end{lemma}
\noindent  Then, let $0\leq n \leq N_T-1$ be fixed, we first notice, thanks to equation~\eqref{compacted numerical scheme}, that it holds
\begin{multline*}
    \theta^{\pm,n+1,\per}_{i,j} + L = \left(1-\deltat \left( \lambda^\pm_{i+1,j}[\rho^{n+1,\per}]_- + \lambda^\pm_{i,j}[\rho^{n+1,\per}]_+ \right)\right) \left(\theta^{\pm,n,\per}_{i+1/2,j}+L\right)\\
     + \deltat \lambda^\pm_{i+1,j}[\rho^{n+1,\per}]_+ \left(\theta^{\pm,n,\per}_{i+3/2,j}+L\right) + \deltat \lambda^\pm_{i,j}[\rho^{n+1,\per}]_- \left(\theta^{\pm,n,\per}_{i-1/2,j}+L\right).
\end{multline*}
\noindent  Then, defining the constants
\[
a_1 := \frac{\Delta t}{\Delta x} \, \lambda^\pm_{i+1,j}\left[\rho^{n+1,\per}\right]_+, \quad a_2 := \frac{\Delta t}{\Delta x} \, \lambda^\pm_{i,j} \left[\rho^{n+1,\per}\right]_-,
\]
and
\[
a_3 := 1- \frac{\Delta t}{\Delta x}\left(\lambda^\pm_{i+1,j} \left[\rho^{n+1,\per }\right]_- +  \lambda^\pm_{i,j} \left[\rho^{n+1,\per}\right]_+\right),
\]
we obtain
\begin{align}\label{eq on theta}
    \theta^{\pm,n+1,\per}_{i+1/2,j} + L = a_1 \left(\theta^{\pm,n,\per}_{i+3/2,j} + L \right) + a_2 \left( \theta^{\pm,n,\per}_{i-1/2,j} + L\right) + a_3 \left(\theta^{\pm,n,\per}_{i+1/2,j} + L\right).
\end{align}
We also define the quantity $\mu^{\pm,n+1,\per}_{i+1/2,j} := 1-(a_1 + a_2 + a_3)$ and we notice that the following relation holds:
\begin{align}\label{defmu}
1 - \mu_{i +1/2,j}^{\pm,n+1,\per} = a_1+a_2+a_3 = 1 - \frac{\Delta t}{\Delta x}\left(\lambda^\pm_{i,j} \left[\rho^{n+1,\per}\right] -  \lambda^\pm_{i + 1,j} \left[\rho^{ n+1,\per}\right]\right).
\end{align}
The conditions~\eqref{delta t delta x} and bound \eqref{borne-vitesse} imply that $a_3 > 0$ and $1-\mu^{\pm,n+1,\per}_{i+1/2,j} \in (0,2)$. Hence, applying Lemma~\ref{convexity inequality of g} and the definition of the coefficients $a_1$, $a_2$ and $a_3$, we obtain
\begin{align*}
\sum_{(i,j)\in\I^2} (\Delta x)^2\, &g\left(\theta_{i +1/2,j}^{\pm, n+1,\per} + L\right)  \leq \\
&+  \sum_{(i,j) \in\I^2} (\Delta x)^2 \, g\left(\theta_{i + 1/2,j}^{\pm, n,\per} + L\right) + \sum_{(i,j)\in\I^2} (\Delta x)^2 \,\left(\theta_{i +1/2,j}^{\pm,n+1,\per} + L\right) \,\ln\left(1 - \mu_{i +1/2,j}^{\pm,n+1,\per}\right) \\
&+ \Delta x\sum_{(i,j)\in\I^2} \Delta t \, \left( \lambda^\pm_{i + 1,j} [\rho^{n+1,\per}]_+ \, g\left(\theta_{i +3/2,j}^{\pm, n,\per} + L\right) - \lambda^\pm_{i,j}[\rho^{n+1,\per}]_+ \, g\left(\theta_{i +1/2,j}^{\pm, n,\per } + L\right)\right)\\ 
& + \Delta x\sum_{(i,j) \in \I^2} \Delta t \, \left(\lambda^\pm_{i,j}[\rho^{ n+1,\per}]_- \, g\left(\theta_{i -1/2,j}^{\pm, n,\per } + L\right) - \lambda^\pm_{i + 1,j}[\rho^{ n+1,\per}]_- \, g\left(\theta_{i +1/2,j}^{\pm, n,\per} + L\right)\right). 
\end{align*}
Now, thanks to the periodicity of the vectors involved in the two last telescopic sums, we get
\begin{multline}\label{aux1}
\sum_{\pm}\sum_{(i,j)\in\I^2} (\Delta x)^2\, g\left(\theta_{i +1/2,j}^{\pm, n+1,\per} + L\right) 
\leq  \sum_{\pm}\sum_{(i,j) \in\I^2} (\Delta x)^2 \, g\left(\theta_{i + 1/2,j}^{\pm, n,\per} + L\right)  \\ - \sum_{\pm}\sum_{(i,j)\in\I^2} (\Delta x)^2 \, \left(\theta_{i +1/2,j}^{\pm,n+1,\per} + L\right)  \mu_{i +1/2,j}^{\pm,n+1,\per},
\end{multline}
where we have used the inequality $\ln(1-\mu)\leq -\mu$ for all $\mu < 1$. Then, denoting $J_1$ the last term in the right hand side and recalling definition~\eqref{diff rho plus moins} of $\rho^{n+1,\per}$, we observe that
\begin{align*}
J_1 &= -\Delta t \sum_{(i,j)\in\I^2} \Delta x \, \left(\theta_{i +1/2,j}^{+,n+1,\per} + L \right) \, \left(\lambda^+_{i,j}[\rho^{n+1,\per}]-\lambda^+_{i+1,j}[\rho^{n+1,\per}]\right)\\
&-\Delta t \sum_{(i,j)\in\I^2} \Delta x \, \left(\theta_{i +1/2,j}^{-,n+1,\per} + L \right) \, \left(\lambda^-_{i,j}[\rho^{n+1,\per}]-\lambda^-_{i+1,j}[\rho^{n+1,\per}]\right)\\
&= \Delta t \sum_{(i,j)\in\I^2} (\Delta x) \, \left(\theta_{i +1/2,j}^{+,n+1,\per} + L \right) \sum_{(\ell,r)\in\I^2} (\Delta x)^2\, \sigma_{M,\ell,r}^\K \left(\rho^{n+1,\per}_{i-\ell,j-r}- \rho^{n+1,\per}_{i+1-\ell,j-r}\right)\\
&-\Delta t \sum_{(i,j)\in\I^2} (\Delta x) \, \left(\theta_{i +1/2,j}^{-,n+1,\per} + L \right) \sum_{(\ell,r)\in\I^2} (\Delta x)^2\, \sigma_{M,\ell,r}^\K \left(\rho^{n+1,\per}_{i-\ell,j-r}- \rho^{n+1,\per}_{i+1-\ell,j-r}\right).
\end{align*}
In order to simplify the notation, for any $0\leq n \leq N_T$, we introduce the vectors
\begin{align*}
\theta^{n,\per}_{x_1} = \left(\theta^{n,\per}_{i+1/2,j}\right)_{(i,j)\in\I^2} = \left(\theta^{+,n,\per}_{i+1/2,j}-\theta^{-,n,\per}_{i+1/2,j}\right)_{(i,j)\in\I^2},
\end{align*}
and we get
\begin{align*}
J_1 = -\Delta t \sum_{(i,j)\in\I^2} (\Delta x)^2 \, \theta_{i +1/2,j}^{n+1,\per} \sum_{(\ell,r)\in\I^2} (\Delta x)^2 \,\sigma_{M,\ell,r}^\K \, \theta^{n+1,\per}_{i+1/2-\ell,j-r}.
\end{align*}
Hence, as in Subsection~\ref{section-periodic}, we intend to show that $J_1$ is nonpositive. For this, we will use the Discrete Fourier Transform (DFT). Then, for any vector $v=(v_{i,j})_{(i,j)\in\I^2}$, we also introduce the discrete counterpart of the (periodic) convolution product, still denoted $\ast$, between $v$ and $\theta^{n,\per}_{x_1}$, for any $0\leq n \leq N_T$, as
\begin{align*}
    \left(v \ast \theta^{n,\per}_{x_1}\right)_{i,j} := \sum_{(\ell,r) \in \I^2} (\Delta x)^2 \, v_{\ell,r} \, \theta^{n,\per}_{i+1/2-\ell,j-r}, \quad \forall (i,j) \in \I^2.
\end{align*}
Then, defining $\bar{\sigma}_M^\K = \left(\sigma_{M,i,j}^\K\right)_{(i,j)\in\I^2}$, we can rewrite the term $J_1$ as
\begin{align*}
    J_1 = -\Delta t \sum_{(i,j)\in\I^2} (\Delta x)^2 \, \theta^{n+1,\per}_{i+1/2,j} \, \left( \bar{\sigma}_M^\K\ast \theta^{n+1,\per}_{x_1}\right)_{i,j}. 
\end{align*}
Now, for any $v = (v_{i,j})_{(i,j) \in \I^2}$, we denote by $c^d(v) = (c^d_{(m_1,m_2)}(v))_{(m_1,m_2)\in\I^2}$ the coefficients of the DFT of $v$ with
\begin{align}\label{c-DFT}
    c^d_{(m_1,m_2)}(v) := \frac{1}{N^2} \sum_{(n_1,n_2) \in \I^2} v_{n_1,n_2} \, \exp\left(-2i\pi\left(m_1 \frac{n_1}{N} + m_2 \frac{n_2}{N}\right)\right), \quad \forall (m_1,m_2) \in \I^2.
\end{align}
Applying the discrete version of Parseval's equality and using the property of the DFT (in $2$d) with respect to $\ast$, we have
\begin{align*}
    J_1 = - \, \Delta t  \, \sum_{(m_1,m_2) \in \I^2}  \, c^d_{(m_1,m_2)}\left(\bar{\sigma}_M^\K\right) \, \left|c^d_{(m_1,m_2)}\left(\theta^{n+1,\per}_{x_1}\right)\right|^2.
\end{align*}
Therefore, in order to show that $J_1 \leq 0$, it is sufficient to prove that the coefficients of the DFT of the vector $\bar{\sigma}_M^\K$ are real nonpositive numbers. Then, let $(m_1,m_2)\in\I^2$ be fixed, we have
\begin{align*}
         &c^d_{(m_1,m_2)}(\bar{\sigma}_{M}^\K) =  \frac{1}{N^2}\sum_{(n_1,n_2)\in\I^2}  \sigma_{M, n_1, n_2}^\K\exp\left(-2i \pi\left(m_1\frac{n_1}{N} + m_2 \frac{n_2}{N}\right)\right)\\
         &= \frac{1}{(N \, M)^2} \sum_{(n_1,n_2)\in\I^2}\left(\sum_{k_1, k_2 = 0}^{M- 1} \sum_{\substack{|p_1|\leq k_1 \\|p_2|\leq k_2}} c_{(p_1, p_2)}(\K) e^{2 i \pi (p_1 n_1+p_2 n_2)/N}\right) e^{-2i \pi (m_1 n_1 + m_2 n_2)/N}\\
         &= \frac{1}{(N \, M)^2} \sum_{k_1, k_2 = 0}^{M- 1} \sum_{\substack{|p_1|\leq k_1 \\|p_2|\leq k_2}} c_{(p_1, p_2)}(\K) \,\left[ \sum_{n_1\in\I} \left(e^{2i\pi(p_1-m_1)/N}\right)^{n_1}\right] \left[ \sum_{n_2\in\I} \left(e^{2i\pi(p_2-m_2)/N}\right)^{n_2}\right].
 \end{align*}
The two last sums are either equal to $N$ or $0$, depending on the values of $p_1$ and $p_2$. Therefore, the coefficient $c^d_{(m_1,m_2)}(\bar{\sigma}_M)$ reduces to a finite sum of Fourier coefficients of $\K$ (see relation~\eqref{coeff discret sigmaM} below). However, according to~\eqref{coeff Fourier K}, these coefficients are nonnegative. Hence, we deduce that
\begin{align}\label{J1}
    J_1 = - \Delta t \, \mathcal{D}[\rho^{n+1,\per}] \leq 0
\end{align}
where we recall definition~\eqref{discrete dissipation} of the discrete dissipation term $\mathcal{D}[\rho^{n+1,\per}]$. Thus, collecting~\eqref{aux1}--\eqref{J1}, we end up with 
\begin{align}\label{gradentdis.onesteptime}
    \sum_{\pm}\sum_{(i,j)\in\I^2} (\Delta x)^2 \, g\left(\theta_{i +1/2,j}^{\pm, n+1,\per} + L\right) + \Delta t \, \mathcal{D}[\rho^{n+1,\per}] 
\leq  \sum_{\pm}\sum_{(i,j) \in\I^2} (\Delta x)^2 \, g\left(\theta_{i + 1/2,j}^{\pm,n,\per} + L\right).
\end{align}
It remains to sum over $n$ and to notice that $g(x) \leq x \ln(x+e)=f(x)$ for any $x\geq 0$, to obtain 
\begin{multline*}
    \sum_{\pm}\sum_{(i,j)\in\I^2} (\Delta x)^2 \, g\left(\theta_{i +1/2,j}^{\pm, n+1,\per} + L\right)  + \sum_{k=0}^n \Delta t \, \mathcal{D}[\rho^{k+1,\per}] \\
\leq  \sum_{\pm}\sum_{(i,j) \in\I^2} (\Delta x)^2 \, f\left(\theta_{i + 1/2,j}^{\pm,0,\per} + L\right), \quad \forall n \in \{0,\ldots,N_T-1\}.
\end{multline*}
Moreover, classical computations show that
\begin{align*}
f(\theta) = f(\theta) \, \mathbf{1}_{\{\theta <e\}} + f(\theta) \, \mathbf{1}_{\{\theta\geq e\}} \leq f(e)+\theta \ln(2) + g(\theta), \quad \forall \theta \geq 0.
\end{align*}
Therefore, thanks to the nonnegativity of the discrete dissipation term, we conclude that for any $n \in \{0,\ldots,N_T-1\}$ we have
\begin{align*}
    \sum_{\pm}\sum_{(i,j)\in\I^2} (\Delta x)^2 \, f\left(\theta_{i +1/2,j}^{\pm, n+1,\per} + L\right) &\leq 2f(e) + \ln(2) \sum_{\pm}\sum_{(i,j)\in\I^2} (\Delta x)^2 \left(\theta^{\pm,n+1,\per}_{i+1/2,j}+L\right)\\
 &+ \sum_{\pm}\sum_{(i,j) \in\I^2} (\Delta x)^2 \, f\left(\theta_{i + 1/2,j}^{\pm,0,\per} + L\right) - \sum_{k=0}^n \Delta t \, \mathcal{D}[\rho^{k+1,\per}].\\
\end{align*}
Finally, we notice that
\begin{align*}
\sum_{\pm}\sum_{(i,j)\in\I^2} (\Delta x)^2 \left(\theta^{\pm,n+1,\per}_{i+1/2,j}+L\right) = 2L,
\end{align*}
which concludes the proof of Theorem~\ref{theorem-properties}.

\section{Uniform estimates with respect to the meshes}\label{Uniform-estimates}

In this section, we establish some uniform with respect to $\eps = \max(\Delta t, \Delta x)$ estimates. As in the previous section, in order to lighten the notation, we will neglect $M$ in the indices of the vectors $\rho^{\pm,n,\per}_M$ for any $0\leq n\leq N_T$. However, let us emphasize that some estimates will depend on $M$ while others will be independent of $M$. For the convenience of the reader, we split this section in two subsections making clear the dependency on $M$ in those estimates.

\subsection{Uniform estimates with respect to \texorpdfstring{$M$}{M}}\label{L2 bound of rho}

\begin{proposition}\label{prop mean L2}
Let the assumptions of Theorem~\ref{theorem-properties} hold. Then, for any $n \in \{0,\ldots,N_T-1\}$, we have
\begin{multline*}
\left(\sum_{j \in \I} \Delta x \, \left|\langle \rho^{\pm,n+1,\per}\rangle_j\right|^2 \right)^{1/2} \leq \|\rho^{\pm,\per}_0\|_{L^2(\T)}+L \,T \| a\|_{L^{\infty}(0,T)}\\ 
+ 8L \,\,\left(\sum_{\pm} \int_{\T} f\left(\pa_{x_1} \rho^{\pm,\per}_{M,0} + L\right) \, dx + 2(f(e) + L \ln(2))\right)^{1/2},
\end{multline*}
where we recall definition~\eqref{mean_expression} of $\langle \cdot \rangle_j$.
\end{proposition}


\begin{proof}
For this, we adapt at the discrete level the proof of~\cite[Lemma 5.6]{cannone2010global}. Indeed, multiplying~\eqref{compacted numerical scheme} by $\Delta x$, summing over $i \in \I$ and using a discrete integration by parts, we get for any fixed $j \in \I$
\begin{align}\label{aux2}
\meannp_j=   \meann_j + \Delta t \left(J^\pm_{2,j}+J^\pm_{3,j}+J^\pm_{4,j}\right), 
\end{align}      
with
\begin{align*}  
J^\pm_{2,j} &= \sum_{i \in \I}  \, \left(\lambda_{i + 1, j}^\pm[\rho^{n + 1,\per}]_- - \left|\lambda_{i, j}^\pm[\rho^{n + 1,\per}]\right| + \lambda_{i-1,j}^\pm[\rho^{n+1,\per}]_+\right)\left(\rho_{i,j}^{\pm,n, \per}- \meann_j\right) \\
J^\pm_{3,j} &= \sum_{i \in \I}  \, \left(\lambda_{i + 1, j}^\pm[\rho^{n + 1,\per}]_- - \left|\lambda_{i, j}^\pm[\rho^{n + 1,\per}]\right| + \lambda_{i-1,j}^\pm[\rho^{n+1,\per}]_+\right)  \meann_j \\
J^\pm_{4,j} &= L \sum_{i \in \I} \Delta x \, \lambda_{i, j}^\pm[\rho^{n + 1,\per}].
\end{align*}
Let us first study the term $J^\pm_{2,j}$. For this, we apply the estimate~\eqref{bound rho mass rho} and we obtain
\begin{align*}
|J^\pm_{2,j}| &\leq 2L \sum_{i \in \I}  \left( \left|\lambda_{i+1,j}^\pm[\rho^{n+1,\per}]_- - \lambda_{i,j}^\pm[\rho^{n+1,\per}]_-\right|+\left|\lambda_{i,j}^\pm[\rho^{n+1,\per}]_+ -\lambda_{i-1,j}^\pm[\rho^{n+1,\per}]_+\right|\right)\\
&\leq 4 L \, \sum_{i \in \I}  \left|\lambda_{i+1,j}^\pm[\rho^{n+1,\per}] - \lambda_{i,j}^\pm[\rho^{n+1,\per}]\right|.
\end{align*}
Now, for $J^\pm_{3,j}$ thanks to the periodicity of the vectors involved in this term, we have
\begin{align*}
   J^\pm_{3,j} &= \meann_j  \sum_{i \in \I}  \, \left(\lambda_{i, j}^\pm[\rho^{n + 1,\per}]_- - \left|\lambda_{i, j}^\pm[\rho^{n + 1,\per}]\right| + \lambda_{i,j}[\rho^{n+1,\per}]_+\right)\\ 
   &= \meann_j  \sum_{i \in \I}  \, \left(\left|\lambda_{i, j}^\pm[\rho^{n + 1,\per}]\right| - \left|\lambda_{i, j}^\pm[\rho^{n + 1,\per}]\right| \right) = 0.
\end{align*}
Similarly, for $J^\pm_{4,j}$, using the definition~\eqref{lambda} of $\lambda_{i, j}^\pm[\rho^{n + 1,\per}]$, we can write
\begin{align*}
    J^\pm_{4,j}&= \mp L\; a(t_{n+1})  \mp L \sum_{i \in \I} \Delta x \, \sum_{(\ell, r)\in\I^2} (\Delta x)^2 \, \sigma_{M, i-\ell,j- r}^\K  \, \rho^{n + 1, \per}_{\ell, r} \\
    &=  \mp L\; a(t_{n+1}) \mp L  (\Delta x )^3\, \sum_{(\ell, r)\in\I^2} \, \left(\sum_{i \in \I} \sigma_{M, i-\ell,j- r}^\K  \right) \rho^{n + 1, \per}_{\ell, r}.
\end{align*}
Now, we observe, as in the proof of Lemma \ref{borne-sur-lambda},  that
\begin{align*}
\sum_{i\in\I} \sigma_{M, i-\ell,j- r} &= \frac{1}{M^2} \sum_{i \in \I} \left(\sum_{k_1,k_2=0}^{M-1} \sum_{\substack{|p_1| \leq k_1 \\ |p_2| \leq k_2}} c_{(p_1,p_2)}(\K) e^{2i\pi(p_1(i-\ell)+p_2(j-r))/N} \right)\\
&= \frac{1}{M^2} \sum_{k_1,k_2=0}^{M-1} \sum_{\substack{|p_1| \leq k_1 \\ |p_2| \leq k_2}} c_{(p_1,p_2)}(\K) \left( \sum_{i \in \I} \left[e^{2i\pi p_1/N}\right]^{i} \right)  e^{2i\pi( p_2(j-r)-p_1 \ell)/N}.
\end{align*}
Obviously, the geometric sum appearing in this last expression is equal to $0$ except when $p_1=0$. However, in this former case, according to formula~\eqref{coeff Fourier K}, we have $c_{(0,p_2)}(\K) = 0$. Therefore, we deduce that
\begin{align*}
J^\pm_{4,j} =  \mp L\; a(t_{n+1}).
\end{align*}
Hence, we deduce from relation~\eqref{aux2} and Minkowski's inequality the following estimate
\begin{multline}\label{aux3}
\left(\sum_{j \in \I} \Delta x \, \left|\meannp_j\right|^2\right)^{1/2} \leq \left(\sum_{j \in \I} \Delta x \, \left|\meann_j\right|^2\right)^{1/2} + L\Delta t \| a\|_{L^{\infty}(0,T)}\\ + 4 L \Delta t \left( \sum_{j \in \I} \Delta x \, \left(\sum_{i\in\I}  \left|\lambda_{i+1,j}^\pm[\rho^{n+1,\per}] - \lambda_{i,j}^\pm[\rho^{n+1,\per}]\right|  \right)^2 \right)^{1/2}.
\end{multline}
Let us now focus on the last term in the right hand side of the above inequality. Thanks to Cauchy-Schwarz inequality we obtain
\begin{align*}
 \sum_{j \in \I} \Delta x \, \left(\sum_{i\in\I} \left|\lambda_{i+1,j}^\pm[\rho^{n+1,\per}] - \lambda_{i,j}^\pm[\rho^{n+1,\per}]\right|  \right)^2 
  \leq \sum_{(i,j)\in\I^2}\, \left|\lambda^\pm_{i+1,j}[\rho^{n+1,\per}] - \lambda^\pm_{i,j}[\rho^{n+1,\per}] \right|^2.
\end{align*}
Reproducing similar computations as in Subsection~\ref{section gradient}, we have
\begin{align*}
\sum_{(i,j)\in\I^2} \, \Big|\lambda^\pm_{i+1,j}[\rho^{n+1,\per}] - \lambda^\pm_{i,j}[\rho^{n+1,\per}] \Big|^2 &= \sum_{(i,j)\in\I^2} (\Delta x)^2 \, \left|\left(\bar{\sigma}_M^\K\ast \theta^{n+1,\per}_{x_1}\right)_{i,j} \right|^2\\
&= \sum_{m \in \I^2} \left| c^d_m(\bar{\sigma}_M^\K) \, c^d_m \left(\theta^{n+1,\per}_{x_1}\right) \right|^2.
\end{align*}
where we recall that $\bar{\sigma}_M^\K = (\sigma_{M,i,j}^\K)_{(i,j)\in\I^2}$. Now, as in the proof of the discrete gradient entropy estimate~\eqref{gradient entropy estimate}, for any given $m=(m_1,m_2) \in \I^2$, we write
\begin{multline*}
         c^d_{(m_1,m_2)}(\bar{\sigma}_{M}^\K) 
         = \frac{1}{(N \, M)^2} \sum_{k_1, k_2 = 0}^{M- 1} \sum_{\substack{|p_1|\leq k_1 \\|p_2|\leq k_2}} c_{(p_1, p_2)}(\K) \,\left[ \sum_{n_1\in\I} \left(e^{2i\pi(p_1-m_1)/N}\right)^{n_1}\right] \left[ \sum_{n_2\in\I} \left(e^{2i\pi(p_2-m_2)/N}\right)^{n_2}\right].
 \end{multline*} 
It remains to notice that, depending on the values of $p_1$ and $p_2$, the product of the two geometric sums is equal to $N^2$ or $0$. In particular, meticulous but rather straightforward computations lead to
\begin{multline}\label{coeff discret sigmaM} 
c^d_{(m_1,m_2)}(\bar{\sigma}_{M}^\K) = \frac{1}{M^2} \sum_{k_1, k_2 = 0}^{M- 1} \left(c_{(m_1, m_2)}(\K) + c_{(m_1-N, m_2)}(\K) + c_{(m_1, m_2-N)}(\K) + c_{(m_1-N, m_2-N)}(\K) \right)\\
= \left(c_{(m_1, m_2)}(\K) + c_{(m_1-N, m_2)}(\K) + c_{(m_1, m_2-N)}(\K) + c_{(m_1-N, m_2-N)}(\K) \right) \leq 4,
\end{multline}
where we have used for the last inequality~\eqref{coeff Fourier K}. Therefore, recalling definition~\eqref{discrete dissipation} of the discrete dissipation functional, we obtain
\begin{multline}\label{aux4}
\sum_{j \in \I} \Delta x \, \left(\sum_{i\in\I} \Delta x \left|\lambda_{i+1,j}^\pm[\rho^{n+1,\per}] - \lambda_{i,j}^\pm[\rho^{n+1,\per}]\right|  \right)^2 \\
 \leq 4  \sum_{m \in \I^2}  c^d_m(\bar{\sigma}_M) \left| c^d_m\left(\theta^{n+1,\per}_{x_1}\right) \right|^2 = 4 \mathcal{D}[\rho^{n+1,\per}] .
\end{multline}
Now, gathering~\eqref{aux3} and~\eqref{aux4}, summing over $n$ and applying the Cauchy-Schwarz inequality yield to
\begin{multline*}
\left(\sum_{j \in \I} \Delta x \, \left|\meannp_j\right|^2\right)^{1/2} \leq \left(\sum_{j \in \I} \Delta x \, \left|\langle \rho^{\pm,0,\per}\rangle_j\right|^2\right)^{1/2}
+ L\, N_T \, \Delta t \| a\|_{L^{\infty}(0,T)} \\
 + 8L \, \left(\sum_{k=0}^n \Delta t \, \mathcal{D}[\rho^{k+1,\per}]\right)^{1/2}.
\end{multline*}
Now, in order to bound from above the last term in the r.h.s.~we apply the discrete gradient entropy estimate~\eqref{gradient entropy estimate}, and to obtain a uniform estimate w.r.t.~the mesh we apply the following inequality 
\begin{multline}\label{bound entropy init}
\sum_{(i,j)\in\I^2} (\Delta x)^2 f\left(\theta^{\pm,0,\per}_{i,j}+L\right) \leq \sum_{j \in \I} \Delta x \int_{\mathbb{T}} f\left(\pa_{x_1} \rho^{\pm,\per}_{M,0}(x_1,x_j)+L\right) \, dx_1\\ 
\leq \int_{\T} f\left(\pa_{x_1} \rho^{\pm,\per}_{M,0}(x_1,x_2) +L\right) \, dx_1\,dx_2,
\end{multline}
which is a consequence of the convexity and growth of the function $f$. Finally, thanks to definitions~\eqref{IC.reg.per} of the initial condition and Hölder inequality we have
\begin{align*}
\sum_{j \in \I} \Delta x \, \left|\langle \rho^{\pm,0,\per}\rangle_j\right|^2 
&\leq \sum_{(i,j) \in \I^2} (\Delta x)^2 \, \left|\rho^{\pm,0,\per}_{M,i,j}\right|^2 
= \sum_{(i,j) \in \I^2} (\Delta x)^2  \left(\int_{\T}F_M^2(x_i-y_1, x_j-y_2)\rho^{\pm,\per}_0(y) dy \right)^2  \\
& \le \sum_{(i,j) \in \I^2} (\Delta x)^2  \left(\int_{\T}F_M^2(x_i-y_1, x_j-y_2) dy\right) 
 \left(\int_{\T}F_M^2(x_i-y_1, x_j-y_2)|\rho^{\pm,\per}_0(y)|^2 dy \right)\\ 
& \le \|F_M^2\|_{L^1(\T)}  \left(\int_{\T} \left(\sum_{(i,j) \in \I^2} (\Delta x)^2 F_M^2(x_i-y_1, x_j-y_2)\right)|\rho^{\pm,\per}_0(y)|^2 dy \right)\\
&\leq \|F_M^2\|_{L^1(\T)} \, \|\rho^{\pm,\per}_0\|^2_{L^2(\T)} \leq  \|\rho^{\pm,\per}_0\|^2_{L^2(\T)}, 
\end{align*}
where we have used the fact that $F_M^2(x) \ge 0$ and $\sum_{(i,j) \in \I^2} (\Delta x)^2 F_M^2(x_i-y_1, x_j-y_2)=1$ when $N \ge M$.
\end{proof}

\begin{corollary}\label{corollary L2 bound}
Let the assumptions of Theorem~\ref{theorem-properties} hold. Then, for any $n \in \{0,\ldots,N_T-1\}$, it holds
\begin{multline}\label{estimation_L2_uni}
\sum_{(i,j)\in\I^2} (\Delta x)^2 \, \left|\rho^{\pm,n+1,\per}_{i,j}\right|^2 \leq 8L^2 + 8\|\rho^{\pm,\per}_0\|^2_{L^2(\T)} + 8L^2 \,T^2 \, \| a\|_{L^{\infty}(0,T)}^2 \\ 
+ 128 L^2 \left(\,\sum_{\pm} \, \int_{\T} f\left(\pa_{x_1} \rho^{\pm,\per}_{M,0} + L\right)\, dx + 2(f(e) + L \ln(2))\right). 
\end{multline}
\end{corollary}

\begin{proof}
Let $n \in \{0,\ldots,N_T-1\}$ be fixed. We first apply the elementary inequality $(a + b)^2 \leq 2(a^2 + b^2)$ for any $a$ and $b \in \R$, and we obtain
\begin{align*}
    \sum_{(i,j)\in\I^2} (\Delta x)^2 \, &\left|\rho^{\pm,n+1,\per}_{i, j}\right|^2\\
 &\leq 2\sum_{(i,j)\in\I^2} (\Delta x)^2 \left(\rho^{\pm,n+1,\per}_{i, j} - \langle \rho^{\pm,n+1,\per}\rangle_j\right)^2 + 2 \sum_{(i,j)\in\I^2} (\Delta x)^2 \left|\langle \rho^{\pm,n+1,\per}\rangle_j\right|^2 \\
&\leq 8L^2 + 2\sum_{j\in\I} (\Delta x) \left|\langle \rho^{\pm,n+1,\per}\rangle_j\right|^2.
\end{align*}
It remains to apply the estimate established in Proposition~\ref{prop mean L2} to conclude the proof.
\end{proof}


\subsection{Non uniform estimates with respect to \texorpdfstring{$M$}{M}}

In this section we prove that the functions $\rho^{\pm,\per,\eps}$ given by~\eqref{continuous rho} are uniformly bounded in $W^{1,\infty}((0,T)\times\T)$.  As previously said, these estimates are uniform in $\eps = \max(\Delta t,\Delta x)$ but non uniform with respect to $M$. We first introduce the discrete derivative with respect to $x_2$ of $(\rho^{\pm,n,\per})_{0\leq n \leq N_T}$ as
\begin{equation}\label{theta_j}
    \theta^{\pm,  n, \per}_{i, j+1/2} = \frac{\rho^{\pm,n, per}_{i, j+1} - \rho^{\pm,n, per}_{i, j}}{\Delta x}, \quad \forall (i,j) \in \I^2,
\end{equation}
and the associated vector $\theta^{\pm,n,\per}_{x_2} = (\theta^{\pm,n,\per}_{i,j+1/2})_{(i,j)\in\I^2}$ (and similarly for the vector $\theta^{\pm,n,\per}_{x_1}$). We first prove the following technical

\begin{lemma}\label{derivative sigma}
Let $M\geq 1$. Then, it holds
\begin{align}\label{derivative sigmaM}
\|\partial_{x_1} \sigma_M^\K\|_{L^\infty(\T)} = \|\partial_{x_2} \sigma_M^\K\|_{L^\infty(\T)} \leq 
\frac{2\pi}{3}  M (M-1)(M+1).
\end{align}
\end{lemma}

\begin{proof}
We derive with respect to the first component the function $\sigma_M^\K$ defined by~\eqref{sigma_def}, we have
\begin{align*}
\pa_{x_1} \sigma_M^\K(y) = \frac{2i\pi }{M^2} \sum_{n_1, \, n_2 = 0}^{M-1} \sum_{|m_1| \leq n_1}\sum_{|m_2|\leq n_2} m_1 \, c_{(m_1,m_2)}(\K)\, e^{2i \pi(m_1 y_1+m_2 y_2)}, \quad \forall y=(y_1,y_2) \in \T.
\end{align*}
Thanks to~\eqref{coeff Fourier K}, we obtain, for any $y \in \T$,
\begin{align*}
|\pa_{x_1} \sigma_M^\K(y)| \leq \frac{2\pi}{M^2} \sum_{n_1, \, n_2 = 0}^{M-1} (2n_2 +1) \left( \sum_{|m_1| \leq n_1} |m_1| \right)
&= \frac{2\pi}{M^2} \left(\sum_{\, n_2 = 0}^{M-1} (2n_2 +1)\right) 
\left(\sum_{\, n_1 = 0}^{M-1} n_1(n_1+1)\right) \\
&= \frac{2\pi \, M (M-1)(M+1)}{3}. 
\end{align*}
We argue similarly for the $L^\infty$ norm of its derivative with respect to the second component.
\end{proof}

\begin{proposition}\label{prop lip}
Let the assumptions of Theorem~\ref{theorem-properties} hold. Then, for any $n \in \{1,\ldots,N_T\}$, there exists a positive constant $C_M$ which depends on $M$ and $L$ such that
\begin{align}\label{unf_theta_i}
       \|\theta^{\pm ,n, \per}_{x_1} +  L \|_{\ell^{\infty}(\I^2)}\leq e^{ T C_M} \, \left(\|\pa_{x_1}\rho^{\pm,\per}_{M,0}\|_{L^\infty(\T)} +  L \right),
\end{align}
and 
\begin{align}\label{unf_theta_j}
    \|\theta_{x_2}^{\pm, n, \per}\|_{\ell^{\infty}(\I^2)}  \le \|\pa_{x_2} \rho^{\pm,\per}_{M,0}\|_{L^\infty(\T)} + 2 T \, C_M\, e^{T \, C_M} \left(\|\pa_{x_1} \rho^{\pm,\per}_{M,0}\|_{L^\infty(\T)}+L\right).
\end{align}
\end{proposition}

\begin{proof} Let $n \in\{0,\ldots, N_T-1\}$ be fixed. We first prove estimate~\eqref{unf_theta_i}. To this end, we use the relation~\eqref{eq on theta}, the CFL condition~\eqref{delta t delta x} and the nonnegativity of $\theta^{\pm,n,\per}_{i+1/2,j}+L$ for any $(i,j)\in\I^2$ to get 
\begin{align*}
  \|\theta^{\pm,n + 1, \per}_{x_1} +  L\|_{\ell^{\infty}(\I^2)} \leq \left(1 + \deltat \left(\lambda^\pm_{i+1, j}[\rho^{n + 1,\per}]- \lambda^\pm_{i, j}[\rho^{n + 1,\per}]\right)\right) \, \|\theta^{\pm, n, \per}_{x_1} +  L\|_{\ell^{\infty}(\I^2)}.
\end{align*}
However, if we denote $\bar{\rho}^{n + 1,\per}_{i,j}= {\rho}^{n + 1,\per}_{i,j}- \langle \rho^{n+1,\per}\rangle_j$, then by \eqref{nulle-vitesse}  we can verify that
$$\lambda^\pm_{i, j}[\rho^{n + 1,\per} ] = \lambda^\pm_{i, j}[\bar{\rho}^{n + 1,\per}]
\quad \forall (i,j)\in\I^2.$$
Therefore
\begin{align*}
    1 + \deltat \left(\lambda^\pm_{i+1, j}[\rho^{n + 1,\per} ]- \lambda^\pm_{i , j}[\rho^{n + 1,\per}]\right)
=   1 + \deltat \left(\lambda^\pm_{i+1, j}[\bar{\rho}^{n + 1,\per}]- \lambda^\pm_{i , j}[
\bar{\rho}^{n + 1,\per}]\right)\\
\leq 1 + \Delta t \sum_{(\ell,r)\in\I^2} (\Delta x)^2 \, \frac{\left|\sigma_{M, i+1-\ell, r} - \sigma_{M, i-\ell, r}\right|}{\Delta x} \, \left|\rho^{n + 1,\per}_{ \ell, r} -\langle \rho^{n+1,\per}\rangle_r \right|.
\end{align*}
Then, thanks to Lemma~\ref{derivative sigma} and bound and \eqref{bound rho mass rho}, we deduce 
that
\begin{align*}
    1 + \deltat (\lambda^\pm_{i+1, j}[\rho^{n + 1,\per}]- \lambda^\pm_{i , j}[\rho^{n + 1,\per}]) 
    &\leq 1 + \Delta t \frac{8\pi L}{3}  M (M-1)(M+1).
\end{align*}
This implies the existence of a constant $C_M > 0$  only depending on $M$ and $L$,  such that
\begin{align*}
         \|\theta^{\pm,n + 1, \per}_{x_1} +  L\|_{\ell^{\infty}(\I^2)} &\leq \left( 1 + \Delta t C_M \right) \|\theta^{\pm,n , \per}_{x_1} +  L\|_{\ell^{\infty}(\I^2)}\\
         &\leq \left( 1 + \Delta t C_M \right)^{n+1} \|\theta^{\pm,0, \per}_{x_1} +  L\|_{\ell^{\infty}(\I^2)},
\end{align*}  
which yields estimate~\eqref{unf_theta_i}. Let us now prove estimate~\eqref{unf_theta_j}. For this purpose, we use equation~\eqref{compacted numerical scheme} to derive the following relation
\begin{align*}
\theta^{\pm,n+1,\per}_{i,j+1/2} = \theta^{\pm,n,\per}_{i,j+1/2} &+ \deltat \left(\lambda^\pm_{i,j+1}[\rho^{n+1,\per}]_+-\lambda^\pm_{i,j}[\rho^{n+1,\per}]_+ \right) \left(\theta^{\pm,n,\per}_{i+1/2,j+1} + L \right)\\
&+ \deltat \left(\lambda^\pm_{i,j}[\rho^{n+1,\per}]_- - \lambda^\pm_{i,j+1}[\rho^{n+1,\per}]_-\right) \left(\theta^{\pm,n,\per}_{i-1/2,j+1}+L\right) \\
&+ \deltat \, \lambda^\pm_{i,j}[\rho^{n+1,\per}]_+ \, \left(\theta^{\pm,n,\per}_{i+1/2,j+1} - \theta^{\pm,n,\per}_{i+1/2,j}\right)\\
&+ \deltat \, \lambda^\pm_{i,j}[\rho^{n+1,\per}]_- \, \left(\theta^{\pm,n,\per}_{i-1/2,j} - \theta^{\pm,n,\per}_{i-1/2,j+1}\right).
\end{align*}
Then, thanks to definition~\eqref{theta_j} we notice that for any $(i,j) \in \I^2$ it holds
\begin{align*}
\theta^{\pm,n,\per}_{i+1/2,j+1} - \theta^{\pm,n,\per}_{i+1/2,j} = \theta^{\pm,n,\per}_{i+1,j+1/2} - \theta^{\pm,n,\per}_{i,j+1/2},
\end{align*}
and
\begin{align*}
\theta^{\pm,n,\per}_{i-1/2,j} - \theta^{\pm,n,\per}_{i-1/2,j+1} = \theta^{\pm,n,\per}_{i-1,j+1/2} - \theta^{\pm,n,\per}_{i,j+1/2}.
\end{align*}
Therefore, we get
\begin{align*}
\theta^{\pm,n+1,\per}_{i,j+1/2} &= \left(1-\deltat \left(\lambda^\pm_{i,j}[\rho^{n+1,\per}]_+ + \lambda^\pm_{i,j}[\rho^{n+1,\per}]_-\right)\right) \theta^{\pm,n,\per}_{i,j+1/2} \\
&+ \deltat \, \lambda^\pm_{i,j}[\rho^{n+1,\per}]_+ \, \theta^{\pm,n,\per}_{i+1,j+1/2} + \deltat \, \lambda^\pm_{i,j}[\rho^{n+1,\per}]_- \, \theta^{\pm,n,\per}_{i-1,j+1/2}\\
&+ \deltat \left(\lambda^\pm_{i,j+1}[\rho^{n+1,\per}]_+-\lambda^\pm_{i,j}[\rho^{n+1,\per}]_+ \right) \left(\theta^{\pm,n,\per}_{i+1/2,j+1} + L \right)\\
&+ \deltat \left(\lambda^\pm_{i,j}[\rho^{n+1,\per}]_- - \lambda^\pm_{i,j+1}[\rho^{n+1,\per}]_-\right) \left(\theta^{\pm,n,\per}_{i-1/2,j+1}+L\right).
\end{align*}
Applying the CFL condition~\eqref{delta t delta x} we obtain
\begin{align*}
\left|\theta^{\pm,n+1,\per}_{i,j+1/2}\right| \leq \|\theta^{\pm,n, \per}_{x_2}\|_{\ell^{\infty}(\I^2)} + 2 \deltat \left|\lambda^\pm_{i,j+1}[\rho^{n+1,\per}]-\lambda^\pm_{i,j}[\rho^{n+1,\per}]\right|\,\|\theta^{\pm,n, \per}_{x_1} +  L\|_{\ell^{\infty}(\I^2)}.
\end{align*}
Hence, arguing as in the proof of~\eqref{unf_theta_i}, we deduce the existence of a constant, still denoted $C_M>0$, such that
\begin{align*}
\deltat \left|\lambda^\pm_{i,j+1}[\rho^{n+1,\per}]-\lambda^\pm_{i,j}[\rho^{n+1,\per}]\right| \leq \Delta t \, C_M,
\end{align*} 
so that
\begin{align*}
\left|\theta^{\pm,n+1,\per}_{i,j+1/2}\right| \leq \|\theta^{\pm,n, \per}_{x_2}\|_{\ell^{\infty}(\I^2)} + 2 \Delta t \, C_M\, e^{T \, C_M} \left(\|\pa_{x_1}\rho^{\pm,\per}_{M,0}\|_{L^\infty(\T)} +  L \right),
\end{align*}
where we have used estimate~\eqref{unf_theta_i}. This concludes the proof of Proposition~\ref{prop lip}.
\end{proof}

\begin{corollary}\label{coro unif bounds Q1} 
Let the assumptions of Theorem~\ref{theorem-properties} hold. Then, there exists a constant $\widetilde{C}_M>0$, independent of $\eps=\max(\Delta x, \Delta t)$, such that
\begin{align*}
\|\rho^{\pm,\per,\eps}\|_{L^\infty(0,T;W^{1,\infty}(\T))} + \|\pa_t \rho^{\pm,\per,\eps}\|_{L^\infty((0,T)\times \T)} \leq \widetilde{C}_M,
\end{align*}
where the functions $\rho^{\pm,\per,\eps}$ are given by~\eqref{continuous rho}.
\end{corollary}

\begin{proof}
The uniform $L^\infty(0,T;W^{1,\infty}(\T))$ estimate is a direct consequence of the definition~\eqref{continuous rho} of the function $\rho^{\pm,\per,\eps}$, estimate 
\eqref{bound of rho} and Proposition~\ref{prop lip} (see \cite[Section 4.1]{MR4221324} for more details). For the in $L^\infty((0,T)\times\T)$ estimate on $\pa_t \rho^{\pm,\per,\eps}$, we first define
\begin{align}\label{definition of tau}
    \tau_{i, j}^{\pm, n +1/2, \per} = \dfrac{\rho^{\pm,n + 1, \per}_{i, j} - \rho^{\pm,n ,\per}_{i, j}}{\Delta t}, \qquad \forall (i, j) \in \I^2, \,0 \leq n \leq N_T-1.
\end{align}
Then, thanks to definition~\eqref{continuous rho}, for $(t, x_1, x_2) \in [t_n , t_{n+1}] \times \overline{C}_{i,j}$, we have
{\small
\begin{equation}\begin{array}{llll}\label{absolute value delta t}
\displaystyle    \partial_t \rho^{\pm, \per, \eps}(t, x_1, x_2) &
\displaystyle = \left(\frac{1}{\Delta t}\right) \left[ \left(\frac{x_1 - x_i}{\Delta x}\right) \left(\frac{x_2 - x_j}{\Delta x}\right) \rho_{i+1, j+1}^{\pm, n + 1, \per} + \left(1 - \frac{x_1 - x_i}{\Delta x}\right) \left(\frac{x_2 - x_j}{\Delta x}\right) \rho_{i, j+1}^{\pm, n + 1, \per} \right. \\
&\displaystyle  \left. + \left(\frac{x_1 - x_i}{\Delta x}\right) \left(1 - \frac{x_2 - x_j}{\Delta x}\right) \rho_{i+1, j}^{\pm, n + 1, \per} + \left(1 - \frac{x_1 - x_i}{\Delta x}\right) \left(1 - \frac{x_2 - x_j}{\Delta x}\right) \rho_{i, j}^{\pm, n + 1, \per} \right] \\
& \displaystyle  + \left( - \frac{1}{\Delta t}\right) \left[ \left(\frac{x_1 - x_i}{\Delta x}\right) \left(\frac{x_2 - x_j}{\Delta x}\right) \rho_{i+1, j+1}^{\pm, n, \per} + \left(1 - \frac{x_1 - x_i}{\Delta x}\right) \left(\frac{x_2 - x_j}{\Delta x}\right) \rho_{i, j+1}^{\pm, n, \per} \right. \\
&\displaystyle  \left. + \left(\frac{x_1 - x_i}{\Delta x}\right) \left(1 - \frac{x_2 - x_j}{\Delta x}\right) \rho_{i+1, j}^{\pm, n, \per} + \left(1 - \frac{x_1 - x_i}{\Delta x}\right) \left(1 - \frac{x_2 - x_j}{\Delta x}\right) \rho_{i, j}^{\pm, n, \per} \right],
\end{array}\end{equation}}
so that
\begin{multline*}
      \left|\partial_t \rho^{ \pm, \per, \eps}(t, x_1, x_2)\right| \leq  \left(\frac{x_1 - x_i}{\Delta x}\right) \left(\frac{x_2 - x_j}{\Delta x}\right) \left|\tau_{i+1, j+1}^{\pm,n+1/2, \per} \right| +  \left(1 - \frac{x_1 - x_i}{\Delta x}\right) \left(\frac{x_2 - x_j}{\Delta x}\right) \left|\tau_{i, j+1}^{\pm, n+1/2, \per}\right| \\
       + \left(\frac{x_1 - x_i}{\Delta x}\right) \left(1 - \frac{x_2 - x_j}{\Delta x}\right) \left|\tau_{i+1, j}^{\pm, n+1/2, \per}\right| + \left(1 - \frac{x_1 - x_i}{\Delta x}\right) \left(1 - \frac{x_2 - x_j}{\Delta x}\right) \left|\tau_{i, j}^{\pm, n+1/2, \per}\right|.
\end{multline*}
Now, using the definition of the scheme, we notice that for any $(i,j) \in \I^2$
\begin{align}
       \tau_{i, j}^{\pm,n +1/2, \per} = \lambda^\pm_{i, j}[\rho^{n + 1,\per}]_+\left( \theta_{i +1/2, j}^{\pm, n, \per}  + L\right) - \lambda^\pm_{i, j}[\rho^{n + 1,\per}]_- \left(\theta_{i  -1/2, j}^{\pm, n, \per} + L\right).
\end{align}
Hence, applying Proposition~\ref{prop lip} and Lemma~\ref{borne-sur-lambda}, we conclude that there exists a constant $\widetilde{C}_M>0$ independent of $\eps$ such that
\begin{align*}
    \left|\partial_t \rho^{\pm,\per, \eps}(t, x_1,x_2)\right| \leq \widetilde{C}_M, \quad \mbox{for }(t,x_1,x_2) \in (0,T)\times\T.   
\end{align*}
This concludes the proof of Corollary~\ref{coro unif bounds Q1}.
\end{proof}

\section{Convergence of the scheme}\label{convergence for fixed M}

In this section, we prove Theorem~\ref{convergence}. For this purpose,  as a direct consequence of Corollary~\ref{coro unif bounds Q1} and the compact embedding of 
$W^{1,\infty}((0,T) \times \T)$ into $C([0,T]\times \T)$, we deduce the following:

\begin{proposition}\label{prop compactness}
Let the assumptions of Theorem~\ref{convergence} hold. Then, for any positive integer $M$, there exist functions
\begin{align*}
\rho^{\pm,\per}_M \in   W^{1,\infty}((0,T) \times \T),
\end{align*}
such that, up to a subsequence, as $m \to \infty$, it holds
\begin{align*}
\rho^{\pm,\per,m}_M \to \rho^{\pm,\per}_M \quad \mbox{strongly in } C([0,T]\times \T). 
\end{align*}
\end{proposition}

\begin{remark}\label{Rk reg}
The improved estimates established in Section~\ref{section M to infinity}, together with the arguments of~\cite{cannone2010global}, allow to deduce weaker convergence properties for the sequences $(\rho^{\pm,\per,m}_M)_{m \in \N}$. Moreover, we notice that such estimates could be established without any regularization of the singular kernel~$\K$. In fact, we mainly need to regularize the kernel in order to identify the limit functions constructed in Proposition~\ref{prop compactness} as solutions to~\eqref{original1}-\eqref{IC.reg.per} in the distributional sense.
\end{remark}

Then, in the next section we identify the functions $\rho^{\pm,\per}_M$ obtained in Proposition~\ref{prop compactness} as solutions to~\eqref{original1}-\eqref{IC.reg.per} in the distributional sense. This proof relies on classical arguments and can be seen, for instance, as a slight adaptation of~\cite[Theorem 2]{ZHIZ25}. However, in the present work this proof is rather tedious since we work on the two dimensional torus. For the convenience of the reader we will write $\rho^{\pm,\per,m}$ instead of $\rho^{\pm,\per,m}_M$.

\subsection{Identification of the limit} 

The proof works as follows: first we identify the equation satisfies by $\rho^{\pm,\per,m}$ by deriving in time this function, see~\eqref{partial t with e}. In particular, this equation is mainly composed of two terms. One term can be seen as an error term (denoted $e^\pm_m$ in the sequel) and we have to show its convergence (in a weak sense) towards zero as $m \to\infty$. Then, we will prove that the second term converges towards the right hand side of~\eqref{original1}. Let us observe that this former convergence property is possible thanks to our regularization of the kernel $\K$. In order to clarify as much as possible our proof, we divide it in several steps.

\subsubsection{Equations satisfy by \texorpdfstring{$\rho^{\pm,\per,m}$}{\rho^{\pm,\per,m}} and definition of the error term \texorpdfstring{$e^\pm_m$}{e^\pm_m}}

First, for all $i \in \Imm$ and $x \in [x_i,x_{i+1}]$, we define the functions
\begin{align}\label{def : ai}
a^m_i (x)= \dfrac{x - x_i}{\Delta x_m}, \quad b^m_i(x)= 1 - a^m_i(x).
\end{align}
Then, for all $(t,x_1,x_2) \in (t_n,t_{n+1})\times C_{i,j}$ we have, similarly to~\eqref{absolute value delta t},
\begin{align*}
   \partial_t &\rho^{\pm, \per, m}(t, x_1, x_2)\\
    &=   a^m_i (x_1)\bigg\{a^m_j(x_2) \bigg[\lambda_{i+1, j + 1}^\pm[\rho^{n + 1,\per}]_+\bigg(\theta^{\pm,n, \per}_{i +3/2, j + 1}+ L\bigg) - \lambda^\pm_{i+1, j + 1}[\rho^{n + 1,\per}]_-\bigg(\theta^{\pm, n,\per}_{i +1/2, j + 1} + L\bigg)\bigg]\\
    &\phantom{xxxxxxxxxxx}+ b^m_j(x_2) \bigg[\lambda_{i+1, j}^\pm[\rho^{n + 1,\per}]_+ \bigg(\theta^{\pm,n,\per}_{i +3/2, j} + L\bigg) - \lambda_{i+1, j}^\pm[\rho^{n + 1,\per}]_-\bigg(\theta^{\pm,n, \per}_{i +1/2, j} + L\bigg)\bigg]\bigg\}\\ 
    &+b^m_i(x_1)\bigg\{a^m_j(x_2)\bigg[\lambda^\pm_{i, j + 1}[\rho^{n + 1,\per}]_+\bigg(\theta^{\pm,n,\per}_{i +1/2, j + 1} + L\bigg) - \lambda^\pm_{i, j + 1}[\rho^{n + 1,\per}]_-\bigg(\theta^{\pm, n, \per}_{i -1/2, j + 1}+ L\bigg)\bigg]  \\
    &\phantom{xxxxxxxxxxx}+b^m_j(x_2)\bigg[\lambda_{i, j}^\pm[\rho^{n + 1,\per}]_+\bigg(\theta^{\pm,n,\per}_{i +1/2, j} + L\bigg) - \lambda^\pm_{i, j}[\rho^{n + 1,\pm}]_-\bigg(\theta^{\pm,n,\per}_{i -1/2, j} + L\bigg)\bigg]\bigg\}.
\end{align*}
Now, let $\rho^{\per} := \rho^{+,\per}-\rho^{-,\per}$, we add and subtract the positive and negative parts of the function
\begin{align}\label{definiyion of lambda_M}
\lambda^\pm[\rho^{\per}](t, x_1, x_2):= \mp
\left(\left( \sigma_{M}^\K(\cdot) \ast \rho^{\per}(t, \cdot)\right)(x_1, x_2)+a(t)\right),
\end{align}
in the above expression, and we obtain
\begin{align}\label{partial t with e}
   & \partial_t \rho^{\pm,\per, m}(t, x_1, x_2) \\&=  \lambda^\pm[\rho^{\per}]_+(t , x_1, x_2) \bigg\{a^m_i(x_1) a^m_j(x_2) \bigg(\theta^{\pm,n,\per}_{i +3/2, j + 1}+ L\bigg) + a^m_i(x_1) b^m_j(x_2)\bigg(\theta^{\pm,n, \per}_{i +3/2, j} + L\bigg) \nonumber\\&+b^m_i(x_1)a^m_j(x_2)\bigg(\theta^{\pm,n,\per}_{i +1/2, j + 1} + L\bigg) + b^m_i(x_1)b^m_j(x_2) \bigg(\theta^{\pm,n,\per}_{i +1/2, j} + L\bigg) \bigg\}\nonumber\\
    &- \lambda^\pm[\rho^{\per}]_-(t , x_1, x_2)\bigg\{a^m_i(x_1) a^m_j(x_2)\bigg(\theta^{\pm,n,\per}_{i +1/2, j + 1} + L\bigg)+ a^m_i(x_1)b^m_j(x_2)\bigg(\theta^{\pm,n,\per}_{i +1/2, j} + L\bigg) \nonumber\\
    & +  b^m_j(x_1)a^m_j(x_2) \bigg(\theta^{\pm,n,\per}_{i -1/2, j + 1}+ L\bigg) +  b^m_i(x_1)b^m_j(x_2) \bigg(\theta^{\pm,n,\per}_{i -1/2, j} + L\bigg)\bigg\} +   e^\pm_m (t, x_1, x_2) \nonumber.
\end{align} 
where
\begin{align*}
    e^\pm_m (t, x_1, x_2)=& \,a^m_i(x_1) \bigg\{ a^m_j(x_2) \bigg[\lambda_{i+1, j + 1}^\pm[\rho^{n + 1,\per}]_+ - \lambda^\pm[\rho^{\per}]_+(t , x_1, x_2)\bigg]\bigg(\theta^{\pm,n,\per}_{i +3/2, j + 1}+ L\bigg) \\
    &\phantom{xxxxx}- a^m_j(x_2)\bigg[ \lambda_{i+1, j + 1}^\pm[\rho^{n + 1,\per}]_- - \lambda^\pm[\rho^{\per}]_-(t , x_1, x_2)\bigg]\bigg(\theta^{\pm,n,\per}_{i +1/2, j + 1} + L\bigg)\\
    &\phantom{xxxxx}+ b^m_j(x_2)\bigg[\lambda_{i+1, j}^\pm[\rho^{n + 1,\per}]_+ -\lambda^\pm[\rho^{\per}]_+(t , x_1, x_2)\bigg] \bigg(\theta^{\pm,n, \per}_{i +3/2, j} + L\bigg)\\
    &\phantom{xxxxx}-b^m_j(x_2)\bigg[\lambda^\pm_{i+1, j}[\rho^{n + 1,\per}]_- - \lambda^\pm[\rho^{\per}]_-(t , x_1, x_2)\bigg]\bigg(\theta^{\pm, n,\per}_{i +1/2, j} + L\bigg) \bigg\}\\
 +& \, b^m_i(x_1)\bigg\{a^m_j(x_2)\bigg[\lambda_{i, j + 1}^\pm[\rho^{n + 1,\per}]_+ - \lambda^\pm[\rho^{\per}]_+(t , x_1, x_2)\bigg] \bigg(\theta^{\pm, n,\per}_{i +1/2, j + 1} + L\bigg)\\
 &\phantom{xxxxx}-a^m_j(x_2) \bigg[\lambda_{i, j + 1}^\pm[\rho^{n + 1,\per}]_- - \lambda^\pm[\rho^{\per}]_-(t , x_1, x_2)\bigg] \bigg(\theta^{\pm, n,\per}_{i -1/2, j + 1}+ L\bigg)\\
 &\phantom{xxxxx}+ b^m_j(x_2) \bigg[\lambda_{i, j}^\pm[\rho^{n + 1,\per}]_+- \lambda^\pm[\rho^{\per}]_+(t , x_1, x_2)\bigg] \bigg(\theta^{\per,m}_{i +1/2, j} + L\bigg)\\
 &\phantom{xxxxx}- b^m_j(x_2) \bigg[ \lambda_{i, j}^\pm[\rho^{n + 1,\per}]_- - \lambda^\pm[\rho^{\per}]_-(t , x_1, x_2) \bigg]\bigg(\theta^{\pm,n, \per}_{i - 1/2, j} + L\bigg)\bigg\}.
\end{align*}

\subsubsection{Study of the error term \texorpdfstring{$e^\pm_m$}{e^\pm_m}}

Now, let $\varphi \in C_c^{\infty}([0, T] \times \T)$, our main objective is to show that
\begin{align}\label{conv em}
\int_{[0,T] \times \T} \varphi(t,x_1, x_2)e_m(t, x_1, x_2)\, dt \, dx_1 \,dx_2  \to 0, \quad \mbox{as }m \to \infty.
\end{align}
For this purpose, we notice, thanks to~\eqref{unf_theta_i}, that it holds
\begin{multline*}
\left|\int_{[0,T] \times \T} \varphi(t,x_1, x_2)e_m(t, x_1, x_2)\, dt \, dx_1 \,dx_2 \right| \leq 8 T \, e^{T \, C_M} \, \left(\|\pa_{x_1} \rho^{\pm,\per}_0\|_{L^\infty(\T)} + L \right) \\
\times \sup_{(\tau,y_1,y_2) \in \mathrm{supp}(\varphi)} \left( \sup_{|t_{n+1}-\tau|\leq \Delta t_m} \, \sup_{\substack{|x_i-y_1|\leq \Delta x_m\\|x_j-y_2|\leq\Delta x_m}} \left|\lambda^{\pm}_{i,j}[\rho^{n+1,\per}]-\lambda^\pm[\rho^{\per}](\tau,y_1,y_2)\right|\right).
\end{multline*}
Now, let $(\tau,y_1,y_2) \in \mathrm{supp}(\varphi)$ such that $|t_{n+1}-\tau|\leq \Delta t_m$, $|x_i-y_1|\leq \Delta x_m$ and $|x_j-y_2|\leq \Delta x_m$. We notice that
\begin{align*}
 \left|\lambda^{\pm}_{i,j}[\rho^{n+1,\per}]-\lambda^\pm[\rho^{\per}](\tau,y_1,y_2)\right| \leq I_{m,1} + I_{m,2} + I_{m,3}+ I_{m,4},
\end{align*}
where
\begin{align*}
 I_{m,1}&= \sum_{(\ell,r)\in\Imm^2} \int_{C_{\ell,r}} \left|\left(\sigma_M^\K(x_\ell,x_r) - \sigma_M^\K(z_1,z_2)\right) \, \rho^{\per,m}(t_{n+1},x_i-x_\ell,x_j-x_r)\right| \,dz_1 \, dz_2, \\
 I_{m,2} &= \sum_{(\ell,r)\in\Imm^2} \int_{C_{\ell,r}} \left| \sigma_M^\K(z_1,z_2) \,\left( \rho^{\per,m}(t_{n+1},x_i-x_\ell,x_j-x_r)- \rho^{\per,m}(\tau,y_1-z_1,y_2-z_2)\right)\right| \,dz_1 \, dz_2,\\
 I_{m,3} &= \sum_{(\ell,r)\in\Imm^2} \int_{C_{\ell,r}} \left| \sigma_M^\K(z_1,z_2) \, \left(\rho^{\per,m}(\tau,y_1-z_1,y_2-z_2)-\rho^{\per}(\tau,y_1-z_1,y_2-z_2)\right)\right| \, dz_1 \, dz_2,
\end{align*}
and
\begin{align*}
I_{m,4} &= |a(t_{n+1})-a(\tau)|. 
\end{align*}
Thanks to the bounds~\eqref{bound of rho} and~\eqref{derivative sigmaM}, we can easily show, as in~\cite{ZHIZ25}, that there exists a constant $C>0$, dependent on $M$ but independent of $\eps_m$, such that
\begin{align}\label{Im1}
I_{m,1} \leq C \, \Delta x_m \to 0 \quad \mbox{as }m \to +\infty.
\end{align}
Now, for the term $I_{m,2}$, using the estimate~\eqref{bound of the kernel}, we obtain
\begin{align*}
I_{m,2} \leq M^2 \left(I^+_{m,2} + I^-_{m,2}\right),
\end{align*}
with
\begin{align*}
I^{\pm}_{m,2} = \sum_{(\ell,r)\in\Imm^2} \int_{C_{\ell,r}} \left|\rho^{\pm,\per,m}(t_{n+1},x_i-x_\ell,x_j-x_r)- \rho^{\pm,\per,m}(\tau,y_1-z_1,y_2-z_2)\right| \,dz_1 \, dz_2.
\end{align*}
Let us observe, that we can rewrite the term $I^\pm_{m,2}$ as follows
\begin{multline*}
I^{\pm}_{m,2} \leq \sum_{(\ell,r)\in\Imm^2} \int_{C_{\ell,r}} \Bigg(\int_\tau^{t_{n+1}} |\pa_t \rho^{\pm,\per,m}(t,x_i-x_\ell,x_j-x_r)| \, dt \\ + \int_{y_1-z_1}^{x_i-x_\ell} |\pa_{x_1} \rho^{\pm,\per,m}(\tau,s_1,x_j-x_r)| \, ds_1 + \int_{y_2-z_2}^{x_j-x_r} |\pa_{x_2} \rho^{\pm,\per,m}(\tau,y_1-z_1,s_2)| \, ds_2 \Bigg)
\end{multline*}
It remains to apply the results of Corollary~\ref{coro unif bounds Q1} in order to deduce the existence of a constant $\widetilde{C}_M>0$, such that
\begin{align*}
|I^\pm_{m,2}| \leq \widetilde{C}_M(\Delta t_m + 4 \Delta x_m),
\end{align*}
so that
\begin{align}\label{Im2}
I_{m,2} \to 0 \quad \mbox{as }m \to +\infty.
\end{align}
Finally, applying the bound~\eqref{bound of the kernel} and assumption~\ref{H2}, we can see that 
\begin{align}\label{Im3}
I_{m,3} \leq M^2 \, \left(\|\rho^{+,\per,m}-\rho^{+,\per}\|_{L^\infty((0,T)\times\T)} + \|\rho^{+,\per,m}-\rho^{+,\per}\|_{L^\infty((0,T)\times\T)}\right) \to 0 \quad \mbox{as }m\to +\infty, 
\end{align}
as well as 
\begin{align}\label{Im4}
I_{m,4}\to 0 \quad \mbox{as }m\to +\infty.
\end{align}
Therefore, collecting~\eqref{Im1}--\eqref{Im4} we deduce that the convergence property~\eqref{conv em} holds true.

\subsubsection{Conclusion of the proof}

We now define the function
\begin{align*}
\theta^{\pm, \per, m}_{x_1}(t,x_1,x_2) = \theta^{\pm,n,\per}_{i +1/2, j} \quad \mbox{for } (t,x_1,x_2) \in [t_n, t_{n + 1})\times[x_i, x_{i + 1})\times [x_j, x_{j + 1}). 
 \end{align*}
Then, tanks to Proposition~\ref{prop lip}, we deduce the existence of $\theta^{\pm, \per}_{x_1}$ such that for any $\varphi \in C_c^{\infty}((0, T) \times \mathbb{T}^2)$, it holds
\begin{align}\label{weak convergence of theta}
    \int_{(0,T)\times \T} \theta^{\pm, \per, m}_{x_1}\, \varphi \,dt \, dx_1 \, dx_2 \to \int_{(0,T)\times\T} \theta^{\pm, \per}_{x_1}\, \varphi \, dt\, dx_1 \, dx_2, \quad \text{as} \ m \to + \infty.
\end{align}
Moreover, for all $i \in \Imm$ and $x \in [x_i,x_{i+1})$, we rewrite the functions $a^m_i$ and $b^m_i$ as follows
\[
a^m_i (x)= \frac{x}{\Delta x_m} - \left\lfloor\frac{x}{\Delta x_m}\right\rfloor, \quad b^m_i(x) = 1 - a^m_i(x),
\]
where $\lfloor \cdot \rfloor$ denotes the floor function. With these notations at hand, and for $(t,x_1,x_2) \in (t_n,t_{n+1})\times C_{i,j}$, the relation~\eqref{partial t with e} can be rewritten as
\begin{align*}
    \left(\partial_t \rho^{\pm,\per, m} - e^\pm_m\right)(t, x_1, x_2) &=  \lambda^\pm[\rho^{\per}]_+(t , x_1, x_2) \bigg\{a^m_i(x_1) a^m_j(x_2) \bigg(\theta^{\pm,\per,m}_{x_1}(t,x_1+\Delta x_m, x_2+\Delta x_m)+ L\bigg)\\
   &+ a^m_i(x_1) b^m_j(x_2)\bigg(\theta^{\pm,\per,m}_{x_1}(t,x_1+\Delta x_m,x_2) + L\bigg)\\
   &+b^m_i(x_1)a^m_j(x_2)\bigg(\theta^{\pm,\per,m}_{x_1}(t,x_1,x_2+\Delta x_m) + L\bigg) \\
   &+ b^m_i(x_1)b^m_j(x_2) \bigg(\theta^{\pm,\per,m}_{x_1}(t,x_1,x_2) + L\bigg) \bigg\} \\
    &- \lambda^\pm[\rho^{\per}]_-(t , x_1, x_2)\bigg\{a^m_i(x_1) a^m_j(x_2)\bigg(\theta^{\pm,\per,m}_{x_1}(t,x_1,x_2+\Delta x)+ L\bigg)\\
    &+ a^m_i(x_1)b^m_j(x_2)\bigg(\theta^{\pm,\per,m}_{x_1}(t,x_1,x_2) + L\bigg) \\
    & +  b^m_j(x_1)a^m_j(x_2) \bigg(\theta^{\pm,\per,m}_{x_1}(t,x_1-\Delta x_m,x_2+\Delta x_m)+ L\bigg)\\
    &+  b^m_i(x_1)b^m_j(x_2) \bigg(\theta^{\pm,\per,m}_{x_1}(t,x_1-\Delta x_m,x_2)+ L\bigg)\bigg\}.
\end{align*} 
Now, we define, for any  $\varphi \in C_c^{\infty}((0, T) \times \T)$, the term $A^\pm_m$ by 
\begin{align*}
    A^\pm_m =   \int_{(0,T)\times \T} \left(  \partial_t \rho^{\pm, \per,m} -  e^\pm_m \right) (t, x_1, x_2) \,\varphi(t, x_1, x_2)\, dt\, dx_1\,dx_2,
\end{align*}
and we rewrite this term as follows
\begin{align*}
   A^\pm_m = \int_{(0,T)\times\T} \theta^{\pm,\per,m}_{x_1}(t,x_1,x_2) \, A^\pm_{m,1}(t,x_1,x_2) \, dt \, dx_1 \, dx_2
    + L \int_{(0,T)\times \T} \left(\lambda^\pm[\rho^\per] \, \varphi \right)(t,x_1,x_2) \, dt \, dx_1 \, dx_2,
\end{align*} 
with
\begin{align*}
A^\pm_{m,1} &= a^m_i(x_1) a^m_j(x_2) \left(\lambda^\pm[\rho^{\per}]_+ \, \varphi\right)(t,x_1-\Delta x_m, x_2-\Delta x_m) \\
   &+ a^m_i(x_1) b^m_j(x_2) \left(\lambda^\pm[\rho^{\per}]_+ \, \varphi\right)(t,x_1-\Delta x_m,x_2) \\
   &+b^m_i(x_1)a^m_j(x_2)\left(\lambda^\pm[\rho^{\per}]_+ \, \varphi\right)(t,x_1,x_2-\Delta x_m)  \\
   &+ b^m_i(x_1)b^m_j(x_2) \left(\lambda^\pm[\rho^{\per}]_+ \, \varphi\right)(t,x_1,x_2)   \\
    &- a^m_i(x_1) a^m_j(x_2) \left(\lambda^\pm[\rho^{\per}]_- \, \varphi\right)(t,x_1,x_2-\Delta x_m)\\
    &- a^m_i(x_1)b^m_j(x_2)\left(\lambda^\pm[\rho^{\per}]_- \, \varphi\right)(t,x_1,x_2) \\
    &-  b^m_j(x_1)a^m_j(x_2) \left(\lambda^\pm[\rho^{\per}]_- \, \varphi\right)(t,x_1+\Delta x_m,x_2-\Delta x_m)\\
    &-  b^m_i(x_1)b^m_j(x_2) \left(\lambda^\pm[\rho^{\per}]_- \, \varphi\right)(t,x_1+\Delta x_m,x_2).
\end{align*}
Let us now introduce the term $B^\pm_m$:
\begin{multline*}
B^\pm_m =\int_{(0,T)\times \T} \theta^{\pm, \per, m}_{x_1}(t, x_1, x_2) \, \left(\lambda^\pm[\rho^\per] \, \varphi \right)(t,x_1,x_2) \, dt \, dx_1\, dx_2 \\ + L \int_{(0,T)\times \T} \left(\lambda^\pm[\rho^{\per}] \, \varphi\right)(t, x_1, x_2) \, dt \, dx_1 \,dx_2.
\end{multline*}
Hence, using the definition~\eqref{definiyion of lambda_M} of $\lambda^\pm[\rho^\per]$, the fact that this function is continuous and the convergence property~\eqref{weak convergence of theta}, we conclude that
\begin{align}
B^\pm_m \to \mp \int_{(0,T) \times \T} \left(
(\sigma_M^\K(\cdot) \ast \rho^\per (t,\cdot)\right)(x_1,x_2)+a(t)) \, \left(\theta^{\pm,\per}_{x_1}+L\right)(t,x_1,x_2) \, \varphi(t,x_1,x_2) \, dt \, dx_1 \, dx_2.
\end{align}
Now, in the first integral of $B^\pm_m$, we split the term $\lambda^\pm \,\varphi$ as in $A^\pm_{m,1}$ without $\pm\Delta x_m$ and we get the following estimate
\begin{align*}
\left|A^{\pm}_m - B^{\pm}_m\right| \leq \|\theta^{+,\per,m}_{x_1}&\|_{L^\infty((0,T)\times \T)}\\ 
&\times \max \Big\{ \|\left(\lambda^\pm[\rho^{\per}]_+ \, \varphi\right)(\cdot,\cdot-\Delta x_m, \cdot-\Delta x_m)-\left(\lambda^\pm[\rho^{\per}]_+ \, \varphi\right)(\cdot,\cdot, \cdot)\|_{L^\infty((0,T)\times\T)}\\
&\phantom{xxxxxxx} \|\left(\lambda^\pm[\rho^{\per}]_+ \, \varphi\right)(\cdot,\cdot-\Delta x_m, \cdot)-\left(\lambda^\pm[\rho^{\per}]_+ \, \varphi\right)(\cdot,\cdot, \cdot)\|_{L^\infty((0,T)\times\T)}\\
&\phantom{xxxxxxx} \|\left(\lambda^\pm[\rho^{\per}]_+ \, \varphi\right)(\cdot,\cdot, \cdot-\Delta x_m)-\left(\lambda^\pm[\rho^{\per}]_+ \, \varphi\right)(\cdot,\cdot, \cdot)\|_{L^\infty((0,T)\times\T)}\\
&\phantom{xxxxxxx} \|\left(\lambda^\pm[\rho^{\per}]_- \, \varphi\right)(\cdot,\cdot, \cdot-\Delta x_m)-\left(\lambda^\pm[\rho^{\per}]_- \, \varphi\right)(\cdot,\cdot, \cdot)\|_{L^\infty((0,T)\times\T)}\\ 
&\phantom{xxxxxxx} \|\left(\lambda^\pm[\rho^{\per}]_- \, \varphi\right)(\cdot,\cdot+\Delta x_m, \cdot-\Delta x_m)-\left(\lambda^\pm[\rho^{\per}]_- \, \varphi\right)(\cdot,\cdot, \cdot)\|_{L^\infty((0,T)\times\T)}\\
&\phantom{xxxxxxx} \|\left(\lambda^\pm[\rho^{\per}]_- \, \varphi\right)(\cdot,\cdot+\Delta x_m, \cdot)-\left(\lambda^\pm[\rho^{\per}]_- \, \varphi\right)(\cdot,\cdot, \cdot)\|_{L^\infty((0,T)\times\T)}\Big\}.
\end{align*}
Besides, for all $(t, x_1, x_2) \in (0, T)\times \T$, we notice, for instance, that
\begin{align*}
    \Big|&\left(\lambda^\pm[\rho^{\per}]_+ \, \varphi\right)(t,x_1-\Delta x_m,x_2-\Delta x_m)-\left(\lambda^\pm[\rho^{\per}]_+ \, \varphi\right)(t,x_1,x_2)\Big| \\
    &\leq\|\varphi\|_{L^\infty((0,T)\times\T)} \int_\T \left|\sigma_M^\K(y_1,y_2) \left(\rho^\per(t,x_1-\Delta x_m-y_1,x_2-\Delta x_m-y_2)- \rho^\per(t,x_1-y_1,x_2-y_2)\right)\right| \,dy_1 \, dy_2\\
    &+ \|\rho^\per\|_{L^\infty((0,T)\times \T)} \, \int_\T \left|\sigma_M(y_1,y_2)\left(\varphi(t,x_1,x_2)-\varphi(t,x_1-\Delta x_m,x_2-\Delta x_m)\right)\right| \, dy_1 \, dy_2.
\end{align*}
Thereby, the regularity of the functions $\varphi$ and $\rho^\per$, allow us to conclude, for a fixed  $M > 0$, that
\begin{align*}
   \|\left(\lambda^\pm[\rho^{\per}]_+ \, \varphi\right)(\cdot,\cdot-\Delta x_m, \cdot-\Delta x_m)-\left(\lambda^\pm[\rho^{\per}]_+ \, \varphi\right)(\cdot,\cdot, \cdot)\|_{L^\infty((0,T)\times\T)} \to 0, \quad \mbox{as } m \to \infty.
\end{align*}
Similarly, we can show that the remaining terms in  $A^\pm_m - B^\pm_m$ converge to zero as $m \to +\infty$, which implies that
\begin{align*}
    | A^\pm_m - B^\pm_m | \to 0, \quad \mbox{as } m \to +\infty.
\end{align*}
Therefore, we deduce that
\begin{align*}
\pa_t \rho^{\pm,\per} = \mp \left(\sigma_M^\K \ast \rho^\per +a(\cdot) \right) \, \left(\theta^{\pm,\per}_{x_1} + L \right), \quad \mbox{in } \mathcal{D}'((0, T) \times \T).
\end{align*}
Finally, using the expression of $\pa_{x_1} \rho^{\pm,\per,m}$ and the convergence property~\eqref{weak convergence of theta}, it is clear that
\[
\theta^{\pm\,\per}_{x_1} = \partial_{x_1} \rho^{\pm,\per}, \quad \text{in} \quad \mathcal{D}'((0, T) \times \T).
\]
Hence, we deduce that
\begin{align*}
\partial_t \rho^{\pm,\per} = \mp \left(\sigma_M^\K \ast \rho^\per +a(\cdot) \right) \, \left(\pa_{x_1} \rho^{\pm,\per} + L \right), \quad \mbox{in } \mathcal{D}'((0, T) \times \T),
\end{align*}
which concludes the proof of Theorem~\ref{convergence}.

\section{Passing to the limit in the regularized system}\label{section M to infinity}

As already described, in order to prove Theorem~\ref{theorem-final convergence}, the main step is to prove that $\sigma_M^\K \ast \rho^{\per}_M$, with $\rho^\per_M = \rho^{+,\per}_M-\rho^{-,\per}_M$, is uniformly bounded in $L^2(0,T;H^1(\T))$, where $\rho^{\pm,\per}_M$ are the solutions obtained in Theorem~\ref{convergence}. Then, in a second step we will adapt the arguments given in the proof of~\cite[Theorem 1.4]{cannone2010global} to conclude.

\subsection{Uniform estimate on the reconstructed velocity field}
In the sequel, we denote by $\rho^{\per,m}_M$ the function $\rho^{+,\per,m}_M - \rho^{-,\per,m}_M$. Then, our main objective in this section is to prove a convenient estimate on the $L^2(0,T;H^1(\T))$ norm of $\sigma_M^\K \ast \rho^{\per,m}_M$.

\begin{proposition}\label{prop velocity field H1}
Let the assumptions of Theorem~\ref{convergence} hold. Then, for any positive integer $M \leq N_m =1/\Delta x_m$, there exists a constant $C_M>0$ independent of $\eps_m$ such that
\begin{multline}\label{norme-h1-vitesse}
\|\nabla (\sigma_M^\K \ast \rho^{\per,m}_M)\|^2_{L^2((0,T)\times \T)} \leq  8\sum_{\pm} \int_\T f\left(\pa_{x_1} \rho^{\pm,\per}_{M,0}(x_1,x_2)+L\right)\,dx_1\,dx_2 \\+ 16\left(f(e)+L \ln(2)\right) + C_M \, \eps_m.
\end{multline}
Furthermore, it holds
\begin{align}\label{norme-l2-vitesse}
\|\sigma_M^\K \ast \rho^{\per,m}_M\|_{L^\infty(0,T;L^2(\T))} \leq \|\rho^{\per,m}_M\|_{L^\infty(0,T;L^2(\T))}.
\end{align}
\end{proposition}

Before to show this result, we need to establish a technical lemma.

\begin{lemma}\label{lemma techn}
Let the assumptions of Theorem~\ref{convergence} hold. Then, for any positive integer $M \leq N_m$, any $k =(k_1,k_2) \in \Z^2$ with $|k_1|, \, |k_2| < M$, and any $0\leq n \leq (N_T)_m-1 = T/\Delta t_m -1$, it holds
\begin{align*}
\left|c_k(\pa_{x_1} \rho^{\per,m}_M(t,\cdot))\right| \leq \left|c^d_k(\theta^{n+1,\per,m}_{M,x_1})\right| + \left|c^d_k(\theta^{n,\per,m}_{M,x_1})\right|, \quad \forall  t \in [t_n,t_{n+1}],
\end{align*}
where $\theta^{n,\per,m}_{M,x_1} = \left(\theta^{n,\per}_{M,i+1/2,j}\right)_{(i,j)\in\I} = \left(\theta^{+,n,\per}_{M,i+1/2,j}- \theta^{-,n,\per}_{M,i+1/2,j}\right)_{(i,j)\in\I}$ for any $0\leq n \leq N_T$.
\end{lemma}

\begin{proof}
Let, $0\leq n\leq N_T-1$ and $(t,x_1,x_2) \in [t_n, t_{n+1}] \times \overline{C}_{i,j}$ for $(i,j)\in \I$, we have 
\begin{align*}
\partial_{x_1}\rho^{\pm,\per,m}_{M}(t, x_1, x_2) &= a^m_n(t) \,\left[a_j^m(x_2) \, \theta^{\pm,n + 1,\per}_{M,i +1/2, j + 1} + b_j^m(x_2)  \, \theta^{\pm, n + 1,\per}_{M,i +1/2, j}\right]\\
      &+ b^m_n(t) \,\left[ a^m_j(x_2) \, \theta^{\pm, n, \per}_{M,i +1/2, j + 1} + b^m_j(x_2)\,\theta^{\pm, n, \per}_{M,i +1/2, j}\right],
\end{align*}
where we recall definition~\eqref{def : ai} of the functions $a^m_j$ and $b^m_j$. Similarly, for any $t \in [t_n,t_{n+1}]$, we define
\[
a^m_n(t) = \frac{t-t_n}{\Delta t_m}, \quad b^m_n(t) = 1 - a^m_n(t).
\]
Therefore, for any $k =(k_1,k_2) \in \Z^2$ with $|k_1|, \, |k_2| < M$, we have
\begin{align*}
c_k\left(\partial_{x_1}\rho^{\pm,\per,m}_{M}(t, \cdot)\right) &= a^m_n(t) \sum_{(i,j)\in \I^2}\int_{C_{i,j}}\left[a^m_j(x_2)\,\theta^{\pm, n + 1, \per}_{M,i +1/2, j + 1} + b^m_j(x_2) \, \theta^{\pm, n + 1, \per}_{M,i + 1/2, j}\right]\, e^{- 2 i\pi  (k_1 x_1 + k_2 x_2)}\, dx_1 \, dx_2\\
      &+ b^m_n(t) \sum_{(i,j)\in \I^2}\int_{C_{i,j}} \left[ a^m_j(x_2) \, \theta^{\pm,n, \per}_{M,i +1/2, j + 1} + b^m_j(x_2) \, \theta^{\pm,n,\per}_{M,i +1/2, j}\right]\, e^{- 2 i\pi  (k_1 x_1 + k_2 x_2)} \, dx_1 \, dx_2.
\end{align*}
Now, using some change of variables and classical computations, we obtain
\begin{align*}
\Delta x_m \, c_k&\left(\partial_{x_1}\rho^{\pm,\per,m}_{M}(t, \cdot)\right)\\
&=  a^m_n(t)\left( \int_{0}^{\Delta x_m}e^{- 2 i \pi k_1 s} \,ds\right)\left( \int_{0}^{\Delta x_m} y\, e^{-2i \pi k_2 y}\, dy \right) \sum_{(i,j)\in\Imm^2} \theta^{\pm, n + 1,\per}_{M,i + 1/2, j + 1} e^{-2i \pi (k_1 x_i +k_2 x_j)} \\
&+ a^m_n(t) \left(\int_{0}^{\Delta x_m}e^{- 2i \pi k_1 s} \, ds\right)\left(\int_{0}^{\Delta x_m} (\Delta x_m - y)\, e^{-2i\pi k_2 y}\,d y\right) \sum_{(i,j) \in \Imm^2} \theta^{\pm,n + 1, \per}_{M,i + 1/2, j} e^{-2i \pi (k_1 x_i + k_2 x_j)}\\
&+ b^m_n(t) \left(\int_{0}^{\Delta x_m}e^{- 2 i\pi k_1 s} \, ds\right)\left(\int_{0}^{\Delta x_m} y\, e^{-2i \pi k_2 y}\, dy\right) \sum_{(i,j)\in \Imm^2} \theta^{\pm, n,\per}_{M, i +1/2, j + 1} e^{-2i \pi(k_1 x_i +  k_2 x_j)}\\
&+ b^m_n(t) \left(\int_{0}^{\Delta x_m}e^{- 2i \pi k_1 s} \, ds\right)\left(\int_{0}^{\Delta x_m} (\Delta x_m - y)\, e^{-2 i\pi k_2 y}\, dy \right)\sum_{(i,j) \in \Imm^2} \theta^{\pm, n, \per}_{M,i +1/2, j} e^{-2i \pi (k_1 x_i + k_2 x_j)}.
\end{align*}
Besides, for any $0 \leq n \leq (N_T)_m-1$, we observe that
\begin{align*}
\sum_{(i,j)\in \I^2} \theta^{\pm, n, \per}_{M,i +1/2, j + 1} \, e^{-2 i\pi (k_1 x_i +k_2 x_j)} =e^{2 i\pi k_2 \Delta x_m} \sum_{(i,j)\in \I^2} \theta^{\pm,n, \per}_{M, i +1/2, j} \, e^{-2 i\pi (k_1 x_i +k_2 x_j)}.
\end{align*}
Hence, we end up with
\begin{align*}
\Delta x_m \, c_k&\left(\partial_{x_1}\rho^{\per,m}_{M}(t, \cdot)\right)\\
&=  a^m_n(t) e^{2 i\pi k_2 \Delta x_m} \left( \int_{0}^{\Delta x_m}e^{- 2 i \pi k_1 s} \,ds\right)\left( \int_{0}^{\Delta x_m} y\, e^{-2i \pi k_2 y}\, dy \right) \sum_{(i,j)\in\Imm^2} \theta^{n + 1,\per}_{M,i + 1/2, j} e^{-2i \pi (k_1 x_i +k_2 x_j)} \\
&+ a^m_n(t)  \left(\int_{0}^{\Delta x_m}e^{- 2i \pi k_1 s} \, ds\right)\left(\int_{0}^{\Delta x_m} (\Delta x_m - y)\, e^{-2i\pi k_2 y}\,d y\right) \sum_{(i,j) \in \Imm^2} \theta^{n + 1, \per}_{M,i + 1/2, j} e^{-2i \pi (k_1 x_i + k_2 x_j)}\\
&+ b^m_n(t) e^{2 i\pi k_2 \Delta x_m} \left(\int_{0}^{\Delta x_m}e^{- 2 i\pi k_1 s} \, ds\right)\left(\int_{0}^{\Delta x_m} y\, e^{-2i \pi k_2 y}\, dy\right) \sum_{(i,j)\in \Imm^2} \theta^{n,\per}_{M, i +1/2, j} e^{-2i \pi(k_1 x_i +  k_2 x_j)}\\
&+ b^m_n(t) \left(\int_{0}^{\Delta x_m}e^{- 2i \pi k_1 s} \, ds\right)\left(\int_{0}^{\Delta x_m} (\Delta x_m - y)\, e^{-2 i\pi k_2 y}\, dy \right)\sum_{(i,j) \in \Imm^2} \theta^{\pm, n, \per}_{M,i +1/2, j} e^{-2i \pi (k_1 x_i + k_2 x_j)}.
\end{align*}
It remains to notice that
\begin{align*}
\left| \int_{0}^{\Delta x_m} y\, e^{-2i \pi k_2 y}\, dy \right| \le \int_{0}^{\Delta x_m} y \, dy = \frac{(\Delta x_m)^2}{2},
\end{align*}
and similarly
\begin{align*}
\left| \int_{0}^{\Delta x_m} (\Delta x_m - y) \, e^{-2\pi i k_2 y}\,\dd y \right| \le  \frac{(\Delta x_m)^2}{2},
\end{align*}
so that
\begin{align*}
c_{k}\left(\partial_{x_1}(\rho^{\per, m}_{M}(t,\cdot)\right) &\le \left|\sum_{(i,j)\in \Imm^2} \left(\Delta x_m\right)^2 \, \theta^{n + 1, \per}_{M,i +1/2, j} \, e^{-2 i\pi (k_1 x_i + k_2 x_j)} \right| \\
 &+ \left|\sum_{(i,j)\in \Imm^2} \left(\Delta x_m\right)^2 \, \theta^{n,\per}_{M, i +1/2, j} \, e^{-2i\pi (k_1 x_i + k_2 x_j)}  \right|.
\end{align*}
This completes the proof of Lemma~\ref{lemma techn}.
\end{proof}

\begin{proof}[Proof of Proposition~\ref{prop velocity field H1}]
Let us first prove that
\begin{align*}
\pa_{x_1} \left(\sigma_M^\K\ast \rho^{\per,m}_M\right) \in L^2((0,T)\times \T),
\end{align*}
with an explicit estimate. Then, let $t \in (0,T)$ be given, thanks to Parseval's equality and the relation $\sigma_M^\K = (F_M^2 \ast \K)$, we have
\begin{align*}
  \|\partial_{x_1} \left(\sigma_{M}^\K(\cdot) \ast \rho^{\per,m}_M(t,\cdot)\right)\|^2_{L^2(\T)} = \sum_{k \in \mathbb{Z}^2}
\left| c_k( \sigma_M^\K )\, c_k(\partial_{x_1}\rho^{\per,m}_M(t,\cdot) ) \right|^2 = \sum_{k \in \mathbb{Z}^2} \left|c_k(F_M^2)\, c_k(\K) \, c_k(\pa_{x_1} \rho^{\per,m}_M(t,\cdot)) \right|^2.
\end{align*}
In fact, the previous sum is finite. Indeed, we notice that for any $k = (k_1,k_2) \in \Z^2$ with $|k_1| \geq M$ or $|k_2|\geq M$, it holds $c_k(F_m^2)=0$. Besides, for any $|k_1|< M$ and $|k_2|< M$ we have $c_k(F_M^2) \leq 1$ and $c_k(\K) \leq 1$. Therefore, we obtain
\begin{align*}
\|\partial_{x_1} \left(\sigma_{M}^\K(\cdot) \ast \rho^{\per,m}_M(t,\cdot)\right)\|^2_{L^2(\T)} \leq \sum_{|k_1|, |k_2| < M}  c_k(\K) \,\left| c_k(\pa_{x_1} \rho^{\per,m}_M(t,\cdot)) \right|^2.
\end{align*}
Now, according to equality~\eqref{coeff discret sigmaM} we have $0 \leq c_k(\K) \leq c_k^d(\overline{\sigma}_M)$. Hence, applying Lemma~\ref{lemma techn} we deduce that
\begin{align*}
\|\partial_{x_1} \left(\sigma_{M}^\K(\cdot) \ast \rho^{\per,m}_M(t,\cdot)\right)\|^2_{L^2(\T)} &\leq 2\sum_{0\le k_1, k_2 < M}  c^d_k(\overline{\sigma}_M) \,\left(\left| c^d_k(\theta^{n+1,\per,m}_{M,x_1})\right|+\left| c^d_k(\theta^{n,\per,m}_{M,x_1})\right| \right)^2\\
&\leq 4 \sum_{k \in \Imm^2}  c^d_k(\overline{\sigma}_M) \,\left(\left| c^d_k(\theta^{n+1,\per,m}_{M,x_1})\right|^2+\left| c^d_k(\theta^{n,\per,m}_{M,x_1})\right|^2 \right).
\end{align*}
Recalling definition~\eqref{discrete dissipation} of the discrete dissipation term $\mathcal{D}\left[\rho^{n,\per}_M\right]$, we have
\begin{align*}
\|\partial_{x_1} \left(\sigma_{M}^\K(\cdot) \ast \rho^{\per,m}_M(t,\cdot)\right)\|^2_{L^2(\T)} \leq 4 \left(\mathcal{D}\left[\rho^{n+1,\per}_M\right] + \mathcal{D}\left[\rho^{n,\per}_M\right]\right),
\end{align*}
so that
\begin{align*}
\int_0^T \|\partial_{x_1} \left(\sigma_{M}^\K(\cdot) \ast \rho^{\per,m}_M(t,\cdot)\right)\|^2_{L^2(\T)} \, dt &\leq 4 \sum_{n=0}^{(N_T)_m-1} \Delta t_m \,\left( \mathcal{D}\left[\rho^{n+1,\per}_M\right] + \mathcal{D}\left[\rho^{n,\per}_M\right] \right) \\
&\leq 8\sum_{n=1}^{(N_T)_m} \Delta t_m \, \mathcal{D}\left[\rho^{n,\per}_M\right] + 4 \, \eps_m \,\mathcal{D}\left[\rho^{0,\per}_M\right].
\end{align*}
The first term in the r.h.s.~is bounded thanks to the discrete gradient entropy estimate~\eqref{gradient entropy estimate} together with~\eqref{bound entropy init}. Concerning the second term, reproducing the computations of Section~\ref{section gradient}, we notice that
\begin{align*}
\mathcal{D}\left[\rho^{0,\per}_M\right] &=  \sum_{(i,j)\in\Imm^2} (\Delta x_m)^2 \, \, \theta_{M,i +1/2,j}^{0,\per} \sum_{(\ell,r)\in\Imm^2} (\Delta x_m)^2 \,\, \sigma_{M,\ell,r}^\K\, \theta^{0,\per}_{M,i+1/2-\ell,j-r}\\
&\leq M^2 \left(\|\rho^{+,\per}_{M,0}\|_{W^{1,\infty}(\T)} + \|\rho^{-,\per}_{M,0}\|_{W^{1,\infty}(\T)} \right)^2.
\end{align*} 
Hence, we deduce the existence of a constant $C_M> 0$ independent of $m$ such that
\begin{align*}
\int_0^T \|\partial_{x_1} \left(\sigma_{M}^\K(\cdot) \ast \rho^{\per,m}_M(t,\cdot)\right)\|^2_{L^2(\T)} \, dt
&\leq 8\sum_{n=1}^{(N_T)_m} \Delta t_m \, \mathcal{D}\left[\rho^{n,\per}_M\right] + \eps_m \, C_M.
\end{align*}
Let us now study the $x_2$-derivative of $\sigma_M \ast \rho^{\per,m}_M$. For this, using once more Parseval's equality, we have, for $t \in (0,T)$,
\begin{align*}
  \|\partial_{x_2} \left(\sigma_{M}^\K(\cdot) \ast \rho^{\per,m}_M(t,\cdot)\right)\|^2_{L^2(\T)} &= \sum_{k \in \Z^2} \left|(2i \pi k_2) \, c_k\left( \sigma_M^\K\right) \, c_k\left(\rho^{\per,m}_M(t,\cdot)\right) \right|^2\\
  &= \sum_{k \in \Z^2} \left|(2i \pi k_2) \, c_k(F_M^2) \, c_k(\K) \, c_k(\rho^{\per,m}_M(t,\cdot)) \right|^2\\
  &=  \sum_{k \in \Z^2} \left|(2i \pi) \, c_k(F_M^2) \, \dfrac{k_1^2 \, k_2^3}{|k|^4} \, c_k(\rho^{\per,m}_M(t,\cdot)) \right|^2\\
 &\le \sum_{k \in \Z^2} \left|c_k(F_M^2) \, \dfrac{k_1 \, k_2^3}{|k|^4} \, c_k(\pa_{x_1}
\rho^{\per,m}_M(t,\cdot)) \right|^2.
\end{align*}
Now, we observe that
\begin{align*}
 \|\partial_{x_2} \left(\sigma_{M}^\K(\cdot) \ast \rho^{\per,m}_M(t,\cdot)\right)\|^2_{L^2(\T)} &= \sum_{|k_1|,\,|k_2|<M} c_k(F_M^2) \, \dfrac{k_2^4}{|k|^4} \, c_k(F_M^2) \, \dfrac{k_1^2 k_2^2}{|k|^4} \, \left|c_k(\pa_{x_1}\rho^{\per,m}_M(t,\cdot)) \right|^2\\
 &\leq \sum_{|k_1|,\,|k_2|<M} c_k(\K) \, \left|c_k(\pa_{x_1}\rho^{\per,m}_M(t,\cdot)) \right|^2
\end{align*}
Therefore, we conclude that
\begin{align*}
\int_0^T \|\partial_{x_2} \left(\sigma_{M}^\K(\cdot) \ast \rho^{\per,m}_M(t,\cdot)\right)\|^2_{L^2(\T)} \, dt \leq 8\sum_{n=1}^{N_T} \Delta t_m \, \mathcal{D}\left[\rho^{n,\per}_M\right] + \eps_m \, C_M.
\end{align*}
Finally, arguing as before, we get, for $t \in (0,T)$,
\begin{align*}
\|\left(\sigma_{M}^\K(\cdot) \ast \rho^{\per,m}_M(t,\cdot)\right)\|^2_{L^2(\T)} = \sum_{k \in \Z^2} \left|c_k(\sigma_M^\K) \, c_k(\rho^{\per,m}_M(t,\cdot))\right|^2 \leq \sum_{k \in \Z^2} \left|c_k(\rho^{\per,m}_M(t,\cdot))\right|^2 = \|\rho^{\per,m}_M(t,\cdot)\|^2_{L^2(\T)}
\end{align*}
This last quantity is uniformly bounded w.r.t.~$m$ thanks to Corollary~\ref{corollary L2 bound}.
\end{proof}

\subsection{Proof of Theorem~\ref{theorem-final convergence}}

In this section, we show that the solution $\rho^{\pm,\per}_\ell = \rho^{\pm,\per}_{M_\ell}$ of the system~\eqref{original1}-\eqref{IC.reg.per} (given by Theorem \ref{convergence}) converges, up to a subsequence, to a solution of the main system~\eqref{original.per}--\eqref{IC.original.per} in the distributional sense. To do so, we proceed as in the proof of~\cite[Theorem 1.4]{cannone2010global}. We prove that the uniform discrete estimates established previously, in particular~\eqref{gradient entropy estimate} and~\eqref{norme-h1-vitesse}--\eqref{norme-l2-vitesse}, provide sufficient compactness to pass to the limit in the bilinear term of the system 
\begin{equation}\label{bil-ter}
\left(\sigma_{M_\ell}^\K \ast (\rho^{+,\per}_\ell- \rho^{-,\per}_\ell) \right) \partial_{x_1} \rho^{\pm,\per}_\ell. 
\end{equation}
This allows passing to the limit in the system \eqref{original1}-\eqref{IC.reg.per}, 
given that the other terms are much simpler to handle.

First, thanks to the discrete entropy estimate \eqref{gradient entropy estimate} and 
the discrete $L^2$ estimate~\eqref{estimation_L2_uni}, we can show, 
using~\eqref{norm} and some convexity properties (as in~\cite[Subsection 5.1]{ZHIZ25}), that there exists a constant $C_1$ independent of $M_\ell$ and $m$ such that 
\begin{equation*}
\|\rho^{\pm,\per,m}_\ell\|_{L^\infty(0,T; L^2(\T))}+ 
\|\partial_{x_1} \rho^{\pm,\per,m}_\ell\|_{L^\infty(0,T; L\log L(\T))} \le C_1. 
\end{equation*}
By the lower semi-continuity of the weak and weak-* topologies, we deduce that the function 
$\rho^{\pm,\per}_\ell$, defined as the limit of a subsequence of  $\rho^{\pm,\per,m}_\ell$, also satisfies the same estimate, namely  
\begin{equation}\label{estimation-m-M}
\|\rho^{\pm,\per}_\ell\|_{L^\infty(0,T; L^2(\T))}+ 
\|\partial_{x_1} \rho^{\pm,\per}_\ell\|_{L^\infty(0,T; L\log L(\T))} \le C_1. 
\end{equation}
Hence, there exists a subsequence of $(\rho^{\pm,\per}_{\ell})_{\ell\in\N}$, which converges to a function $\rho^{\pm,\per}$ in the following sense:  
\begin{align*}
\rho^{\pm, \per}_\ell &\rightharpoonup \rho^{\pm, \per} \quad \mbox{weakly in } L^2((0,T) \times \mathbb{T}^2),\\
\partial_{x_1} \rho^{\pm, \per}_\ell &\overset{*}{\rightharpoonup} \partial_{x_1} \rho^{\pm, \per} \quad \mbox{in } L^\infty((0,T); L \log L(\mathbb{T}^2)). 
\end{align*}
By Fejér's theorem and the $L^2$ norm preservation by the kernel $\K$, we then deduce that 
$$
\sigma_{M_\ell}^\K \ast (\rho^{+, \per}_\ell - \rho^{-,\per}_\ell) \rightharpoonup  \K\ast (\rho^{+,\per} - \rho^{-,\per}) \quad \mbox{weakly in } L^2((0,T) \times \mathbb{T}^2).
$$
Furthermore, using a similar argument based on the discrete $H^1$ estimates~\eqref{norme-h1-vitesse}--\eqref{norme-l2-vitesse}, we control the velocity uniformly with respect to $M_\ell$ in $L^2(0,T; H^1(\T))$. This establishes the existence of a constant $C_2$ independent of $M_\ell$ such that
\begin{equation}\label{estimation-v-M}
\|\sigma_{M_\ell}^\K \ast (\rho^{+,\per}_{\ell}-\rho^{-,\per}_{\ell}) \|_{L^2(0,T; H^1(\T))}
 \le C_2.
\end{equation}
Moreover, since $\rho^{\pm,\per}_{\ell}$ is a 
solution of the system~\eqref{original1}--\eqref{IC.reg.per} in  $\mathcal{D}'((0,T)\times\T)$, the estimates~\eqref{estimation-m-M}--\eqref{estimation-v-M} show (as in the proof of~\cite[Lemma 5.8]{cannone2010global}) that there exists a constant $C_3$ independent of $M_\ell$ such that
\begin{equation}\label{estimation-t-M} 
\|\partial_{t}\rho^{\pm, \per}_\ell \|_{L^2(0,T; H^{-1}(\T))}  + 
\|\partial_{t} \sigma_{M_\ell}^\K \ast  (\rho^{+,\per}_{\ell}-\rho^{-,\per}_{\ell}) \|_{L^2(0,T; H^{-1}(\T))} \le C_3. 
\end{equation} 
Now, we apply Simon's lemma (recalled in Appendix~\ref{app2}) in the particular setting $B=EXP_\beta(\T)$ ($1\le \beta <2$), $X=H^1(\T)$, and $Y=H^{-1}(\T)$, together with lemmas~\ref{EC:lem:comp} and~\ref{EC:weak}, to establish, using estimates~\eqref{estimation-v-M}--\eqref{estimation-t-M}, that 
\begin{align*}
\sigma_{M_\ell}^\K \ast (\rho^{+, \per}_\ell - \rho^{-,\per}_\ell) &\rightarrow \K\ast (\rho^{+,\per} - \rho^{-,\per}) \quad \mbox{strongly in } L^1(0,T;E_\exp(\T)),
\end{align*}
where the space $E_\exp(\T)$ is defined as the dual of $L\log L(\T)$ (cf. Appendix~\ref{app2}). This allows us to define the limit of the bilinear term~\eqref{bil-ter} in the sense of 
\[
L^{1}((0,T); E_{exp}(\T))\text{-strong} \;\times\; L^{\infty}((0,T); L\log L(\T))\text{-weak-*},
\]
and consequently, 
\[ 
(\sigma_{M_\ell}^\K \ast  (\rho^{+,\per}_{\ell}-\rho^{-,\per}_{\ell})) {\partial_{x_1}}\rho^{\pm, \per}_\ell  
    \rightarrow (\K \ast  (\rho^{+,\per}-\rho^{-,\per})) {\partial_{x_1}}\rho^{\pm, \per}
    \;\;\; \mbox{in }\mathcal{D}'(\T\times(0,T)).
\]
In what precedes, we have shown that $\rho^{\pm,\per}$
are solutions of the system \eqref{original.per}. To conclude, we use estimate \eqref{estimation-t-M}, which  provides sufficient 
temporal continuity to verify the initial data \eqref{IC.original.per}.

\section{Numerical experiments}\label{section-numerical}
This section is devoted to presenting the numerical results of our simulations, which illustrate the evolution of the solution $\rho^{\pm, \per}$ of the system \eqref{original}. The objective is to examine whether the numerical results obtained with the proposed scheme agree with the physical observations. From a physical perspective, there are three different mechanisms that can influence the evolution of dislocations.

The first mechanism is the effect of an external stress applied to the material, which in our model is represented by \( a(t) \). It is well known that, when only an external stress is applied, the material will eventually deform. However, an important question arises: does this deformation evolve toward a stationary and stable solution, or not?

A second mechanism affecting the evolution of dislocations, which may not be immediately apparent, is the internal constraint generated by the dislocations themselves. In our model, this internal constraint is expressed by the non-local and non-linear term.

The third case occurs when the material is simultaneously subjected to both internal and external stresses. This is precisely the situation considered in our study, where both types of stress act together to influence material behavior. In our simulations, we consider two different cases. In the first case, we assume that there is no internal stress. Accordingly, we set the external stress
$$a(t) = 3t,$$
together with the following initial data (in the form of a Gaussian distribution)
\[
\rho^{\pm,\mathrm{per}}_0(x_1,x_2) 
= \frac{1}{6}
e^{- 10 \left(
\min\bigl\{|x_1-\frac{1}{2}|,\;1-|x_1-\frac{1}{2}|\bigr\}^2
+
\min\bigl\{|x_2- \frac{1}{2}|,\;1-|x_2- \frac{1}{2}|\bigr\}^2
\right)},\quad \forall (x_1,x_2)\in[0,1]^2, 
\]
extended with a period of $1$. After choosing the  numerical parameters given in Table~\ref{table}, we obtain the results shown in Figure~\ref{fig:comparison_three}. This figure displays only the evolution of the density \( \partial_{x_1}\rho^{+} \). The plots of \( \partial_{x_1}\rho^{-} \) are not shown, since they are equal to those of \( \partial_{x_1}\rho^{+} \). From Figure~\ref{Initial rho+}, we observe that the density initially exhibits large variations. As time increases, these variations gradually decrease, and the density eventually approaches a uniform distribution equivalent to the total density $L$, as illustrated in the Figure \ref{Final evolution, rho+}. 
This behavior is consistent with our expectations, as the presence of an external stress clearly affects the evolution of dislocations, which will be distributed throughout the material over the long term

\begin{center}
 \begin{tabular}{|*{4}{c|}}
     \hline
    \quad $M = N$ \quad  &  \quad $L$  \quad &  \quad $N_T$  \\ \hline 
      \quad $50$  \quad &  \quad $1$  \quad & \quad $200$  \quad\\ \hline
   \end{tabular}
        \captionof{table}{Used parameters in  the Figure \ref{fig:comparison_three}}\label{table} 
   \end{center}
\begin{figure}[ht!]
    \centering

    \begin{subfigure}[t]{0.3\textwidth}
        \centering
        \includegraphics[width=\linewidth]{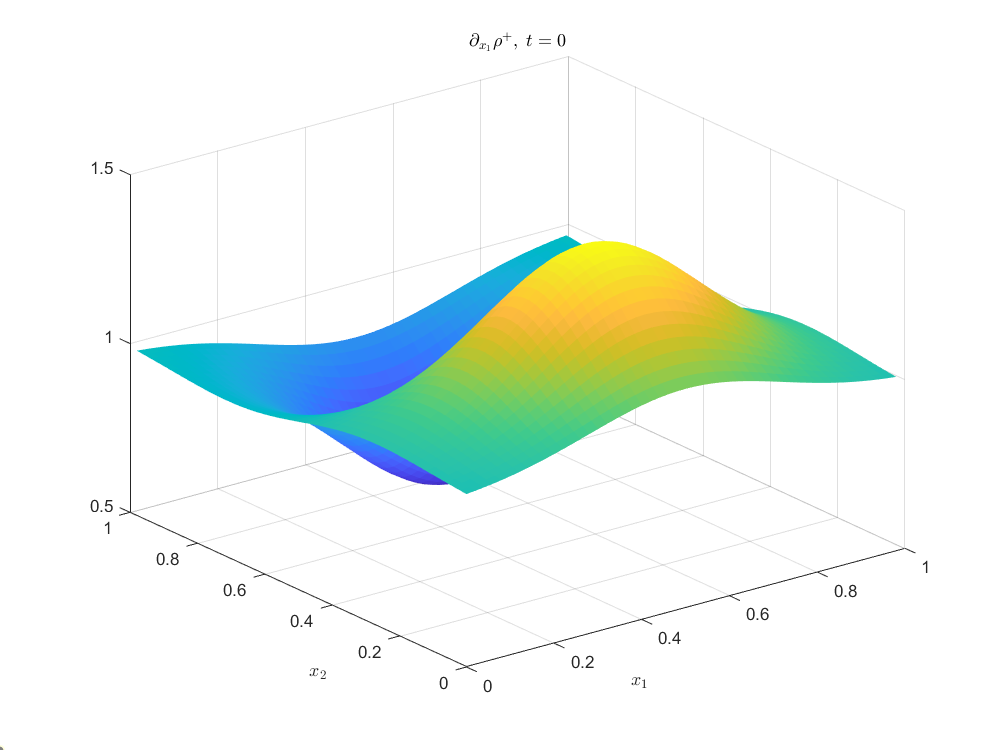}
        \caption{Initial evolution, $t = 0$}
        \label{Initial rho+}
    \end{subfigure}
    \hfill
    \begin{subfigure}[t]{0.3\textwidth}
        \centering
        \includegraphics[width=\linewidth]{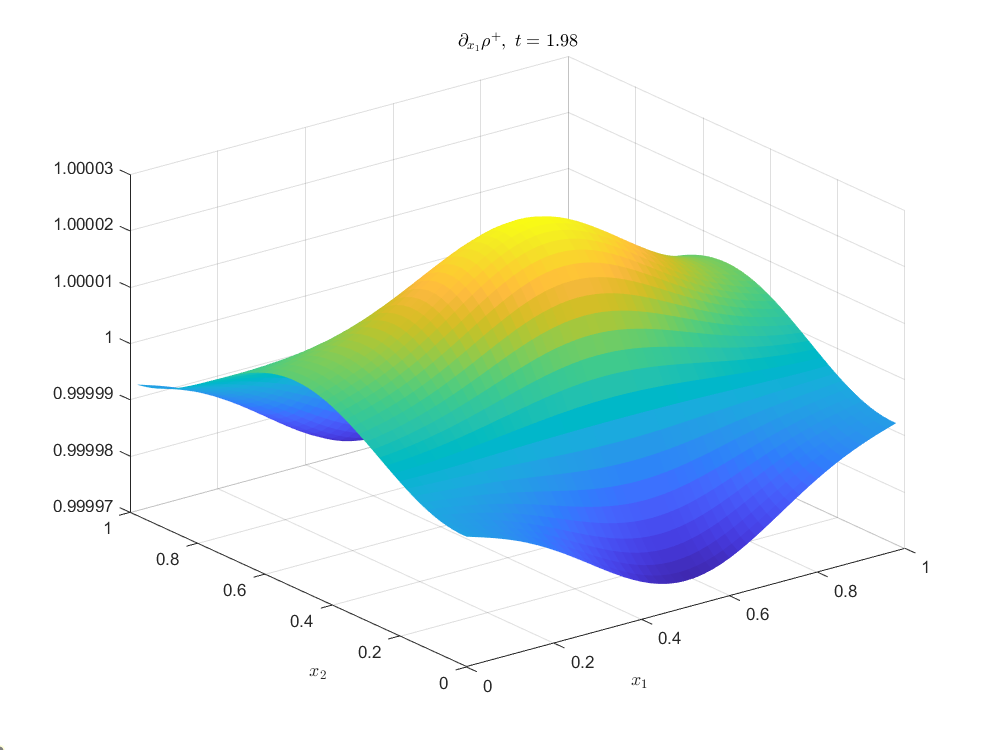}
        \caption{Intermediate evolution, t = 1.98}
        \label{Initial_data_c_pos}
    \end{subfigure}
    \hfill
    \begin{subfigure}[t]{0.3\textwidth}
        \centering
        \includegraphics[width=\linewidth]{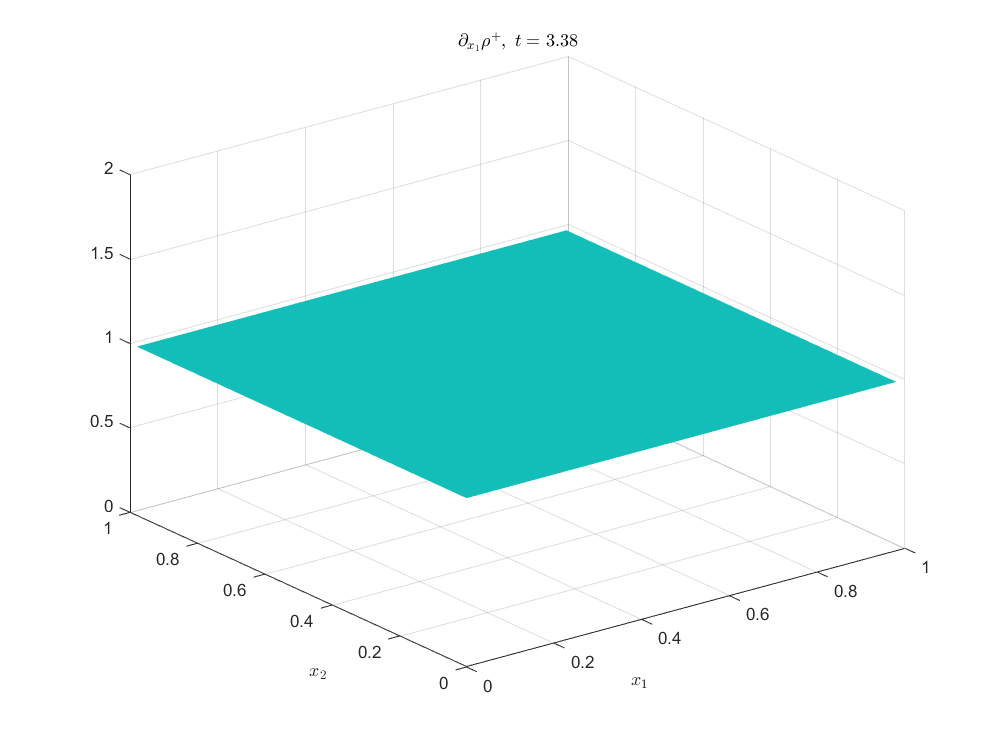}
        \caption{Final evolution, t = 3.38}
        \label{Final evolution, rho+}
    \end{subfigure}

    \caption{Evolution of $\partial_{x_1}\rho^{+}$ under external stress.}
    \label{fig:comparison_three}
\end{figure}
In the second case, we also consider the parameters given in in Table~\ref{table} and we use the same external stress ($a(t) = 3t$), but this time we also take into account the internal stress
generated by the dislocations. For this, we keep the same initial data for
$\rho^{+,\,\mathrm{per}}_0$
and we consider the case where $\rho^{-,\mathrm{per}}_0= \frac{1}{2}\rho^{+,\mathrm{per}}_0$.

By observing the numerical results shown in Figure \ref{fig:comparison_six}, we notice that the same analytical behavior is obtained as in the previous case. Initially, as illustrated in Figures \ref{fig:rho_init} and \ref{fig:rho_m_init},  the dislocation density exhibits a concentration of dislocations within the material. 
Then, similarly to the previously discussed case in which only the external stress was applied, the evolution of the dislocation density becomes smoother and converges toward the total density $L$ (Figures \ref{fig:rho_final} and \ref{fig:rho_m_final}), reflecting a uniform distribution of dislocations across the material. This behavior is natural and expected, since the effect of the external stress is typically much stronger than that of the internal stress, and it therefore dominates the evolution.
\begin{figure}[ht!]
    \centering

    \begin{subfigure}[t]{0.3\textwidth}
        \centering
        \includegraphics[width=\linewidth]{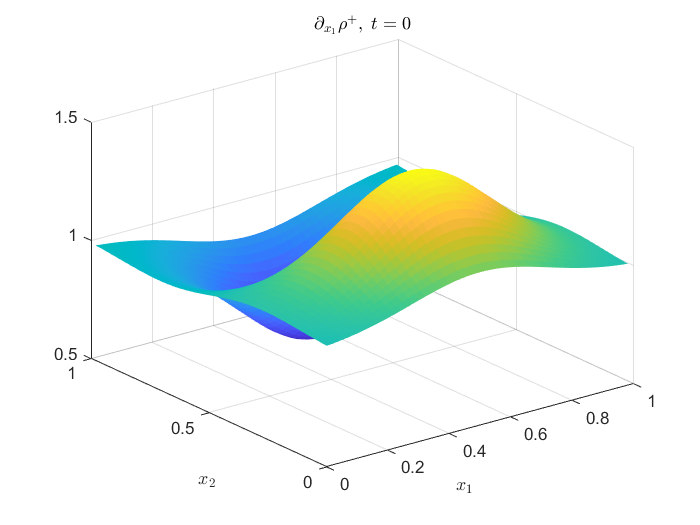}
        \caption{Initial evolution of $\partial_{x_1}\rho^+$, $t = 0$}
        \label{fig:rho_init}
    \end{subfigure}
    \hfill
    \begin{subfigure}[t]{0.3\textwidth}
        \centering
        \includegraphics[width=\linewidth]{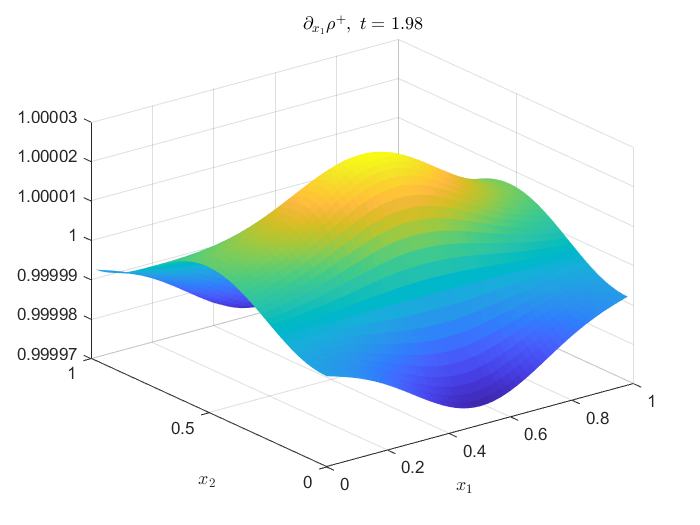}
        \caption{Intermediate evolution of $\partial_{x_1}\rho^+$, \\ $t = 1.98$}
        \label{fig:rho_mid}
    \end{subfigure}
    \hfill
    \begin{subfigure}[t]{0.3\textwidth}
        \centering
        \includegraphics[width=\linewidth]{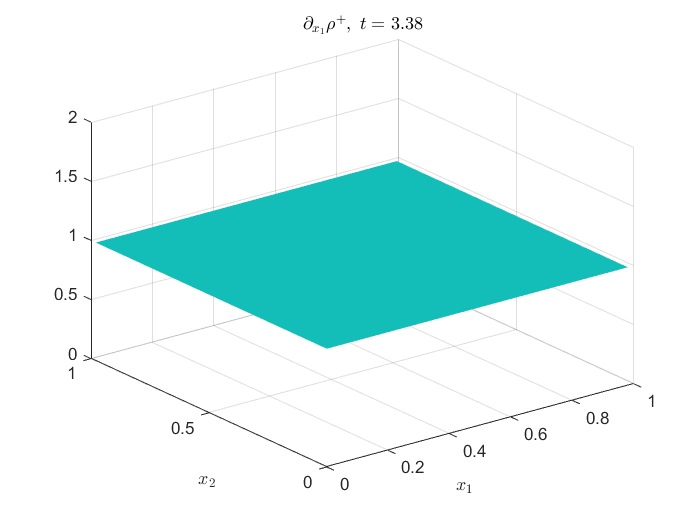}
        \caption{Final evolution of $\partial_{x_1}\rho^+$, \\$t = 3.38$}
        \label{fig:rho_final}
    \end{subfigure}

    \par\medskip 

    \begin{subfigure}[t]{0.3\textwidth}
        \centering
        \includegraphics[width=\linewidth]{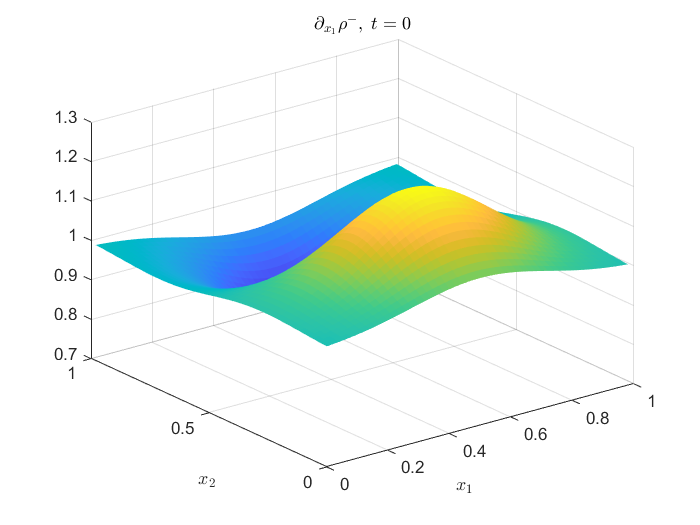}
        \caption{Initial evolution of $\partial_{x_1}\rho^-$, $t = 0$}
        \label{fig:rho_m_init}
    \end{subfigure}
    \hfill
    \begin{subfigure}[t]{0.3\textwidth}
        \centering
        \includegraphics[width=\linewidth]{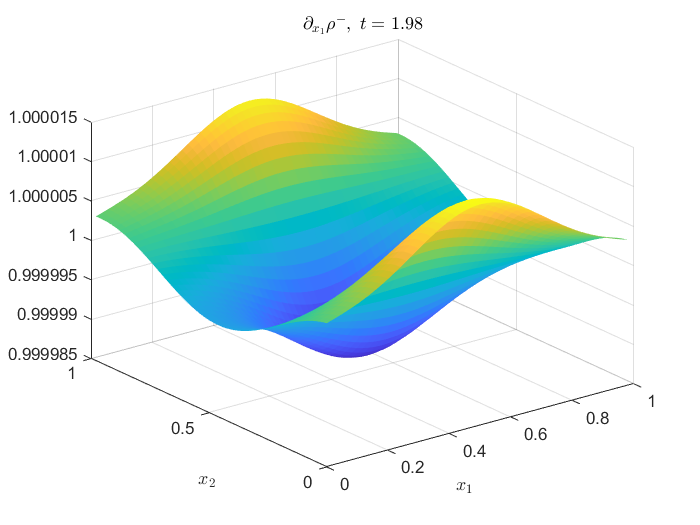}
        \caption{Intermediate evolution of $\partial_{x_1}\rho^-$, $t = 1.98$}
        \label{fig:rho_m_mid}
    \end{subfigure}
    \hfill
    \begin{subfigure}[t]{0.3\textwidth}
        \centering
        \includegraphics[width=\linewidth]{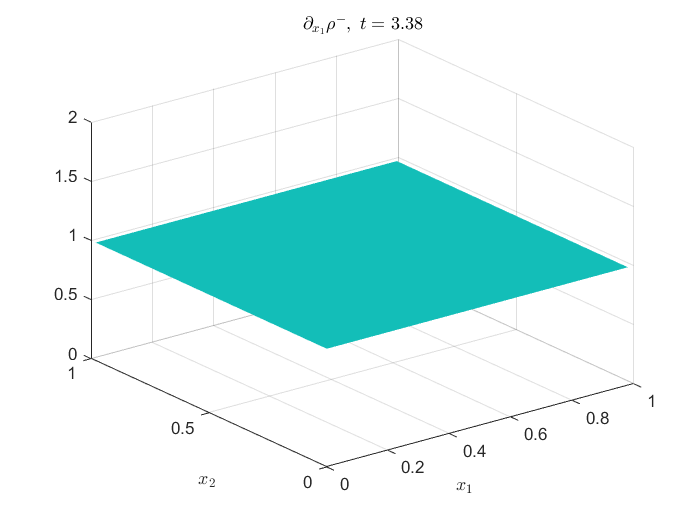}
        \caption{Final evolution of $\partial_{x_1}\rho^-$,\\ $t = 3.38$}
        \label{fig:rho_m_final}
    \end{subfigure}

    \caption{Evolution of $\partial_{x_1}\rho^{\pm}$  under internal and external stresses.}
    \label{fig:comparison_six}
\end{figure}

The numerical results obtained are consistent with our predictions. Starting from an initial concentration of dislocations, they move under the influence of an external stress until they are uniformly distributed throughout the material. These observations confirm that the proposed scheme effectively describes the evolution of dislocations toward equilibrium, under the influence of non-local interactions in a periodic domain.

\begin{appendix}

\section{Definition of some Orlicz spaces}\label{app}
In this section, we recall the definitions of Orlicz spaces, specifying those used in this paper. We refer the reader to \cite[Ch. 8]{111} and \cite{Rao} for further details.  
We begin by introducing the notion of a Young function.
\begin{definition}{\bf(Young
function,  \cite{Neil})}
A real-valued function $\displaystyle{A
  :[0,+\infty)\rightarrow \R}$ is called a Young
function if it satisfies the following properties:
\begin{itemize}
\item $A$ is continuous, non-negative, non-decreasing, and convex.
\item $A(0)=0$ and $\displaystyle{\lim_{t\rightarrow
    +\infty}A(t)=+\infty}$.
\end{itemize}
\end{definition}
\noindent Let $A(\cdot)$ be a Young function. The Orlicz class $K_A(\T)$ is the
set of (equivalence classes of) real-valued measurable functions $h$ on
$\T$ satisfying
$$\displaystyle{\int_{\T} A(|h(x)|)< +\infty}.$$
The Orlicz space $L_A(\T)$ is the linear hull of
$K_A(\T)$, endowed with the Luxemburg norm
$$\|f\|_{L_A(\T)}=\inf\left\{\lambda>0:\displaystyle{\int_{\T}}
A\left(\frac {|h(x)|}{\lambda}\right)\le 1\right\}.$$
Equipped with this norm, the Orlicz space $L_A(\T)$ is a Banach
space. Moreover, for all $f\in L_A(\T)$, the following estimate holds:
\begin{equation}\label{norm}
\|f\|_{L_A(\T)}\le 1+\int_{\T}A(|f(x)|).
\end{equation}

\begin{definition}{\bf(Some Orlicz spaces)}\\
\noindent $\bullet\;\;\mbox{$EXP_{\a}(\T)$ denotes the Orlicz space
defined by the function $A(t)=e^{t^\a}-1$, for $\a \ge 1$}.$\\
\noindent $\bullet\;\;\mbox{$L\log^{\beta}L(\T)$ denotes the Orlicz space
defined by the function
$A(t)=t(\log(e+t))^{\beta}$, for $\beta\ge 0$}.$
\end{definition}
Observe that for $0<\beta\le 1$, the space $EXP_{\frac 1\beta}(\T)$ is the dual of
the Zygmund space $L\log^{\beta}L(\T)$. It is also worth noting that
$L\log^1L(\T)=L\log L(\T)$.
\section{Compactness results and useful embeddings}\label{app2}
In this section, we recall some embeddings and compactness results.  
The first one is Trudinger's embedding, which provides the compactness of 
$H^1(\T)$ in an Orlicz space.
\begin{lemma}\label{EC:lem:comp}{\bf (Trudinger compact embedding, \cite{Trud})}\\
For all $1\le \beta \le 2$, we have  
$$H^1(\T)\hookrightarrow EXP_\beta(\T).$$
\noindent Moreover, this injection is compact for all $1\le \beta <2$.
\end{lemma}
\noindent The second result allows us to verify the compactness of 
$H^1(\T)$ in a space whose dual is $L\log L(\T)$.
\begin{lemma}\label{EC:weak}{\bf (Weak-* topology in $L\log L$, \cite[Th 8.16, 8.18, 8.20]{111})}\\
Let $E_{exp}(\T)$ be the closure in $EXP(\T)$ of the space of bounded functions on $\T$. Then $E_{exp}(\T)$ is a separable
Banach space, which satisfies:
\begin{itemize}
\item $L\log L(\T)$ is the dual space of $E_{exp}(\T)$.\bigskip
\item $EXP_\beta(\T)\hookrightarrow E_{exp}(\T)\hookrightarrow EXP(\T)$ for all $\beta > 1$.
\end{itemize}
\end{lemma}

The third result is due to Simon and provides, under certain regularity assumptions, compactness in space-time.

\begin{lemma}\label{EC:simo}{\bf (Simon's Lemma, \cite[Th 6]{SI87})}\\
Let $X$, $B$, and $Y$ be three Banach spaces, where $X\hookrightarrow B$ with
compact embedding and $B\hookrightarrow Y$ with continuous
embedding. If $(\rho^n)_n$ is a sequence such that
$$\|\rho^n\|_{L^q((0,T); B)}+\|\rho^n\|_{L^1((0,T); X)}+
\left\|\displaystyle{\frac{\partial \rho^n}{\partial
    t}}\right\|_{L^1((0,T); Y)}\le C,
$$
where $q>1$ and $C$ is a constant independent of $n$, 
then $(\rho^n)_n$ is relatively compact in $L^p((0,T); B)$ for all $1\le
p<q$.
\end{lemma}

\end{appendix}

\bibliographystyle{plain}
\bibliography{bibliography}

@article{Trud,
  author  = {Trudinger, N. S.},
  title   = {On Imbeddings into Orlicz Spaces and Some Applications},
  journal = {Journal of Mathematics and Mechanics},
  volume  = {17},
  year    = {1967},
  pages   = {473--483},
}

@article{SI87,
  author  = {Simon, J.},
  title   = {Compact Sets in the Space ${L}^p(0,{T}; {B})$},
  journal = {Annali di Matematica Pura ed Applicata},
  volume  = {146},
  number  = {4},
  year    = {1987},
  pages   = {65--96},
}

@book{Rao,
  title     = {Theory of Orlicz Spaces},
  author    = {Rao, M. M. and Ren, Z. D.},
  series    = {Monographs and Textbooks in Pure and Applied Mathematics},
  volume    = {146},
  publisher = {Marcel Dekker, Inc.},
  address   = {New York},
  year      = {1991},
  isbn      = {0824784782},
  note      = {ix, 449 pages},
}

@article{Neil,
  author  = {O'Neil, R.},
  title   = {Fractional Integration in Orlicz Spaces. I},
  journal = {Transactions of the American Mathematical Society},
  volume  = {115},
  year    = {1965},
  pages   = {300--328},
}

@article{Meurs1, 
  title = {Global Existence and Mean‑Field Limit for a Stochastic Interacting Particle System of Signed Coulomb Charges},
  author = {van Meurs, P. and Peletier, M. A. and Slangen, T. },
  journal = {Potential Analysis},
  year = {2025},
  volume = {63},
  pages = {1699--1733},
  doi = {10.1007/s11118-025-10218-z},
  url = {https://doi.org/10.1007/s11118-025-10218-z}
}

@article{Meurs2,
  title = {Convergence and Non‑convergence of Many‑Particle Evolutions with Multiple Signs},
  author = {Garroni, A. and van Meurs, P. and Peletier, M. A. and Scardia, L. },
  journal = {Archive for Rational Mechanics and Analysis},
  volume = {235},
  number = {1},
  pages = {3--49},
  year = {2020},
  doi = {10.1007/s00205-019-01436-y},
  url = {https://doi.org/10.1007/s00205-019-01436-y}
}

@article{Meurs3,
  title = {Atomistic origins of continuum dislocation dynamics},
  author = {Hudson, T. and van Meurs, P. and Peletier, M. A.},
  journal = {Mathematical Models and Methods in Applied Sciences},
  volume = {30},
  number = {13},
  pages = {2557--2618},
  year = {2020},
  doi = {10.1142/S0218202520500505},
  url = {https://www.worldscientific.com/doi/10.1142/S0218202520500505}
}

@article{DiPernaLions1989,
  title = {Ordinary differential equations, transport theory and Sobolev spaces},
  author = {DiPerna, R.~J. and Lions, P.-L.},
  journal = {Inventiones Mathematicae},
  volume = {98},
  number = {3},
  pages = {511--547},
  year = {1989},
  publisher = {Springer},
  doi = {10.1007/BF01393835}
}

@article{Ambrosio2004,
  title = {Transport equation and Cauchy problem for BV vector fields},
  author = {Ambrosio, L. },
  journal = {Inventiones Mathematicae},
  volume = {158},
  number = {2},
  pages = {227--260},
  year = {2004},
  publisher = {Springer},
  doi = {10.1007/s00222-004-0367-2}
}

@article{AmbrosioSerfaty2008,
  title = {A gradient flow approach to an evolution problem arising in superconductivity},
  author = {Ambrosio, L. and Serfaty, S.},
  journal = {Communications on Pure and Applied Mathematics},
  volume = {61},
  number = {11},
  pages = {1495--1539},
  year = {2008},
  doi = {10.1002/cpa.20223}
}

@article{LiMiaoXue2014,
  title        = {On the well-posedness of a 2D nonlinear and nonlocal system arising from the dislocation dynamics},
  author       = {Li, D. and Miao, C. and Xue, L.},
  journal      = {Communications in Contemporary Mathematics},
  volume       = {16},
  number       = {02},
  pages        = {1350021},
  year         = {2014},
  doi          = {10.1142/S0219199713500219},
  url          = {https://arxiv.org/abs/1210.4072}
}

@article{WanChen2016GBLongtime,
  author       = {Wan, R. and Chen, J.-G.},
  title        = {Longtime well-posedness for the 2D Groma–Balogh model},
  journal      = {Journal of Nonlinear Science},
  volume       = {26},
  number       = {6},
  pages        = {1817--1831},
  year         = {2016},
  note         = {Often cited alongside the large background result, focuses on long-time behavior},
  url          = {https://link.springer.com/article/10.1007/s00332-016-9309-1}
}

@article{ElHajj2010,
  title        = {Short time existence and uniqueness in Hölder spaces for the 2D dynamics of dislocation densities},
  author       = {El Hajj, A. },
  journal      = {Annales de l'I.H.P. Analyse non linéaire},
  volume       = {27},
  number       = {1},
  pages        = {21--35},
  year         = {2010},
  doi          = {10.1016/j.anihpc.2009.03.003},
  url          = {https://www.numdam.org/item/AIHPC_2010__27_1_21_0/}
}

@article{ElHajjOussaily2021,
  title   = {Existence and Uniqueness of Continuous Solution for a Non-Local Coupled System Modeling the Dynamics of Dislocation Densities},
  author  = {El Hajj, A. and Oussaily, A.},
  journal = {Journal of Nonlinear Science},
  volume  = {31},
  number  = {1},
  pages   = {20},
  year    = {2021},
  doi     = {10.1007/s00332-021-09676-7},
  url     = {https://doi.org/10.1007/s00332-021-09676-7}
}

@article{AlZohbiElHajj2024,
  title   = {Existence and Uniqueness Results to a System of Hamilton–Jacobi Equations},
  author  = {Al Zohbi, M. and El Hajj, A.},
  journal = {La Matematica},
  volume  = {3},
  number  = {3},
  pages   = {1137--1161},
  year    = {2024}
}

@article{LaxWendroff1960,
  author  = {Lax, P. D. and Wendroff, B.},
  title   = {Systems of Conservation Laws},
  journal = {Communications on Pure and Applied Mathematics},
  volume  = {13},
  year    = {1960},
  pages   = {217--237}
}

@book{Bouchut2004,
  author    = {Bouchut, F.},
  title     = {Nonlinear Stability of Finite Volume Methods for Hyperbolic Conservation Laws and Well-Balanced Schemes for Sources},
  series    = {Frontiers in Mathematics},
  publisher = {Birkh{\"a}user Verlag},
  address   = {Basel},
  year      = {2004}
}

@book{Bressan2000,
  author    = {Bressan, A.},
  title     = {Hyperbolic Systems of Conservation Laws},
  subtitle  = {The One-Dimensional Cauchy Problem},
  series    = {Oxford Lecture Series in Mathematics and its Applications},
  volume    = {20},
  publisher = {Oxford University Press},
  address   = {Oxford},
  year      = {2000}
}

@article{Glimm1965,
  author  = {Glimm, J.},
  title   = {Solutions in the Large for Nonlinear Hyperbolic Systems of Equations},
  journal = {Communications on Pure and Applied Mathematics},
  volume  = {18},
  year    = {1965},
  pages   = {697--715}
}

@book{LeVeque2002,
  author    = {LeVeque, R. J.},
  title     = {Finite Volume Methods for Hyperbolic Problems},
  volume    = {31},
  series    = {Cambridge Texts in Applied Mathematics},
  publisher = {Cambridge University Press},
  year      = {2002}
}

@article{LeVequeTemple1985,
  author  = {LeVeque, R. J. and Temple, B.},
  title   = {Stability of Godunov's Method for a Class of $2 \times 2$ Systems of Conservation Laws},
  journal = {Transactions of the American Mathematical Society},
  volume  = {288},
  year    = {1985},
  pages   = {115--123}
}

@article{Temple1983a,
  author  = {Temple, B.},
  title   = {Systems of Conservation Laws with Coinciding Shock and Rarefaction Curves},
  journal = {Contemporary Mathematics},
  volume  = {17},
  year    = {1983},
  pages   = {143--151}
}

@article{Temple1983b,
  author  = {Temple, B.},
  title   = {Systems of Conservation Laws with Invariant Submanifolds},
  journal = {Transactions of the American Mathematical Society},
  volume  = {280},
  year    = {1983},
  pages   = {781--795}
}

@article{Liu1977,
  author  = {Liu, T.-P.},
  title   = {The Deterministic Version of the Glimm Scheme},
  journal = {Communications in Mathematical Physics},
  volume  = {57},
  year    = {1977},
  pages   = {135--148}
}

@article{AlvarezCarliniMonneauRouy2006,
  author  = {Alvarez, O. and Carlini, E. and Monneau, R.  and Rouy, E.},
  title   = {Convergence of a First Order Scheme for a Nonlocal Eikonal Equation},
  journal = {Applied Numerical Mathematics},
  volume  = {56},
  number  = {9},
  pages   = {1136--1146},
  year    = {2006},
  doi     = {10.1016/j.apnum.2006.03.002},
  note    = {IMACS Journal "Applied Numerical Mathematics"}
}

@article{ghorbel2010well,
  title={Well-posedness and numerical analysis of a one-dimensional non-local transport equation modelling dislocations dynamics},
  author={Ghorbel, A. and Monneau, R.},
  journal={Mathematics of Computation},
  volume={79},
  number={271},
  pages={1535--1564},
  year={2010},
  publisher={American Mathematical Society}
}

@book{operateursin,
 ISBN = {9780691080796},
 URL = {http://www.jstor.org/stable/j.ctt1bpmb07},
 author = {STEIN, E. M. },
 publisher = {Princeton University Press},
 title = {Singular Integrals and Differentiability Properties of Functions (PMS-30)},
 urldate = {2026-01-08},
 year = {1970}
}

@book{grafakos2014classical,
  title={Classical Fourier Analysis},
  author={Grafakos, L.},
  isbn={9781493911943},
  series={Graduate Texts in Mathematics},
  url={https://books.google.fr/books?id=94FxBQAAQBAJ},
  year={2014},
  publisher={Springer New York}
}

@article{cannone2010global,
  title={Global existence for a system of non-linear and non-local transport equations describing the dynamics of dislocation densities},
  author={Cannone, M. and El Hajj, A. and Monneau, R. and Ribaud, F.},
  journal={Archive for rational mechanics and analysis},
  volume={196},
  pages={71--96},
  year={2010},
  publisher={Springer}
}

@Book{111,
 Author = {Adams, R. A.},
 Title = {Sobolev spaces},
 FSeries = {Pure and Applied Mathematics (Academic Press)},
 Series = {Pure Appl. Math., Academic Press},
 ISSN = {0079-8169},
 Volume = {65},
 Year = {1975},
 Publisher = {Academic Press, New York, NY},
 Language = {English},
 zbMATH = {3491650},
 Zbl = {0314.46030}
}

@article{MoMo14,
author = {Monasse, L. and Monneau, R.},
title = {Gradient Entropy Estimate and Convergence of a Semi-Explicit Scheme for Diagonal Hyperbolic Systems},
journal = {SIAM Journal on Numerical Analysis},
volume = {52},
number = {6},
pages = {2792-2814},
year = {2014},
doi = {10.1137/130950458}
}

@article{alzohbi_elhajj_jazar_2022,
  author    = {Al Zohbi, M. and El Hajj, A. and Jazar, M.},
  title     = {Convergent semi‑explicit scheme to a non‑linear eikonal system},
  journal   = {BIT Numerical Mathematics},
  volume    = {62},
  pages     = {1841--1872},
  year      = {2022},
  publisher = {Springer},
}

@Book{19,
 Author = {Anderson, P. M. and Hirth, J. P. and Lothe, J.},
 Title = {Theory of dislocations},
 Edition = {3rd edition},
 ISBN = {978-0-521-86436-7},
 Year = {2017},
 Publisher = {Cambridge: Cambridge University Press},
 Language = {English},
 zbMATH = {6691972},
 Zbl = {1365.82001}
}

@article{PhysRevB.56.5807,
  title = {Link between the microscopic and mesoscopic length-scale description of the collective behavior of dislocations},
  author = {Groma, I.},
  journal = {Phys. Rev. B},
  volume = {56},
  issue = {10},
  pages = {5807--5813},
  numpages = {0},
  year = {1997},
  month = {Sep},
  publisher = {American Physical Society},
  doi = {10.1103/PhysRevB.56.5807},
  url = {https://link.aps.org/doi/10.1103/PhysRevB.56.5807}
}

@article{GROMA19993647,
title = {Investigation of dislocation pattern formation in a two-dimensional self-consistent field approximation},
journal = {Acta Materialia},
volume = {47},
number = {13},
pages = {3647-3654},
year = {1999},
issn = {1359-6454},
doi = {https://doi.org/10.1016/S1359-6454(99)00215-3},
url = {https://www.sciencedirect.com/science/article/pii/S1359645499002153},
author = {I. Groma and P. Balogh}
}

@article {MR2373180,
    AUTHOR = {El Hajj, A. and Forcadel, N.},
     TITLE = {A convergent scheme for a non-local coupled system modelling
              dislocations densities dynamics},
   JOURNAL = {Math. Comp.},
  FJOURNAL = {Mathematics of Computation},
    VOLUME = {77},
      YEAR = {2008},
    NUMBER = {262},
     PAGES = {789--812},
      ISSN = {0025-5718,1088-6842},
   MRCLASS = {65M06 (35Q72 49L25 74H15 74H20)},
  MRNUMBER = {2373180},
MRREVIEWER = {Riccardo\ Fazio},
       DOI = {10.1090/S0025-5718-07-02038-8},
       URL = {https://doi.org/10.1090/S0025-5718-07-02038-8},
}

@article {MR4734513,
    AUTHOR = {El Hajj, A. and Oussaily, A.},
     TITLE = {Convergent scheme for a non-local transport system modeling
              dislocations dynamics},
   JOURNAL = {J. Comput. Appl. Math.},
  FJOURNAL = {Journal of Computational and Applied Mathematics},
    VOLUME = {448},
      YEAR = {2024},
     PAGES = {Paper No. 115929, 18},
      ISSN = {0377-0427,1879-1778},
   MRCLASS = {33D15 (11F27 11F37)},
  MRNUMBER = {4734513},
MRREVIEWER = {Giuseppe\ Viglialoro},
       DOI = {10.1016/j.cam.2024.115929},
       URL = {https://doi.org/10.1016/j.cam.2024.115929},
}

@article {MR4221324,
    AUTHOR = {Boudjerada, R. and El Hajj, A. and Oussaily, A.},
     TITLE = {Convergence of an implicit scheme for diagonal
              non-conservative hyperbolic systems},
   JOURNAL = {ESAIM Math. Model. Numer. Anal.},
  FJOURNAL = {ESAIM. Mathematical Modelling and Numerical Analysis},
    VOLUME = {55},
      YEAR = {2021},
     PAGES = {S573--S591},
      ISSN = {2822-7840,2804-7214},
   MRCLASS = {65M06 (35A01 35F20 35F21 35Q74 65M12 70H20 74G22)},
  MRNUMBER = {4221324},
MRREVIEWER = {Jiebao\ Sun},
       DOI = {10.1051/m2an/2020049},
       URL = {https://doi.org/10.1051/m2an/2020049},
}

@article {MR2349873,
    AUTHOR = {El Hajj, A.},
     TITLE = {Well-posedness theory for a nonconservative {B}urgers-type
              system arising in dislocation dynamics},
   JOURNAL = {SIAM J. Math. Anal.},
  FJOURNAL = {SIAM Journal on Mathematical Analysis},
    VOLUME = {39},
      YEAR = {2007},
    NUMBER = {3},
     PAGES = {965--986},
      ISSN = {0036-1410,1095-7154},
   MRCLASS = {35L60 (35B30 74E15 74H20 74H25)},
  MRNUMBER = {2349873},
       DOI = {10.1137/060672170},
       URL = {https://doi.org/10.1137/060672170},
}

@book{hull2011introduction,
  title={Introduction to dislocations},
  author={Hull, D. and Bacon, D. J.},
  volume={37},
  year={2011},
  publisher={Elsevier}
}

@article{ZHIZ25,
  title={Convergence of a semi-explicit scheme for a one dimensional periodic nonlocal eikonal equation modeling dislocation dynamics},
  author={Al Zareef, D. and El Hajj, A. and Ibrahim, H.  and Zurek, A.},
  fjournal = {European Series in Applied and Industrial Mathematics (ESAIM): Mathematical Modelling and Numerical Analysis},
  journal = {ESAIM, Math. Model. Numer. Anal.},
  issn = {0764-583X},
  volume = {59},
  number = {6},
  pages = {2957--2990},
  year = {2025},
  language = {English},
  doi = {10.1051/m2an/2025080},
  zbMATH = {8129020}
}
\end{document}